\documentclass[sn-mathphys,Numbered]{sn-jnl}

\usepackage{graphicx}%
\usepackage{multirow}%
\usepackage{amsmath,amssymb,amsfonts}%
\usepackage{amsthm}%
\usepackage{mathrsfs}%
\usepackage[title]{appendix}%
\usepackage{xcolor}%
\usepackage{textcomp}%
\usepackage{manyfoot}%
\usepackage{booktabs}%
\usepackage{algorithm}%
\usepackage{algorithmicx}%
\usepackage{algpseudocode}%
\usepackage{listings}%
\usepackage{bm}%
\usepackage{tikz-cd}%
\usepackage{subfig}

\def\Xint#1{\mathchoice
	{\XXint\displaystyle\textstyle{#1}}%
	{\XXint\textstyle\scriptstyle{#1}}%
	{\XXint\scriptstyle\scriptscriptstyle{#1}}%
	{\XXint\scriptscriptstyle\scriptscriptstyle{#1}}%
	\!\int}
\def\XXint#1#2#3{{\setbox0=\hbox{$#1{#2#3}{\int}$ }
		\vcenter{\hbox{$#2#3$ }}\kern-.6\wd0}}

\def\dashint{\Xint-}
\theoremstyle{thmstyleone}%
\newtheorem{theorem}{Theorem}
\theoremstyle{thmstyletwo}%
\newtheorem{remark}{Remark}%
\newtheorem{lemma}{Lemma}%
\newtheorem{corollary}{Corollary}%
\newtheorem{strategy}{Strategy}%

\newtheorem{problem}{Problem}
\newtheorem{assumption}{Assumption}
\theoremstyle{thmstylethree}%
\newtheorem{definition}{Definition}%
\numberwithin{equation}{section}
\usepackage{chngcntr} 
\numberwithin{theorem}{section} 
\numberwithin{figure}{section} 
\numberwithin{table}{section} 
\numberwithin{definition}{section} 
\numberwithin{problem}{section} 
\numberwithin{lemma}{section}
\numberwithin{corollary}{section}
\numberwithin{remark}{section}
\numberwithin{assumption}{section}
\numberwithin{strategy}{section}
\begin{document}

\title[Exponentially-fitted FE for $\bm{H}({\rm curl})$ and $\bm{H}({\rm div})$]{Exponentially-fitted Finite Elements for $\bm{H}({\rm curl})$ and $\bm{H}({\rm div})$ Convection-Diffusion Problems }


\author[1]{\fnm{Jindong} \sur{Wang}}\email{jdwang@pku.edu.cn}

\author*[1]{\fnm{Shuonan} \sur{Wu}}\email{snwu@math.pku.edu.cn}

\affil[1]{\orgdiv{School of Mathematical Sciences}, \orgname{Peking University}, \orgaddress{\city{Beijing}, \postcode{100871}, \country{P. R. China}}}


\abstract{\unboldmath This paper presents a novel approach to the construction of the lowest order $\bm{H}(\mathrm{curl})$ and $\bm{H}(\mathrm{div})$ exponentially-fitted finite element spaces {${\mathcal{S}_{1^-}^{k}}~(k=1,2)$} on 3D simplicial mesh for corresponding convection-diffusion problems. It is noteworthy that this method not only facilitates the construction of the functions themselves but also provides corresponding discrete fluxes simultaneously. Utilizing this approach, we successfully establish a discrete convection-diffusion complex and employ a specialized weighted interpolation to establish a bridge between the continuous complex and the discrete complex, resulting in a coherent framework. Furthermore, we demonstrate the commutativity of the framework when the convection field is locally constant, along with the exactness of the discrete convection-diffusion complex. Consequently, these types of spaces can be directly employed to devise the corresponding discrete scheme through a Petrov-Galerkin method. 
}




\maketitle

\section{Introduction} \label{sec:intro}
Given a bounded region $\Omega$ in $\mathbb{R}^3$, this paper explores the utilization of finite element methods to address convection-diffusion problems in $\bm{H}(\text{curl})$ and $\bm{H}(\text{div})$. To simplify the analysis, we concentrate on model problems with homogeneous boundary conditions, which are expressed in the following forms.
\begin{itemize}
\item $\bm{H}({\rm curl})$ convection-diffusion problem:
\begin{equation} \label{eq:curl-cd}
 \left\{
\begin{array}{ll}
 \mathcal{L}^1 \bm{u} + \gamma \bm{u} := \nabla \times(\varepsilon \nabla \times \bm{u} + \bm{\beta} \times \bm{u}) + \gamma \bm{u} = \bm{f}
&\quad \text{in }
\Omega,\\ 
\bm{n} \times \bm{u} = \bm{0} & \quad \text{on } \partial \Omega. \\ 
\end{array}
\right.
\end{equation} 
\item $\bm{H}({\rm div})$ convection-diffusion problem:
\begin{equation} \label{eq:div-cd}
\left\{
\begin{array}{ll}
\mathcal{L}^2\bm{u} + \gamma \bm{u} := -\nabla (\varepsilon \nabla \cdot \bm{u} + \bm{\beta} \cdot \bm{u}) + \gamma \bm{u} = \bm{f}
&\quad \text{in }
\Omega,\\ 
\bm{u} \cdot \bm{n} = 0 & \quad \text{on }\partial \Omega.
\end{array}
\right.
\end{equation} 
\end{itemize}
Here, the velocity field is denoted by $\bm{\beta}$,  while the constant $0 < \varepsilon \leq 1$ denotes the diffusion coefficient and $\gamma \geq \gamma_0 > 0$ stands for the reaction coefficient. Additionally, $\bm{n}$ denotes the unit vector normal to the boundary $\partial \Omega$. Notably, convection and diffusion often interact, with the strength of this interaction being influenced by the ratio between the diffusion coefficient $\varepsilon$ and the magnitude of the velocity field $\bm{\beta}$. This interaction, referred to the magnetic Reynolds number in the context of magnetohydrodynamics, gives rise to various models that find crucial applications, especially within the field of magnetohydrodynamics.

The terminology of convection-diffusion is widely recognized and extensively studied for scalar problems, as they play an essential role in mathematical modeling and simulation across various fields, including fluid mechanics, astrophysics, groundwater flow, meteorology, semiconductors, and reactive flows. The scalar convection-diffusion in the conservative form can be written as 
\begin{equation} \label{eq:grad-cd}
\left\{
\begin{array}{ll}
\mathcal{L}^0u + \gamma u := -\nabla\cdot(\varepsilon\nabla u+\bm{\beta}u)+\gamma u = f
&\quad \text{in }
\Omega,\\ 
u = 0 & \quad \text{on }\partial \Omega.
\end{array}
\right.
\end{equation} 

The primary numerical challenge in dealing with convection-diffusion problems \eqref{eq:curl-cd} -- \eqref{eq:grad-cd} lies in the wide range of diffusion coefficients. Specifically, as the diffusion coefficient approaches zero, the occurrence of boundary layers can lead to spurious oscillations in standard finite element discretizations. To address this issue on quasi-uniform meshes, the design of stabilized finite element methods for convection-diffusion problems primarily revolves around two distinct approaches.  

The first type is the {\it upwind methods}, which introduce stabilization terms based on the information of the convection in the variational formulation. 
Prominent techniques within this group comprise residual-based methods such as Streamline-Upwind Petrov Galerkin (SUPG) \cite{brooks1982streamline, franca1992stabilized,hughes1979multidimentional} and residual-free bubble methods \cite{brezzi1994choosing, brezzi1998applications, brezzi1998further,brezzi1999priori,brezzi2000residual,brezzi2005choice}. Symmetric stabilization strategies, exemplified by local projection stabilization \cite{ganesan2010stabilization,matthies2007unified} and continuous interior penalty/edge stabilization \cite{burman2004edge, burman2005unified,burman2006continuous}, also fall under this umbrella, alongside the well-regarded discontinuous Galerkin (DG) methods \cite{brezzi2004discontinuous,cockburn1998local,houston2002discontinuous,ayuso2009discontinuous,egger2010hybrid}.
Notably, the efficacy of upwind methodologies has extended to address more intricate challenges posed by vector convection-diffusion problems. Specifically, for the magnetic convection problem, which involves $\bm{H}({\rm curl})$ convection-diffusion, Heumann and Hiptmair delved into upwind DG methods \cite{heumann2013stabilized}. Expanding on this foundation, other notable contributions include the presentation of discontinuous finite element methods and hybridized discontinuous finite element methods for magnetic convection-diffusion problems in \cite{wang2022discontinuous, wang2023hybridizable}. These advancements inherently belong to the upwind paradigm, showcasing the stabilization mechanisms tailored to $\bm{H}({\rm curl})$ formulation exclusively.

Another class of methods is based on the idea of {\it exponential fitting}, considering that the boundary layer can often be characterized by an exponential function. Therefore, the {\it operator-fitting approach} can be employed, incorporating certain exponential functions in the variational formulation or the construction of the stiffness matrix. For the scalar convection-diffusion problem \eqref{eq:grad-cd}, Brezzi, Marini and Pietra \cite{brezzi1989two} proposed an exponential fitting method based on the exponential transformation. Xu and Zikatanov introduced the Edge-Averaged Finite Element (EAFE) method \cite{xu1999monotone}, based on the concept of flux-patch constant approximation. This method ensures monotonicity under weaker mesh conditions (i.e., the discrete matrix is an $M$-matrix). The extension of the EAFE scheme to various anisotropic diffusion coefficients can be found in \cite{lazarov2012exponential}, while its extension to space-time discretization is discussed in \cite{bank2017arbitrary}. We refer to \cite{bank1990some, angermann1995error, bank1998finite} for other operator-fitting schemes.

In contrast to the operator-fitting approach, a more direct strategy involves crafting a specialized discrete finite element space that incorporates {\it specific exponential functions}. This space is essentially a collection of functions that approximate the convection-diffusion differential operator approaching zero, known as $\mathcal{L}$-spline. To explore the construction and analysis of exponential fitting spaces in one-dimensional settings, we refer to \cite{stynes1986finite, stynes1991analysis}. When dealing with higher-dimensional structured meshes, the pioneering work by \cite{o1991globally} outlines the construction and analysis of finite element spaces embracing an exponential feature. Subsequently, \cite{roos1996novel} introduced a non-matching $\mathcal{L}$-spline finite element space tailored for general convection velocity. Expanding the analytical scope, \cite{dorfler1999uniform,dorfler1999uniform2} established a more comprehensive framework.
Compared to structured meshes, constructing $\mathcal{L}$-spline on unstructured simplical meshes is relatively challenging. 
\cite{sacco1998finite, sacco1999nonconforming} provided special finite element spaces for two-dimensional triangular unstructured meshes. This method, based on the assumption of linear flux, uniquely determines the corresponding discrete space. Further, Wang \cite{wang1997novel} proposed the idea of pointwise low-dimensional confinement approximation, confining the problem to one dimension. This resulted in an algebraic system that uniquely defines the basis functions, leading to the construction of an exponentially-fitted finite element space on simplical meshes \cite{wang2002analysis, angermann2005multidimensional}, which exhibits good stabilization effects. Note that the construction of $\mathcal{L}$-spline mentioned above pertains exclusively to scalar convection-diffusion problem \eqref{eq:grad-cd}.

Different types of convection-diffusion problems share a unified mathematical form, as evident in equations \eqref{eq:grad-cd}, \eqref{eq:curl-cd}, and \eqref{eq:div-cd}, representing the proxy in $\mathbb{R}^3$ of convection-diffusion problems in terms of differential $k$-forms ($k=0,1,2$). In fact, the study of Finite Element Exterior Calculus (FEEC) has been a prominent research focus in recent years \cite{arnold2006finite}, providing a unified mathematical framework and discretization approach for diffusion problems. While investigations into the discretization of convection terms under the FEEC framework are still limited, notable contributions have been made by scholars such as Hiptmair and his colleagues. They explored a range of time-stepping methods, including Eulerian and semi-Lagrangian techniques, to address generalized convection-diffusion problems involving convection terms based on the Lie derivative \cite{heumann2011eulerian, heumann2012fully, heumann2013convergence}. Inspired by the EAFE method \cite{xu1999monotone}, a distinct approach emerged in the form of the simplex-averaged finite element (SAFE) method for both scalar and vector scenarios of advection-diffusion problems \cite{wu2020simplex}. The SAFE method falls within the category of exponential operator-fitting techniques, employing conventional FEM spaces and demonstrating an inherent upwind effect as the diffusion coefficient approaches zero.

\subsection{Main contribution}
Our aim is to provide a unified construction of exponentially-fitted finite element spaces (i.e., $\mathcal{L}^k$-spline for $k=0,1,2$) for the convection-diffusion problem of $H(d)$ on three-dimensional simplicial meshes. In this context, the operator $d$ can be interpreted as the grad ($d^0=\nabla$), curl ($d^1=\nabla \times$), or div ($d^2 = \nabla \cdot$) operator, in terms of proxy in $\mathbb{R}^3$. Introducing $\bm{\theta} = \bm{\beta}/\varepsilon$, we consider the fluxes associated with the convection-diffusion problems \eqref{eq:grad-cd}, \eqref{eq:curl-cd}, and \eqref{eq:div-cd} as follows:
\begin{equation} \label{eq:flux}
\begin{aligned}
\bm{J}_{\bm{\theta}}^0 u &:= \varepsilon \nabla u + \bm{\beta} u = \varepsilon (\nabla u + \bm{\theta} \cdot \nabla u), \\
\bm{J}_{\bm{\theta}}^1 \bm{u} &:= \varepsilon \nabla \times \bm{u} + \bm{\beta} \times \bm{u} 
= \varepsilon (\nabla \times \bm{u} + \bm{\theta} \times \bm{u}), \\
J_{\bm{\theta}}^2\bm{u} &:= \varepsilon \nabla \cdot \bm{u} + \bm{\beta} \cdot \bm{u} 
= \varepsilon (\nabla \cdot \bm{u} + \bm{\theta} \cdot \bm{u} ). \\
\end{aligned}
\end{equation}
It is worth noting that while examining a specific index $k$ in the fluxes outlined above, vector functions are represented using bold typeface. However, for general abstract indices, bold formatting is not employed. The starting point of this article is the exponentially-fitted identity originally proposed in SAFE \cite{wu2020simplex}: When $\bm{\theta}$ is a constant field,
\begin{equation} \label{eq:identity}
J_{\bm{\theta}}^k u=\varepsilon E_{\bm{\theta}}^{-1} d^k E_{\bm{\theta}} u, \quad k=0,1,2,
\end{equation}
where $E_{\bm{\theta}}(\bm{x}) :=\exp(\bm{\theta}\cdot\bm{x})$. Let $\mathbb{E}_{\bm{\theta}}(\Omega) := \{c\exp(\bm{\theta}\cdot\bm{x}): c\in \mathbb{R}, \bm{x} \in \Omega\}$. On the continuous level, identity \eqref{eq:identity} corresponds to the following exact sequence:
\begin{equation} \label{eq:sequence-cont}
\mathbb{E}_{-\bm{\theta}}(\Omega) \hookrightarrow C^\infty(\Omega) \xrightarrow{\bm{J}_{\bm{\theta}}^0} C^\infty(\Omega;\mathbb{R}^3) \xrightarrow{\bm{J}_{\bm{\theta}}^1} C^\infty(\Omega;\mathbb{R}^3) \xrightarrow{J_{\bm{\theta}}^2} C^\infty(\Omega) \rightarrow 0.
\end{equation}

Inspired by Wang's work on scalar problem \cite{wang1997novel}, our construction relies on two key strategies: (i) By confining the convection-diffusion problem of $H(d^k)$ to the interior $(k+1)$-dimensional sub-simplices, and based on the intrinsic identity \eqref{eq:identity}, obtaining and approximating the equations satisfied by $\mathcal{L}^k$-spline; (ii) For the construction of vector exponentially-fitted spaces ($k=1,2$), it also depends on a strategy involving a constant approximation of the $k$-dimensional sub-simplex exponential integral.

Based on the positivity and geometric properties of Bernoulli functions associated with the discrete spaces, we subsequently establish the well-posedness of the construction and prove properties of the resulting exponentially-fitted finite element spaces, including consistency, and unisolvent degrees of freedom (DOFs). This construction provides the lowest-order approximation of the $\mathcal{L}$-spline space, with DOFs consistent with those of classical finite element spaces (which are denoted by $\mathcal{P}_1^-\Lambda^k$ in the FEEC \cite{arnold2006finite}), hence denoted as $\mathcal{S}_{1^-}^k$. It is worth noting that the construction of the space is carried out in a pointwise sense, thus ensuring strict conformity of the discrete space when velocity fields satisfying certain smoothness. Moreover, for a given set of DOFs associated with a triangulation $\mathcal{T}_h$, this construction {\it simultaneously provides the function $u_h$ and the discrete flux $J_{\bm{\theta},h}^ku_h$}, satisfying the following when $\bm{\theta}$ is a constant field:
\begin{equation} \label{eq:sequence-dis}
	\mathbb{E}_{-\bm{\theta}}(\Omega) \hookrightarrow \mathcal{S}_{1^-}^0(\mathcal{T}_h) \xrightarrow{\bm{J}_{\bm{\theta},h}^0} \mathcal{S}_{1^-}^1(\mathcal{T}_h) \xrightarrow{\bm{J}_{\bm{\theta},h}^1} \mathcal{S}_{1^-}^2(\mathcal{T}_h) \xrightarrow{J_{\bm{\theta},h}^2} \mathcal{S}_{1^-}^3(\mathcal{T}_h) \rightarrow 0,
\end{equation}
which serves as a discrete analogue of \eqref{eq:sequence-cont}.

We wish to emphasize that this construction represents a natural finite element method for convection-diffusion problems. Taking the $\bm{H}(\mathrm{curl})$ problem \eqref{eq:curl-cd} as an example, its variational form reads as follows: Find $\bm{u} \in \bm{H}_0(\mathrm{curl}; \Omega)$ such that 
$$ 
(\underbrace{\varepsilon\nabla \times \bm{u} + \bm{\beta} \times \bm{u}}_{\bm{J}_{\bm{\theta}}^1 \bm{u}}, \underbrace{\nabla \times \bm{v}}_{d^1 \bm{v}} ) + (\gamma \bm{u}, \bm{v}) = (\bm{f}, \bm{v}) \quad \forall \bm{v} \in \bm{H}_0(\mathrm{curl}; \Omega).
$$ 
When discretizing the space for $\bm{u}$, what we truly require is a mapping, at every point (or every quadrature point), from the degrees of freedom to the value and flux, for respectively discretizing the lower-order and convection-diffusion terms. Remarkably, our construction precisely fulfills this fundamental requirement. Similarly, for the finite element space containing the test function $\bm{v}_h$, this mapping corresponds to the transformation of degrees of freedom into the values of $\bm{v}_h$ and $\nabla \times \bm{v}_h$. It is evident that, in the context of traditional polynomial finite element discretization, the value of $\nabla \times \bm{v}_h$ within this mapping is nothing but the curl of the function $\bm{v}_h$. In essence, the natural discrete space for $\bm{v}_h$ corresponds to the lowest-order N{\'e}d{\'e}lec elements. From this perspective, a natural discretization scheme takes the form of the Petrov-Galerkin: the discretization of $\bm{u}$ employs the exponentially-fitted space, while for $\bm{v}$, the traditional finite element space is utilized.

On the other hand, when the convection field degenerates to zero, the starting point of our construction (i.e., the two strategies) is naturally consistent with the standard edge element, rendering the finite element spaces compatible with traditional polynomial spaces. From this perspective, this construction serves as a natural extension for problems involving convection terms, and when convection dominates, this space encompasses certain exponential functions.

Finally, we introduce a specialized class of weighted interpolation operators, utilizing them as a bridge to connect the continuous complex with the discrete complex. We establish the commutativity of the corresponding diagrams when $\bm{\theta}$ is a locally constant field. These findings play a pivotal role in the analysis of the numerical scheme.

The rest  of the paper is structured as follows. In Section \ref{sec:pre}, we present preliminary results, including geometric notation, the Bernoulli functions, and a concise review of the $H(\mathrm{grad})$ exponentially-fitted finite element (FE) space. The construction and pertinent properties of the $\bm{H}(\mathrm{curl})$ exponentially-fitted FE space are discussed in Section \ref{sec:curl}. Similarly, Section \ref{sec:div} delves into the construction and associated properties of the $\bm{H}(\mathrm{div})$ exponentially-fitted FE space. In Section \ref{sec:diagram}, we introduce crucial operators for establishing a commutative diagram under the condition of a locally constant vector field $\bm{\beta}$. The utilization of the FE spaces in convection-diffusion problems is demonstrated in Section \ref{sec:application}. To empirically showcase the accuracy and inherent stability benefits of the proposed exponentially-fitted FE spaces, we provide a series of numerical experiments in Section \ref{sec:numerical}.

\section{Preliminary results} \label{sec:pre}
In this section, we introduce pertinent geometric notation and delve into the Bernoulli functions, accompanied by their associated properties. Furthermore, we will provide an overview of Wang's methodology \cite{wang1997novel} for the construction of the exponentially-fitted $H(\mathrm{grad})$ finite element space.

Given $p\in[1,\infty]$ and an integer $m\ge 0$, we use the usual notation $W^{m,p}(\Omega)$, $\|\cdot\|_{m,p,\Omega}$ and $|\cdot|_{m,p,\Omega}$ to denote the usual Sobolev space, norm and semi-norm, respectively. When $p=2$, $H_m(\Omega):=W^{m,p}(\Omega)$ with $|\cdot|_{m,\Omega}:=|\cdot|_{m,2,\Omega}$ and $\|\cdot\|_{m,\Omega}:=\|\cdot\|_{m,2,\Omega}$. 

Let $\mathcal{T}_h$ be a conforming and shape-regular sequence of decompositions of $\Omega$ into tetrahedrons, $h_T$ is the diameter of $T \in \mathcal{T}_h$ and $h:=\max_{T\in\mathcal{T}_h}h_T$. We also define $\mathcal{V}_h$ (resp. $\mathcal{V}_T$) as the set of all vertices belonging to $\mathcal{T}_h$ (resp. $T \in \mathcal{T}_h$). Similarly, $\mathcal{E}_h$ (resp. $\mathcal{E}_T$) represents the collection of all edges within $\mathcal{T}_h$ (resp. $T \in \mathcal{T}_h$), and $\mathcal{F}_h$ (resp. $\mathcal{F}_T$) denotes the set of facets within $\mathcal{T}_h$ (resp. $T \in \mathcal{T}_h$).

Throughout this paper, the notation $A\lesssim B$ is employed to signify that $A$ is bounded above by $CB$, where $C$ represents a constant, independent of mesh size.

\subsection{Geometric notation on tetrahedron}\label{subsec:notation}
Consider a tetrahedron $T \in \mathcal{T}_h$ with vertices $\bm{x}_1, \bm{x}_2, \bm{x}_3$, and $\bm{x}_4$. The edges $E_{ij}$ $(1 \leq i < j \leq 4)$ are formed by connecting the vertices $\bm{x}_i$ and $\bm{x}_j$, while the facets $F_i~(1 \leq i \leq 4)$ are positioned opposite the vertex $\bm{x}_i$. The unit tangential vector of edge $E_{ij}$ is defined as $\bm{t}_{ij} := \frac{\bm{E}_{ij}}{|\bm{E}_{ij}|}$ with $\bm{E}_{ij} := \bm{x}_j-\bm{x}_i$. Furthermore, the unit outward normal vector of facet $F_i$ is denoted by $\bm{n}_i$.

For any point $\bm{x}\in \bar{T}$, we denote by ${l}_i ~(1\leq i \leq 4)$ the sub-edge connecting $\bm{x}$ to $\bm{x}_i$, whose direction is determined by $\bm{\tau}_i: = \frac{\bm{l}_i}{|\bm{l}_i|}$ with $\bm{l}_i := \bm{x}_i - \bm{x}$ (Figure \ref{fig:sub-edge}). We also denote by $f_{ij}~(1\leq i < j \leq 4)$ the sub-facet formed by edge $E_{ij}$ and point $\bm{x}$, and by $\bm{\nu}_{ij}=\frac{\bm{\tau}_i\times \bm{\tau}_j}{|\bm{\tau}_i \times \bm{\tau}_j|}$ the unit normal vector of $f_{ij}$ (Figure \ref{fig:sub-facet}). Furthermore, we use $T_i$ to denote the sub-tetrahedron composed of point $\bm{x}$ and facet $F_i$ (Figure \ref{fig:sub-tetrahedron}). 

\begin{figure}[!htbp]
	\centering
	\subfloat[sub-edges]{
	\includegraphics[width=.3\textwidth]{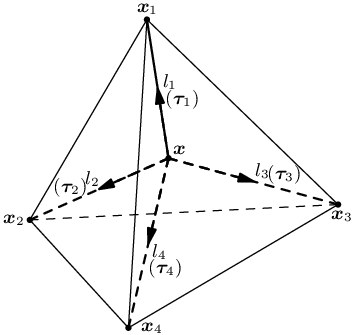} \label{fig:sub-edge}
	}
	\subfloat[sub-facets]{
	\includegraphics[width=.3\textwidth]{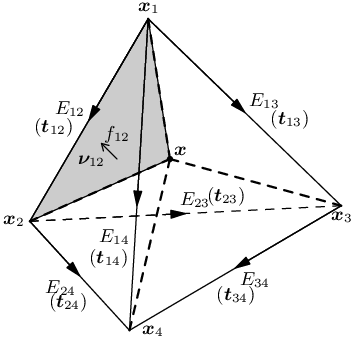} \label{fig:sub-facet}
	}
	\subfloat[sub-tetrahedrons]{
	\includegraphics[width=.3\textwidth]{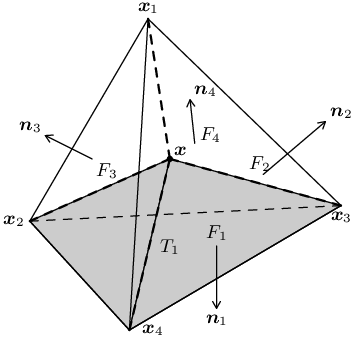} \label{fig:sub-tetrahedron}
	}
	\caption{Geometric notation associated with the tetrahedron $T$.}
	\label{fig:geo-notation}
\end{figure}

We observe that $\bm{l}_i~(1\leq i \leq 4)$ constitutes a linear vector function of $\bm{x}$. It maintains a direct yet frequently-employed correlation with the barycentric coordinates $\lambda_i(\bm{x})$, as demonstrated below:
\begin{equation} \label{eq:identityli}
	\sum_{i=1}^{4} \lambda_i\bm{l}_i = \sum_{i=1}^{4}\lambda_i\left(\bm{x}_i-\bm{x}\right)=\sum_{i=1}^{4}\lambda_i\bm{x}_i - \sum_{i=1}^4 \lambda_i\bm{x}=\bm{x}-\bm{x}=\bm{0}.
\end{equation}

\subsection{Bernoulli functions}
Considering the convection field $\bm{\beta}(\bm{x})$, we introduce its vector product with $\bm{l}_i ~(1\leq i\leq 4)$ and $\bm{E}_{ij}~(1\leq i < j \leq 4)$. These products define the domain for the subsequent definition of Bernoulli functions,
	\begin{equation}\label{eq:sigma}
		\sigma_i(\bm{x}) := {\bm{\beta}}(\bm{x})\cdot \bm{l}_i(\bm{x}), \quad \sigma_{ij}(\bm{x}):= \bm{\beta}(\bm{x}) \cdot \bm{E}_{ij}.
	\end{equation}
A direct calculation shows that
\begin{equation} \label{eq:sigma-diff}
\sigma_{ij}(\bm{x}) = \bm{\beta}(\bm{x})\cdot (\bm{x}_j-\bm{x}_i) = \bm{\beta}(\bm{x})\cdot(\bm{l}_j-\bm{l}_i) = \sigma_j(\bm{x}) - \sigma_i(\bm{x}).
\end{equation}

Secondly, the construction of spline basis is based on the geometric information of $T$ and the following coefficients:
$$ 
\mbox{diffusion coefficient} \times
\frac{\mbox{exponential average on sub-simplex of dimension }k}
{\mbox{exponential average on sub-simplex of dimension }k+1}.
$$ 
Thanks to the affine mapping to reference element, this coefficient can be given by the following Bernoulli functions:
\begin{definition}
	[1D-Bernoulli function, $k=0$]$B_1^\varepsilon:\mathbb{R}\rightarrow \mathbb{R}^+$ denotes the 1D-Bernoulli function defined by
	\begin{equation}\label{eq:bernoulli1d}
		B_1^\varepsilon(\sigma_1) := \varepsilon\frac{1}{\int_0^1 \exp(\sigma_1 x_1/\varepsilon) \mathrm{d}x_1}.
	\end{equation}
\end{definition}

\begin{definition}
	[2D-Bernoulli function, $k=1$]$B_2^\varepsilon:\mathbb{R}^2 \rightarrow \mathbb{R}^+$ denotes the 2D-Bernoulli function defined by
	\begin{equation}\label{eq:bernoulli2d}
B_2^\varepsilon(\sigma_1,\sigma_2) := \varepsilon \frac{\int_0^1 \exp(\sigma_1
	{x}_1/\varepsilon)\,\mathrm{d}{x}_1}
{2 \int_0^1 \int_0^{1-{x}_2}
	\exp((\sigma_1{x}_1 + \sigma_2{x}_2)/\varepsilon)\,\mathrm{d}{x}_1 
	\mathrm{d}{x}_2}.
	\end{equation}
\end{definition}

\begin{definition}
	[3D-Bernoulli function, $k=2$]$B_3^\varepsilon:\mathbb{R}^3 \rightarrow \mathbb{R}^+$ denotes the 3D-Bernoulli function defined by
	\begin{equation}\label{eq:bernoulli3d}
B_3^\varepsilon(\sigma_1,\sigma_2,\sigma_3) := \varepsilon \frac{2 \int_0^1 \int_0^{1-{x}_2}
	\exp((\sigma_1{x}_1 +
	\sigma_2{x}_2)/\varepsilon)\,\mathrm{d}{x}_1\mathrm{d}{x}_2}
{6\int_0^1\int_0^{1-{x}_3} \int_0^{1-{x}_2 - {x}_3} 
	\exp((\sigma_1{x}_1 + \sigma_2{x}_2 + \sigma_3{x}_3)/\varepsilon)\,\mathrm{d}{x}_1
	\mathrm{d}{x}_2\mathrm{d}{x}_3}.
	\end{equation}
\end{definition}
The aforementioned Bernoulli functions $B_k^\varepsilon$ are well-defined and continuous on $\mathbb{R}^k$. It is important to note that these Bernoulli functions remain viable as $\varepsilon\rightarrow 0^+$. For further insight into the limiting behavior, we refer to \cite[(A.5)-(A.7)]{wu2020simplex}.

\subsection{Exponentially-fitted finite element space in $H(\mathrm{grad})$}
We begin by revisiting the construction of the exponentially-fitted FE space, as introduced by Wang \cite{wang1997novel}. This construction is specifically tailored to address the $H(\mathrm{grad})$ convection-diffusion problem \eqref{eq:grad-cd}. The subsequent derivation relies on the assumption that $\bm{\theta}$ is a constant, allowing us to utilize \eqref{eq:identity}:
$$
\bm{J}_{\bm{\theta}}^0u=\varepsilon( \nabla  {{u}} + {\bm{\theta}}  {{u}} ) = \varepsilon E_{\bm{\theta}}^{-1}\nabla(E_{\bm{\theta}} {{u}}).
$$ 
By employing equation \eqref{eq:grad-cd}, we have $\mathcal{L}^0 u = - \nabla \cdot (\bm{J}_{\bm{\theta}}^0u)$. 

Guided by the $\mathcal{L}^0$-spline philosophy, which builds upon the fundamental principles outlined in \cite{wang1997novel}, we aim to {\it confine} the action of the operator $\mathcal{L}^0$ to the one-dimensional sub-edges $l_i~(1 \leq i \leq 4)$. This confinement entails that when the differential operator is applied, it exclusively accounts for the influence along the direction of $\bm{\tau}_i$. Therefore, we obtain the edge $H({\mathrm{grad}})$ convection diffusion operator:
\begin{equation}
	\mathcal{L}_E^0u = -\frac{\partial }{\partial \bm{\tau}}\left((\bm{J}_{\bm{\theta}}^0u\right)\cdot\bm{\tau})= -\frac{\partial }{\partial \bm{\tau}} \left(\varepsilon E_{\bm{\theta}}^{-1}\frac{\partial}{\partial\bm{\tau}}(E_{\bm{\theta}}u)\right),
\end{equation}
where $\bm{\tau}$ is the unit tangential of edge $E$. 

Given an element $T$, we intend to construct the shape function space as defined below:
\begin{equation} \label{eq:curl-shape}
	\mathcal{S}_{1^-}^0(T) := \mathrm{span}\{ \bm{\phi}^V_{i}:~ 1\leq i \leq 4\},
\end{equation}
where $\bm{\phi}_{i}^V$ is the basis function associated with the vertex $\bm{x}_{i}$. 

For a given $\bm{x} \in \bar{T}$, we adopt a geometric notation (as depicted in Figure \ref{fig:sub-edge}). To provide further clarity, we will derive algebraic equations involving the values of the basis function ${\phi}_1^V(\bm{x})$ and its corresponding flux $\bm{j}^V_1(\bm{x})$. A systematic approach involves considering their extensions onto one-dimensional sub-edges $l_i$ ($1\le i\le 4$), denoted as $\widetilde{\phi}_1^V(\bm{y})$ and $\widetilde{\bm{j}}^V_1(\bm{y})$ respectively, while satisfying the confinement on the sub-edges. As a result, when we enforce this confinement on $l_i~(1\leq i \leq 4)$, the application of the $\mathcal{L}^0$-spline leads to the following:
\begin{equation} \label{eq:confinement}
\left\{\begin{aligned}
	-\frac{\partial }{\partial \bm{\tau}_i}\left( \bm{\tau}_i \cdot \widetilde{\bm{j}}_1^V(\bm{y}) \right) &= 
	 -\frac{\partial}{\partial \bm{\tau}_i} \left( 
	 \varepsilon E^{-1}_{\bm{\theta}}(\bm{y})\frac{\partial}{\partial \bm{\tau}_i}
	 (E_{\bm{\theta}}(\bm{y}) \widetilde{\phi}_1^V(\bm{y})) \right)=0 \quad \text{ on } l_i,\\
	\widetilde{\phi}_1^V(\bm{x}_i) &= \delta_{1i},
\end{aligned}\right.
\end{equation}
for $1\leq i \leq 4$.
Whence the product $\bm{\tau}_i \cdot \widetilde{\bm{j}}_1^V(\bm{y})$ remains constant along $l_i$, implying that
\begin{equation} \label{eq:grad-const}
\bm{\tau}_i\cdot {\bm{j}}_1^V(\bm{x}) =\varepsilon E^{-1}_{\bm{\theta}}(\bm{y})\frac{\partial}{\partial \bm{\tau}_i}(E_{\bm{\theta}}(\bm{y})\widetilde{\phi}_1^V(\bm{y})), \quad \forall \bm{y} \in l_i.
\end{equation}
Multiplying the aforementioned equation by $E_{\bm{\theta}}(\bm{y})$ and integrating it over $l_i$, we have
$$
\bm{l}_i\cdot \bm{j}^V_1(\bm{x}) \dashint_{l_i}E_{\bm{\theta}} = \varepsilon\int_{l_i} \frac{\partial}{\partial \bm{\tau}_i}(E_{\bm{\theta}}\widetilde{\phi}_1^V)=\varepsilon E_{\bm{\theta}}(\bm{x}_i)\delta_{1i} -\varepsilon E_{\bm{\theta}}(\bm{x})\phi_1^V(\bm{x}) .
$$
Now, utilizing the geometric interpretation of the 1D-Bernoulli function as expressed in equation \eqref{eq:bernoulli1d}, we arrive at the algebraic equation:
\begin{equation}\label{eq:alggrad}
\bm{l}_i\cdot \bm{j}_1^V(\bm{x}) +  B_1^\varepsilon(\sigma_i)\phi_{1}^V(\bm{x}) =   B_1^\varepsilon(-\sigma_i)\delta_{1i}  \quad  1\leq i \leq 4,
\end{equation}
where $\sigma_i = \bm{\beta} \cdot \bm{l}_i$ is defined in equation \eqref{eq:sigma}. These equations can be organized into a linear system to uniquely determine the values of $\bm{j}_1^V(\bm{x})$ and $\phi_1^V(\bm{x})$.

 \begin{problem}[$\mathcal{L}^0$-spline] \label{problem:grad} 
 Find $\phi_1^V(\bm{x})$ and $\bm{j}_1^V(\bm{x})$ such that for all $\bm{x}\in \bar{T}$ 
\begin{equation}\label{eq:sysgrad}
D^V(\bm{x})\left(\begin{matrix}
	\bm{j}_1^V(\bm{x})\\\phi_1^V(\bm{x})
\end{matrix}\right)= B_1^\varepsilon(-\sigma_1)\bm{e}_1^V,
\end{equation}
where $\bm{e}_1^V=(1,0,0,0)^\top$ and $D^V(\bm{x})$ is a $4\times 4$ matrix defined by
\begin{equation}
	D^V(\bm{x})=\left(\begin{matrix}
		\bm{l}_1^\top &  B_1^\varepsilon(\sigma_1)\\
		\bm{l}_2^\top &  B_1^\varepsilon(\sigma_2)\\
		\bm{l}_3^\top &  B_1^\varepsilon(\sigma_3)\\
		\bm{l}_4^\top &  B_1^\varepsilon(\sigma_4)
	\end{matrix}\right).
\end{equation}
\end{problem}

The well-posedness of Problem \ref{problem:grad} has been established in \cite{wang1997novel}, where the determint of $D^V(\bm{x})$ is shown to be 
$$
\det(D^V(\bm{x})) = 6|T|\sum_{i=1}^{4}\lambda_i B_1^\varepsilon (\sigma_i),
$$ 
which is positive due to the positivity of 1D-Bernoulli function. The function space $\mathcal{S}_{1^-}^0(T)$ shares identical degrees of freedom (DOFs) with the standard $\mathcal{P}_1$-Lagrange element (it is denoted as $\mathcal{P}_{1^-}^0$ following the nomenclature used in the article). It is noteworthy that the construction assumes $\bm{\theta}$ (or $\bm{\beta}$) to be a constant vector field. However, the final formulation of Problem \ref{problem:grad} provides values for the basis functions and their corresponding fluxes solely at the point $\bm{x}$. Consequently, while solving for the point $\bm{x}$, it is only necessary to employ $\bm{\beta}(\bm{x})$ for evaluating the corresponding $\sigma_i$ to construct the matrix $D^V(\bm{x})$ and the right-hand side term. This guarantees $H({\mathrm{grad}})$ conformity as long as $\bm{\beta}$ exhibits tangential continuity across any facet $F\in\mathcal{F}_h^0$ and is piecewise $C^1$. This is a relaxation of the requirement in  \cite{wang1997novel} that $\bm{\beta}\in C^1(\Omega)$.

\section{Exponentially-fitted finite element spaces in $H({\rm curl})$}\label{sec:curl}

In this section, we will delve into the construction of exponentially-fitted finite element spaces in $\bm{H}(\mathrm{curl})$. Our construction also builds upon the assumption of a constant field $\bm{\theta}$, which grants us the privilege to employ \eqref{eq:identity},
$$
\bm{J}_{\bm{\theta}}^1 \bm{u} = \varepsilon(\nabla \times \bm{u} + \bm{\theta} \times \bm{u}) = \varepsilon E_{\bm{\theta}}^{-1} \nabla \times (E_{\bm{\theta}} \bm{u}).
$$ 
By employing equation \eqref{eq:curl-cd}, we can express $\mathcal{L}^1\bm{u} = \nabla \times (\bm{J}_{\bm{\theta}}^1 \bm{u})$. 

\subsection{$\mathcal{L}^1$ on two-dimensional facet}

Next, we consider the $\bm{H}({\rm curl})$ convection-diffusion operator $\mathcal{L}^1$ on the two-dimensional facet $F$ with unit normal vector $\bm{\nu}$. Let $w$ be a function on $F$, and $\bm{v}$ be a tangential vector field along $F$. Then surface/tangential rotational gradient and curl along $F$ are defined by
$$ 
\nabla_F^\perp w := (\nabla_F w) \times \bm{\nu}, \qquad \nabla_F \times \bm{v} := \nabla_F \cdot (\bm{v} \times \bm{\nu}),
$$ 
where $\nabla_F$ and $\nabla_F \cdot$ denote the surface/tangential gradient and divergence, respectively \cite{gilbarg1977elliptic,deckelnick2005computation}. By respectively replacing the two curl operators with $\nabla_F^\perp$ and $\nabla_F \times$, we derive the surface $\bm{H}({\rm curl})$ convection-diffusion operator:
\begin{equation} \label{eq:surface-curl}
\mathcal{L}_F^1 \bm{u} = \nabla_F^\perp\left( (\bm{J}_{\bm{\theta}}^1\bm{u}) \cdot \bm{\nu} \right) =  \nabla_F^\perp \left( \varepsilon E_{\bm{\theta}}^{-1} \nabla_F\times (E_{\bm{\theta}} \bm{u}) \right).
\end{equation}


\subsection{$\mathcal{L}^1$-spline shape function space}
Given an element $T$, we intend to construct the shape function space as defined below:
\begin{equation} \label{eq:curl-shape}
 \mathcal{S}_{1^-}^1(T) := \mathrm{span}\{ \bm{\phi}^E_{ij}:~ 1\leq i < j \leq 4\},
\end{equation}
where $\bm{\phi}_{ij}^E$ is the basis function associated with the edge $E_{ij}$, which we will define explicitly at a later point. To be more precise, we will simultaneously provide their values $\bm{\phi}_{ij}^E(\bm{x})$ and corresponding fluxes $\bm{j}_{ij}^E(\bm{x})$ at every $\bm{x} \in \bar{T}$.

For a fixed point $\bm{x} \in \bar{T}$, we adopt a geometric convention (illustrated in Figure \ref{fig:sub-facet}). In this specific geometric context, we proceed to appropriately extend $\bm{\phi}_{ij}^E(\bm{x})$ and $\bm{j}_{ij}^E(\bm{x})$ onto certain two-dimensional sub-facets $f_{st}~(1\leq s < t\leq 4)$. These extensions are represented by $\widetilde{\bm{\phi}}_{ij}^E(\bm{y})$ and $\widetilde{\bm{j}}_{ij}^E(\bm{y})$. To initiate this process, we introduce two distinct strategies, both of which form the essential components in the construction of the function space.

\begin{strategy}[tangential constant along sub-edges]\label{sg:curlconstant}
Given $\bm{x} \in \overline{T}$, the extension $\widetilde{\bm{\phi}}_{ij}^E(\bm{y})$ satisfies
$$
    \widetilde{\bm{\phi}}_{ij}^E(\bm{y})\cdot\bm{\tau}_s = \bm{\phi}_{ij}^E(\bm{x}) \cdot \bm{\tau}_s \quad \forall \bm{y} \in l_s, ~1\leq s \leq 4.
$$
\end{strategy}
\begin{strategy}[$\mathcal{L}^1$-spline on sub-facets]\label{sg:curlrestrict}
	Given $ \bm{x} \in \bar{T}$,  $\widetilde{\bm{\phi}}^E_{ij}(\bm{y})$ is an $\mathcal{L}_{f_{st}}^1$-spline, i.e., 
	$$
	\mathcal{L}^1_{f_{st}} \widetilde{\bm{\phi}}^E_{ij} = \nabla_{f_{st}}^\perp (\widetilde{\bm{j}}_{ij}^E \cdot \bm{\nu}_{st} )= \bm{0}, \quad \text{ on } f_{st}, ~1\leq s<t\leq 4.
	$$
\end{strategy}

\begin{remark}[consistency with the edge element when $\bm{\beta} = \bm{0}$]
When the convection field $\bm{\beta}$ is absent, then the exponential function $E_{\bm{\theta}}\equiv1$, and the operator $\mathcal{L}^1$ exclusively has the diffusion component. It is well recognized that the tangential aspect of the edge basis within edge element space (denoted by $\mathcal{P}_{1^-}^1(T)$ in accordance with the naming convention used in this paper) remains constant along any fixed line, precisely aligning with Strategy \ref{sg:curlconstant}. 
Moreover, it becomes evident that the conventional edge basis transforms into an $\mathcal{L}_{f_{st}}^1$-spline in the absence of convection, with $\mathcal{L}^1 = \nabla \times (\varepsilon \nabla \times)$.
As a result, these strategies emulate the characteristics exhibited by the standard edge basis of $\mathcal{P}_{1^-}^1(T)$.
\end{remark}

Without loss of generality, we will deduce $\bm{\phi}^E_{12}(\bm{x})$ and an auxiliary flux ${\bm{j}}^E_{12}(\bm{x})$ using the aforementioned strategies. To begin, applying Strategy \ref{sg:curlrestrict} along with equation \eqref{eq:surface-curl}, we obtain the following: 
\begin{equation}\label{eq:subcurl}
	\left\{
	\begin{aligned}
	\nabla_{f_{st}}^\perp \left(\widetilde{\bm{j}}_{12}^E(\bm{y}) \cdot \bm{\nu}_{st} \right) 
	&= \nabla_{f_{st}}^\perp \left( \varepsilon E_{\bm{\theta}}^{-1}(\bm{y})\nabla_{f_{st}} \times (E_{\bm{\theta}}(\bm{y}) \widetilde{\bm{\phi}}_{12}^E(\bm{y})) \right) = \bm{0}, & \text{ on } f_{st},\\
	\widetilde{\bm{\phi}}_{12}^E(\bm{y})\cdot {\bm{t}}_{st} &= \frac{\delta_{1s,2t}}{|E_{st}|}, & \text{ on }E_{st},
		\end{aligned}
	\right.
\end{equation}
for $1\leq s < t \leq 4$. Here, the second equation emulates the properties of the standard edge element basis. it reveals that $\widetilde{\bm{j}}_{12}^E(\bm{y}) \cdot {\bm{\nu}}_{st}$ remains constant across the sub-facet $f_{st}$, whence
\begin{equation} \label{eq:curl-const}
	\bm{j}_{12}^E(\bm{x}) \cdot \bm{\nu}_{st} = \varepsilon E^{-1}_{\bm{\theta}}(\bm{y})\nabla_{f_{st}} \times (E_{\bm{\theta}}(\bm{y}) \widetilde{\bm{\phi}}_{12}^E(\bm{y})) , \ \ \forall \bm{y} \in f_{st}.
\end{equation}

By multiplying the equation mentioned above with $E_{\bm{\theta}}(\bm{y})$ and integrating it over $f_{st}$, we derive
\begin{equation} \label{eq:curljphi}
{\bm{j}}_{12}^E(\bm{x}) \cdot \frac{{\bm{l}}_s \times {\bm{l}}_t}{2} \dashint_{f_{st}} E_{\bm{\theta}} 
= \varepsilon \int_{f_{st}} \nabla_{f_{st}} \times (E_{\bm{\theta}}(\bm{y}) \widetilde{\bm{\phi}}_{12}^E(\bm{y})),
\end{equation}
where we leverage the identity $\bm{\nu}_{st}=\frac{\bm{l}_s\times\bm{l}_t}{2|f_{st}|}$. 
Utilizing the Stokes theorem and Strategy \ref{sg:curlconstant}, we arrive at
$$
\begin{aligned}
 \int_{f_{st}} \nabla_{f_{st}} \times & (E_{\bm{\theta}}(\bm{y}) \widetilde{\bm{\phi}}_{12}^E(\bm{y}))  = 
\int_{\partial f_{st}} E_{\bm{\theta}}(\bm{y}) \widetilde{\bm{\phi}}_{12}^E(\bm{y})\cdot {\bm{\tau}} \\
	&= {\bm{\phi}}_{12}^E(\bm{x}) \cdot \bm{\tau}_s\int_{{{l}}_s} E_{\bm{\theta}}
	  + \frac{\delta_{1s,2t}}{|E_{st}|}  \int_{E_{st}}E_{\bm{\theta}}
	  - {\bm{\phi}}_{12}^E(\bm{x})\cdot \bm{\tau}_t\int_{{{l}}_t}E_{\bm{\theta}} \\
       & = {\bm{\phi}}_{12}^E(\bm{x}) \cdot \bm{l}_s \dashint_{{{l}}_s} E_{\bm{\theta}}  
       +  \delta_{1s,2t} \dashint_{E_{st}} E_{\bm{\theta}} - {\bm{\phi}}_{12}^E(\bm{x})\cdot \bm{l}_t \dashint_{{{l}}_t} E_{\bm{\theta}},
\end{aligned}
$$
where ${\bm{\tau}}$ denotes the unit tangential vector along the boundary of $f_{st}$. Now, utilizing the geometric interpretation of 2D-Bernoulli function \eqref{eq:bernoulli2d}, we have
$$
\varepsilon	\frac{ \dashint_{{{l}}_s} E_{\bm{\theta}} }{ \dashint_{f_{st}} E_{\bm{\theta}} } =  B_2^\varepsilon(\sigma_s,\sigma_t),
	\quad 
	\varepsilon \frac{\dashint_{{{l}}_t} E_{\bm{\theta}} }{\dashint_{f_{st}} E_{\bm{\theta}} }=B_2^\varepsilon(\sigma_t,\sigma_s),
	 \quad 
\varepsilon	\frac{\dashint_{E_{st}} E_{\bm{\theta}} }{\dashint_{f_{st}} E_{\bm{\theta}} }=B_2^\varepsilon(\sigma_{st},-\sigma_s),
$$
where $\sigma_i = \bm{\beta}\cdot \bm{l}_i$, $\sigma_j = \bm{\beta}\cdot \bm{l}_j$, and $\sigma_{ij} = {\bm{\beta}}\cdot \bm{E}_{ij} = \sigma_j-\sigma_i$ from \eqref{eq:sigma-diff}.

Thus, the equation \eqref{eq:subcurl} can be rewritten as
\begin{equation}\label{eq:algcurl}
\begin{aligned}
{\bm{j}}_{12}^E(\bm{x})\cdot \frac{\bm{l}_s\times\bm{l}_t}{2} 
&-  B_2^\varepsilon(\sigma_s,\sigma_t){\bm{\phi}}_{12}^E(\bm{x})  \cdot {\bm{l}}_s
 +  B_2^\varepsilon(\sigma_t,\sigma_s){\bm{\phi}}_{12}^E(\bm{x})\cdot {\bm{l}}_t \\
&=  B_2^\varepsilon(\sigma_{st},-\sigma_s) \delta_{1s,2t}, \quad \forall 1\leq s < t \leq 4,
	 \end{aligned}
\end{equation}
which defines a linear algebraic system for the unknowns ${\bm{\phi}}_{12}^E(\bm{x})$ and ${\bm{j}}_{12}^E(\bm{x})$. 
Once more, these extensions are solely employed for the purpose of deriving the algebraic equation, a determination guided by the outlined strategies.

\begin{problem}[$\mathcal{L}^1$-spline]\label{problem:curl}
	Find $\bm{\phi}_{12}^E(\bm{x}) $ and $\bm{j}_{12}^E(\bm{x})$ such that for all $x\in \bar{T}$

\begin{equation}\label{eq:syscurl}
	D^E(\bm{x})\left(\begin{matrix}
		{\bm{j}}_{12}^E(\bm{x}) \\		
		{\bm{\phi}}_{12}^E(\bm{x})
	\end{matrix}\right)=
	 B_2^\varepsilon(\sigma_{12},-\sigma_{1})\bm{e}_{12}^E,
\end{equation}
where $\bm{e}_{12}^E=(1,0,\cdots,0)^\top\in\mathbb{R}^6$ and $D^E(\bm{x})$ is a $6\times6$ matrix defined by
\begin{equation}
	D^E(\bm{x})=\left(\begin{matrix}
		({\bm{l}}_{1}\times{\bm{l}}_{2})^\top/2 & - B_2^\varepsilon(\sigma_{1},\sigma_{2}){\bm{l}}_{1}^\top+ B_2^\varepsilon(\sigma_{2},\sigma_{1}){\bm{l}}_{2}^\top \\
		({\bm{l}}_{1}\times{\bm{l}}_{3})^\top/2 & - B_2^\varepsilon(\sigma_{1},\sigma_{3}){\bm{l}}_{1}^\top+ B_2^\varepsilon(\sigma_{3},\sigma_{1}){\bm{l}}_{3}^\top \\
		({\bm{l}}_{1}\times{\bm{l}}_{4})^\top/2 & - B_2^\varepsilon(\sigma_{1},\sigma_{4}){\bm{l}}_{1}^\top+ B_2^\varepsilon(\sigma_{4},\sigma_{1}){\bm{l}}_{4}^\top \\
		({\bm{l}}_{2}\times{\bm{l}}_{3})^\top/2 & - B_2^\varepsilon(\sigma_{2},\sigma_{3}){\bm{l}}_{2}^\top+ B_2^\varepsilon(\sigma_{3},\sigma_{2}){\bm{l}}_{3}^\top \\
		({\bm{l}}_{2}\times{\bm{l}}_{4})^\top/2 & - B_2^\varepsilon(\sigma_{2},\sigma_{4}){\bm{l}}_{2}^\top+ B_2^\varepsilon(\sigma_{4},\sigma_{2}){\bm{l}}_{4}^\top \\
		({\bm{l}}_{3}\times{\bm{l}}_{4})^\top/2 & - B_2^\varepsilon(\sigma_{3},\sigma_{4}){\bm{l}}_{3}^\top+ B_2^\varepsilon(\sigma_{4},\sigma_{3}){\bm{l}}_{4}^\top 
	\end{matrix}\right).
\end{equation}
\end{problem}

The solution to Problem~\ref{problem:curl} (whose later proof establishes its uniqueness or existence) defines the point values of the spline basis function $\bm{\phi}_{12}^E$ and the auxiliary flux $\bm{j}_{12}^E$ at $\bm{x} \in \bar{T}$. A similar process is carried out for $\bm{\phi}_{ij}^E$ and $\bm{j}_{ij}^E~(1\leq i < j \leq 4)$ associated with the edge $E_{ij}$. This concludes the construction of the shape function space given by \eqref{eq:curl-shape}.

\begin{remark}[variable $\bm{\beta}$]
In the aforementioned derivation, we assume that $\bm{\theta}$ (or $\bm{\beta}$) is a constant vector field. In fact, the final Problem \ref{problem:curl} only provides the values of the basis functions and their corresponding fluxes at the point $\bm{x}$. Therefore, when solving for the point $\bm{x}$, it is only necessary to use $\bm{\beta}(\bm{x})$ to evaluate the corresponding $\sigma_i$ and $\sigma_{ij}$ for constructing the matrix $D^E(\bm{x})$ and the right-hand side term.
\end{remark}

\begin{remark}[approximation property of $ \mathcal{S}_{1^-}^1$] \label{rm:approx-curl}
	It is straightforward that any constant vector $\bm{c}$ is contained in $\mathcal{S}_{1^-}^1(T)$ since the constant vector $\bm{c}$ and  the corresponding auxiliary flux is $\bm{\beta}(\bm{x})\times\bm{c}$ satisfy all the Strategies in the construction.
\end{remark}

\begin{remark}[local smoothness of $ \mathcal{S}_{1^-}^1(T)$]
If $\bm{\beta}$ is $C^1$ in the element $T$, then $D^E(\bm{x})$ is also $C^1$ in the element, as the Bernoulli function is smooth. Consequently, the spline basis is also $C^1$ within the element.
\end{remark}
 
\subsection{Well-posedness of $\mathcal{S}_{1^-}^1(T)$}
We will establish the well-posedness of Problem~\ref{problem:curl}, confirming its unique solvability for all $\bm{x}\in \bar{T}$.
\begin{lemma}[well-posedness of Problem \ref{problem:curl}] \label{lemma:curlunisolve}
For any $\bm{x}\in \bar{T}$, there exists a unique solution to Problem~\ref{problem:curl}.
\end{lemma}
\begin{proof}
Without loss of generality, we make the assumption that $\lambda_1(\bm{x}) > 0$, and by utilizing the identity \eqref{eq:identityli}, we derive the following expression:
$$
\bm{0} = \left(\sum_{i=1}^4 \lambda_i \bm{l}_i \right) \times \bm{l}_j = \left(\sum_{i\neq j}^4 \lambda_i \bm{l}_i \right) \times \bm{l}_j, \quad \text{ for }j=2,3,4.
$$
This allows us to eliminate the upper left block of the matrix $D^E(\bm{x})$, resulting in the following:
$$
\begin{aligned}
{\lambda_1^3}\det&(D^E(\bm{x}))=\\
&\left|\begin{matrix}
	\bm{0}^\top &\sum_{i\neq 2} \lambda_iB_2^\varepsilon{(\sigma_2,\sigma_i)}\bm{l}_2^\top-\lambda_iB_2^\varepsilon{(\sigma_i,\sigma_2)}\bm{l}_i^\top \\
	\bm{0}^\top  &\sum_{i\neq 3} \lambda_iB_2^\varepsilon{(\sigma_3,\sigma_i)}\bm{l}_3^\top-\lambda_iB_2^\varepsilon{(\sigma_i,\sigma_3)}\bm{l}_i^\top \\
	\bm{0}^\top &\sum_{i\neq 4} \lambda_iB_2^\varepsilon{(\sigma_4,\sigma_i)}\bm{l}_4^\top-\lambda_iB_2^\varepsilon{(\sigma_i,\sigma_4)}\bm{l}_i^\top \\
	({\bm{l}}_{2}\times{\bm{l}}_{3})^\top/2  &- B_2^\varepsilon(\sigma_{2},\sigma_{3}){\bm{l}}_{2}^\top+ B_2^\varepsilon(\sigma_{3},\sigma_{2}){\bm{l}}_{3}^\top \\
	({\bm{l}}_{2}\times{\bm{l}}_{4})^\top/2  &- B_2^\varepsilon(\sigma_{2},\sigma_{4}){\bm{l}}_{2}^\top+ B_2^\varepsilon(\sigma_{4},\sigma_{2}){\bm{l}}_{4}^\top \\
	({\bm{l}}_{3}\times{\bm{l}}_{4})^\top/2  &- B_2^\varepsilon(\sigma_{3},\sigma_{4}){\bm{l}}_{3}^\top+ B_2^\varepsilon(\sigma_{4},\sigma_{3}){\bm{l}}_{4}^\top 
\end{matrix}\right|:=\left|\begin{matrix}
0 & A \\ B & C
\end{matrix}\right|.
\end{aligned}
$$
Substituting $\lambda_1\bm{l}_1=-\sum_{i=2}^{4}\lambda_i\bm{l}_i$, we obtain
$$
A =  \tilde{{A}}
\left(\begin{matrix}
	 \bm{l}_2^\top \\  \bm{l}_3^\top\\\bm{l}_4^\top
\end{matrix}\right),
$$
where
$$\tilde{{A}}= \left(\begin{matrix}
	\lambda_2B^{12}+\sum_{i\neq 2}\lambda_iB^{2i} & \lambda_3(B^{12}-B^{32})& \lambda_4(B^{12}-B^{42})\\
	\lambda_2(B^{13}-B^{23})&	\lambda_3B^{13}+\sum_{i\neq 3}\lambda_iB^{3i} & \lambda_4(B^{13}-B^{43})\\
	\lambda_2(B^{14}-B^{24})& \lambda_3(B^{14}-B^{34})&	\lambda_4B^{14}+\sum_{i\neq 4}\lambda_iB^{4i} 
\end{matrix}\right),
$$
with $B^{ij} :=B_2^\varepsilon(\sigma_i,\sigma_j)$ for simplicity.
Observe that $\bm{l}_2$, $\bm{l}_3$, and $\bm{l}_4$ are linearly independent when $\lambda_1 > 0$. Hence, to establish the nonsingularity of matrix ${A}$, it suffices to demonstrate the nonsingularity of matrix $\tilde{{A}}$. By a direct calculation, we can ascertain that the determinant of $\tilde{{A}}$ follows the expression
$$
\begin{aligned}
\det(\tilde{{A}}) &= \sum_{m=1}^4\sum_{\substack{1\le i<j<k\le 4,\\i,j,k\neq m}}\sum_{\substack{1\le r,s,t\le 4\\(r,s,t)\in S^E_m}} B^{ir}B^{js}B^{kt}\lambda_r\lambda_s\lambda_t,
\end{aligned}
$$
where 
$$
\begin{aligned}
	S^E_1 &=\{(r,s,t)|r\neq 2,s\neq 3,t\neq 4 \text{ and } 1 \in \{r,s,t\} \} \backslash\left\{(1,4,3),(3,2,1),(4,1,2)\right\},\\
	S^E_2 &=\{(r,s,t)|r\neq 1,s\neq 3,t\neq 4 \text{ and } 2 \in \{r,s,t\} \} \backslash \left\{(2,4,3),(3,1,2),(4,2,1)\right\},\\
	S^E_3 &=\{(r,s,t)|r\neq 1,s\neq 2,t\neq 4 \text{ and } 3 \in \{r,s,t\} \} \backslash \left\{(2,1,3),(3,4,2),(4,3,1)\right\},\\
	S^E_4  &=\{(r,s,t)|r\neq 1,s\neq 2,t\neq 3 \text{ and } 4 \in \{r,s,t\} \} \backslash \left\{(2,1,4),(3,1,4),(4,3,2)\right\}.
\end{aligned}
$$
The positivity of the Bernoulli function, $\lambda_1 > 0$, and the inclusion of $(1,1,1)$ in $S_1^E$ collectively establish the nonsingularity of $\tilde{A}$, and consequently, the nonsingularity of the matrix $D^E(\bm{x})$. Therefore, it yields that $\det(D^E(\bm{x})) > 0$.
\end{proof}


We present visualizations of the $ \mathcal{S}_{1^-}^1(T)$ basis function on a reference element. In Figure \ref{fig:curlbasis1}, we depict plots showcasing the first component of the $\bm{H}({\rm curl})$ exponentially-fitted basis function along with a vector plot for varying parameters: $\varepsilon=1$ and $\varepsilon=0.01$, in the case of $\bm{\beta}=(1,2,3)^T$. It can be observed that the function exhibits continuity and a linear-like basis for a large $\varepsilon$, while displaying an exponential-like basis for a small $\varepsilon$. This indicates the flexibility of the function in accommodating both convection-dominated and diffusion-dominated cases. 

\begin{figure}[!htbp]
 \centering
 \subfloat[contour of the first component $(\varepsilon=1)$]{\includegraphics[width=.4\textwidth]{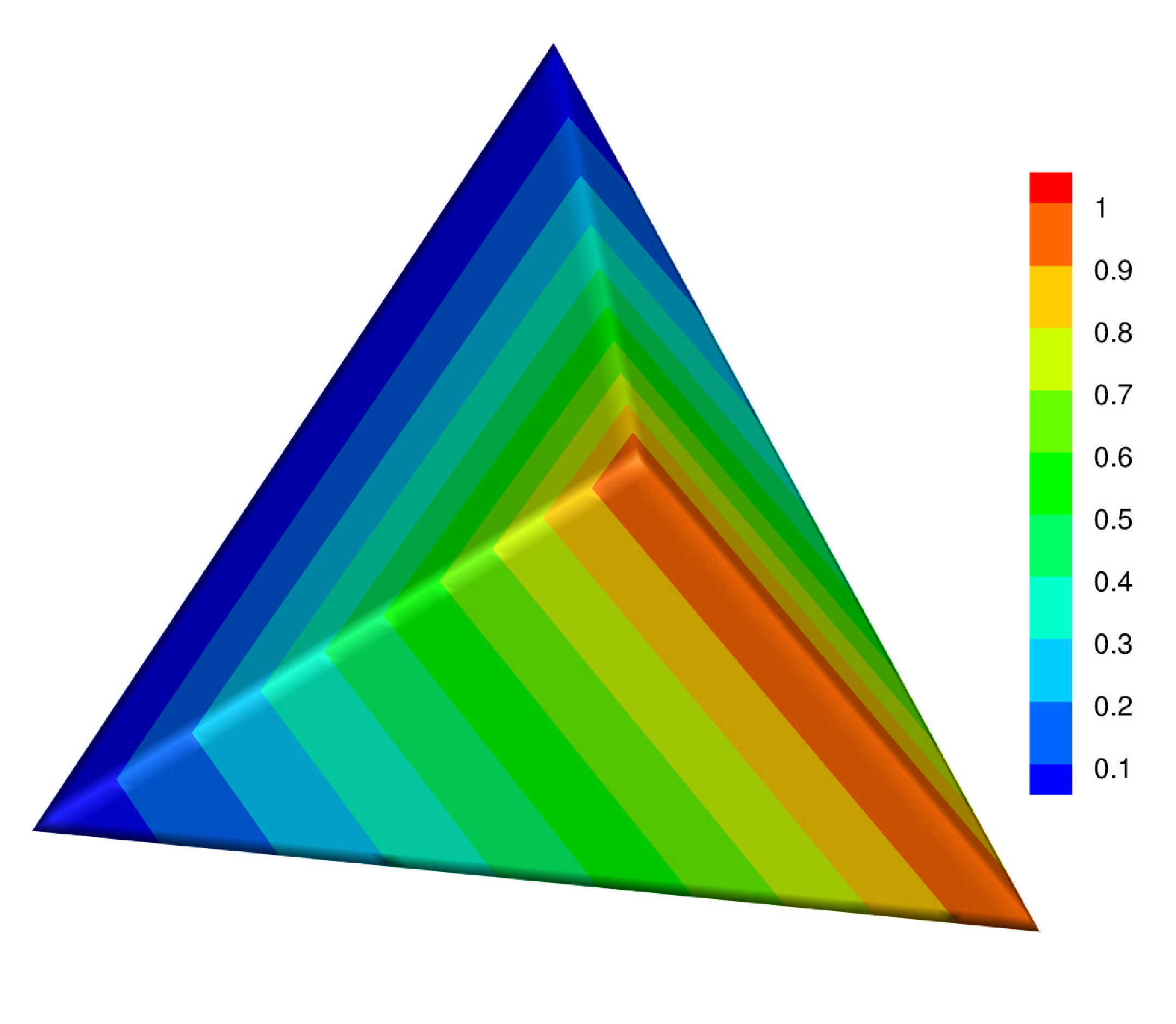}}
 \subfloat[vector plot $(\varepsilon=1)$]{\includegraphics[width=.42\textwidth]{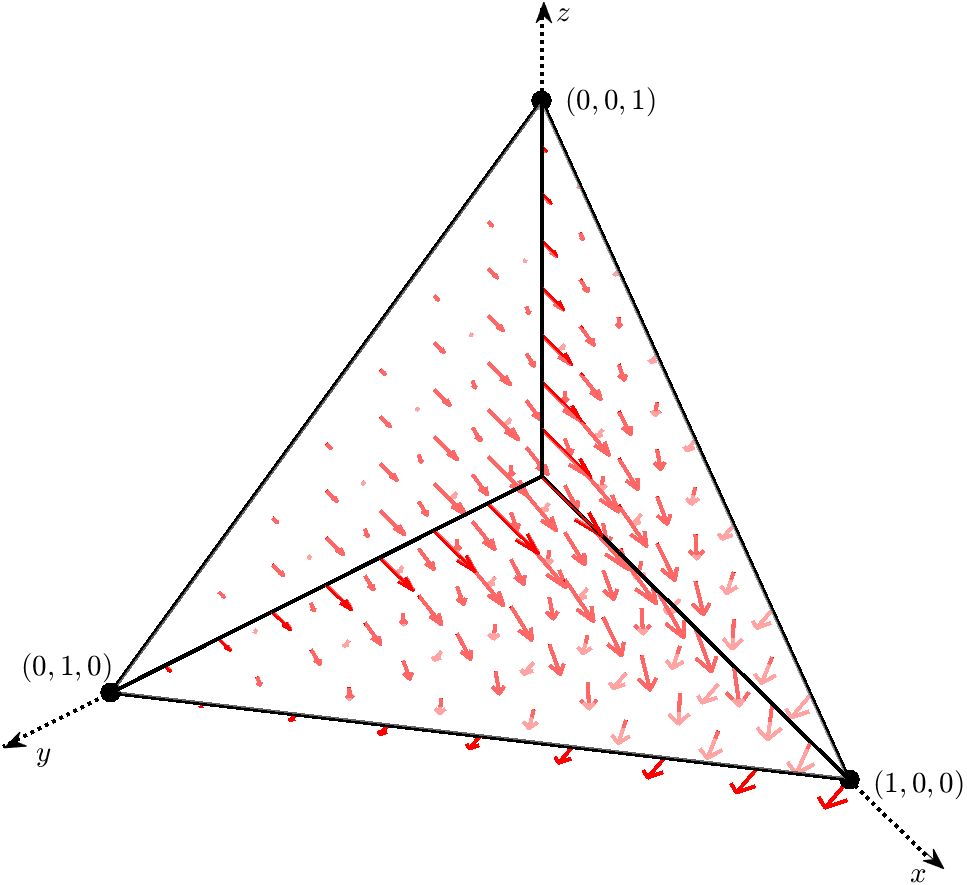}}\\
 \subfloat[contour of the first component $(\varepsilon=0.01)$]{\includegraphics[width=.42\textwidth]{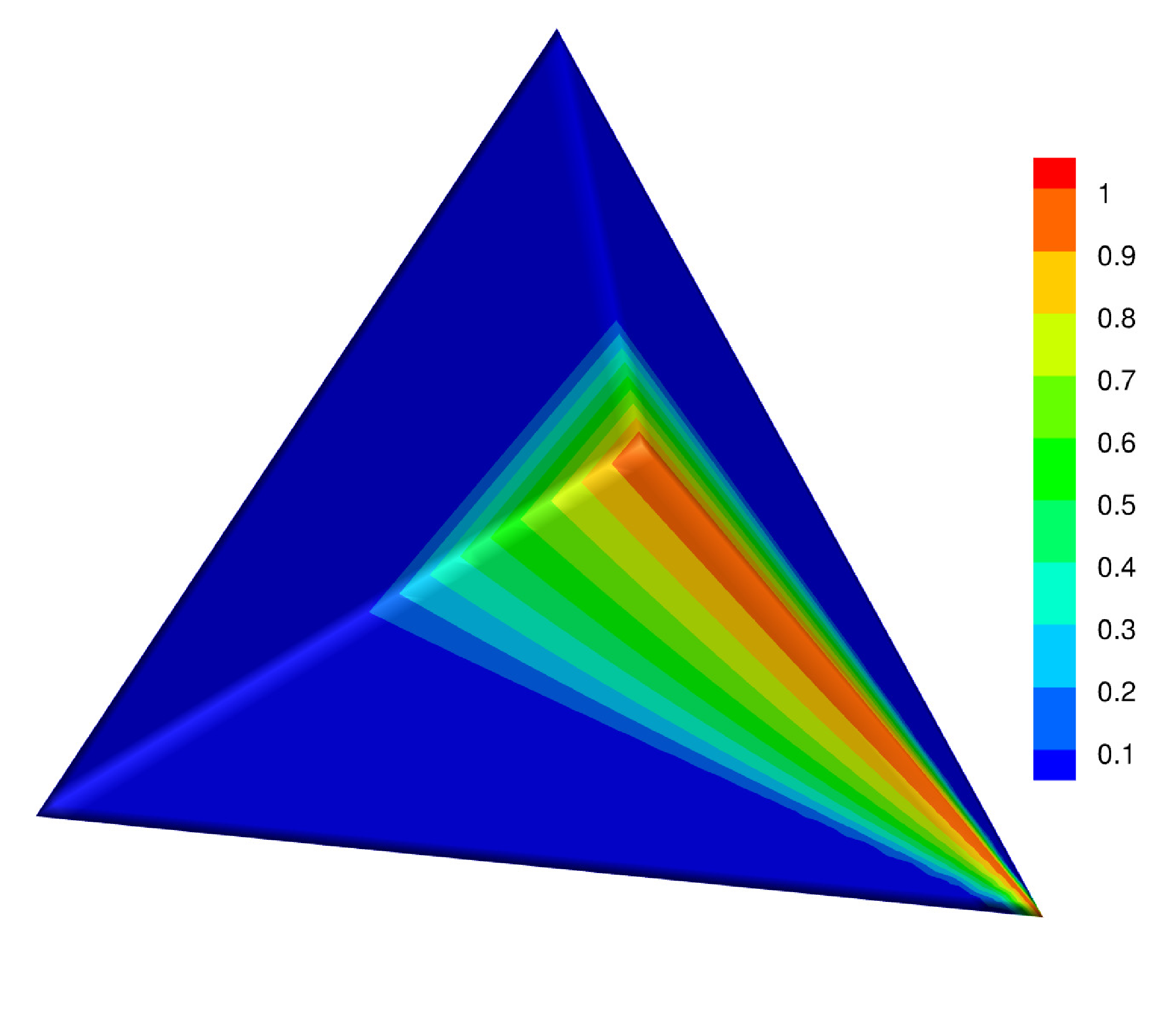}}
 \subfloat[vector plot $(\varepsilon=0.01)$]{\includegraphics[width=.4\textwidth]{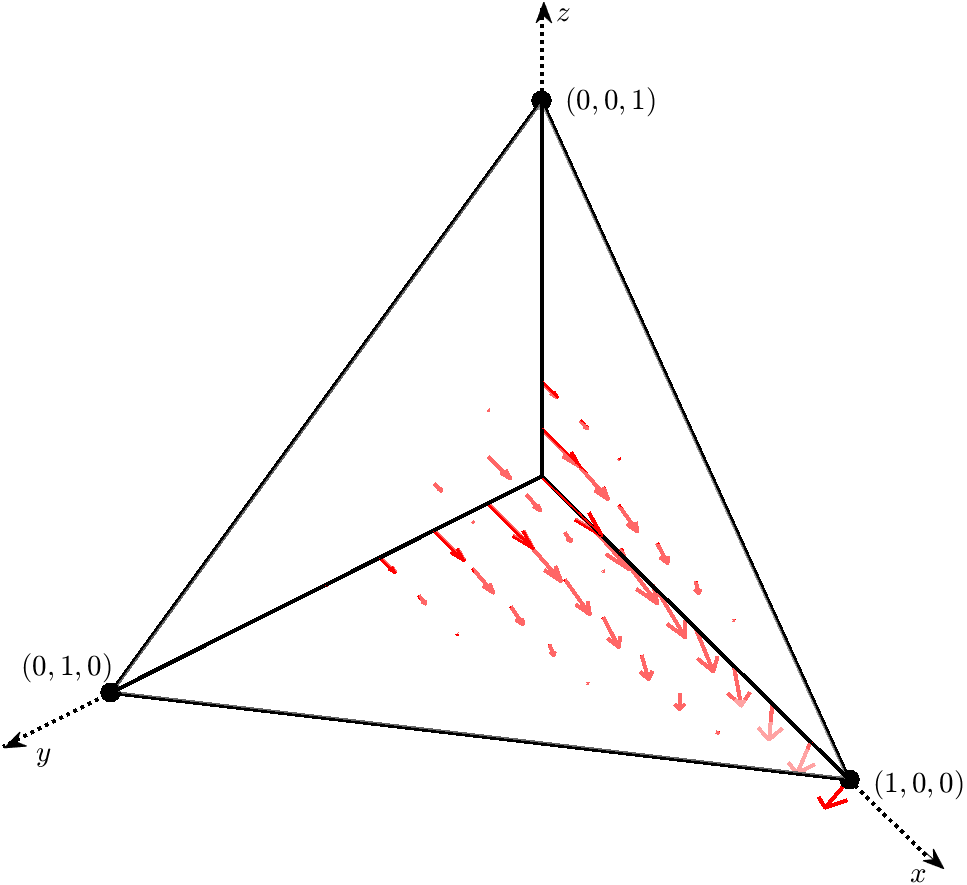}}
 \caption{Visualization of $\bm{H}({\rm curl})$ exponentially-fitted basis on a reference element: contour of the first component (left) and vector plot (right) for cases dominated by diffusion (top) and convection (bottom).}\label{fig:curlbasis1}
\end{figure}

\subsection{Properties of $\mathcal{S}_{1^-}^1$}

We firstly establish an identical set of unisolvent degrees of freedom (DOFs) for $ \mathcal{S}_{1^-}^1(T)$, matching those of the standard edge element.
\begin{lemma}[DOFs of $\mathcal{S}_{1^-}^1(T)$]\label{lemma:curldof} It holds that
	\begin{equation}\label{eq:curldof}
		{\bm{\phi}}_{ij}^E(\bm{x})\cdot {\bm{t}}_{st}= \frac{\delta_{is,jt}}{|E_{st}|}, \quad \forall \bm{x} \in E_{st}, ~1\leq s<t\leq 4,
	\end{equation}
 without necessitating that $\bm{\beta}$ is a constant vector.
\end{lemma}
\begin{proof}
	As shown in Figure \ref{fig:curl-dof}, we revisit equation \eqref{eq:algcurl} governing the basis function $\bm{\phi}^E_{ij}(\bm{x})$ on edge ${E}_{st}$:
	$$
	\frac{\bm{l}_s\times \bm{l}_t}{2} \cdot \bm{j}_{ij}^E(\bm{x}) + \big( -B_2^\varepsilon(\sigma_s,\sigma_t) \bm{l}_s 
	+ B_2^\varepsilon(\sigma_t,\sigma_s) \bm{l}_t\big)\cdot \bm{\phi}^E_{ij}(\bm{x})  = \delta_{is,jt}B_2^\varepsilon(\sigma_{st},-\sigma_{s}).
	$$
	Notably, on edge $E_{st}$, we have $\bm{l}_s\parallel\bm{l}_t$, resulting in $\bm{l}_s\times\bm{l}_t=\bm{0}$. Additionally, $\bm{l}_s = -|\bm{l}_s| \bm{t}_{st}$ and $\bm{l}_t = |\bm{l}_t|\bm{t}_{st}$. Thus, the aforementioned equation simplifies, for $\bm{x} \in E_{st}$:
	$$
	\big(B_2^\varepsilon(\sigma_s,\sigma_t) |\bm{l}_s| + B_2^\varepsilon(\sigma_t,\sigma_s\big) |\bm{l}_t|) \bm{t}_{st} \cdot \bm{\phi}^E_{ij}(\bm{x}) = 
	\delta_{is,jt}B_2^\varepsilon(\sigma_{st},-\sigma_{s}).
	$$
	By employing the geometric interpretation of the 2D-Bernoulli function, we deduce that on edge $E_{st}$:
	$$
	B_2^\varepsilon(\sigma_s,\sigma_t) |\bm{l}_s| + B_2^\varepsilon(\sigma_t,\sigma_s) |\bm{l}_t| 
	=  \varepsilon\frac{  \int_{l_s} E_{\bm{\theta}(\bm{x})} +\int_{l_t} E_{\bm{\theta}(\bm{x})}}{ \dashint_{f_{st}} E_{\bm{\theta}(\bm{x})} }=\varepsilon\frac{\int_{E_{st}} E_{\bm{\theta}(\bm{x})}}{\dashint_{f_{st}} E_{\bm{\theta}(\bm{x})} }=|E_{st}|B_2^\varepsilon(\sigma_{st},-\sigma_s).
	$$
	Hence, the desired result \eqref{eq:curldof} follows.
\end{proof}
\begin{figure}[!htbp]
	\centering
	\subfloat[Illustration for DOF]{
		\includegraphics[width=.3\textwidth]{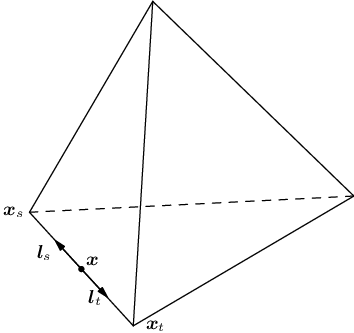} \label{fig:curl-dof}
	}
	\subfloat[Illustration for conformity]{
		\includegraphics[width=.3\textwidth]{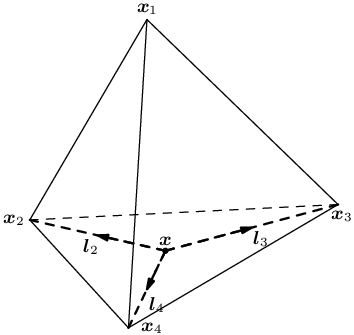} \label{fig:curl-conform}
	}

	\caption{Visual representation in the proof of $\mathrm{curl}$ DOFs and conformity.}
	\label{fig:geo-notation}
\end{figure}
So far, we have established the shape function space of $ \mathcal{S}_{1^-}^1$ and the unisolvent DOFs, thus concluding the construction of the finite element space. For a given mesh $\mathcal{T}_h$, its finite element space is denoted by
\begin{equation} \label{eq:curl-space}
 \mathcal{S}_{1^-}^1(\mathcal{T}_h) := \{\bm{v}|_T \in \mathcal{E}_1^-\Lambda^1(T),\forall T \in \mathcal{T}_h:~ \bm{v}\cdot \bm{t} \text{ is continuous across all edges} \}.
\end{equation}
Surprisingly, despite $ \mathcal{S}_{1^-}^1$ not being comprised of local polynomials, the aforementioned finite element space still maintains the comformity.

\begin{lemma}[$\bm{H}(\mathrm{curl})$-conformity]\label{lm:curlconform}
	The local basis functions satisfy
	\begin{equation} \label{eq:curl-conforming-Fs}
	\bm{n}_s \times \bm{\phi}^E_{ij}=\bm{0},  \text{ on } F_s \subset \partial T, ~\text{for }s=i,j,
	\end{equation}
	without necessitating that $\bm{\beta}$ is a constant vector. As a consequence, if $\bm{\beta}$ is tangential continuous across any facet $F \in \mathcal{F}_h^0$ and piecewise $C^1$, then the function space $\mathcal{S}_{1^-}^1(\mathcal{T}_h)$ is $\bm{H}({\rm curl})$-conforming.
\end{lemma}

\begin{proof}
	Without loss of generality, we study the behavior of basis $\bm{\phi}^E_{ij}~(1\leq i < j \leq 4)$ on the facet $F_1$, as depicted in Figure \ref{fig:curl-conform}. For $\bm{x} \in F_1$,  Problem~\ref{problem:curl} gives
	$$
	\left(\begin{matrix}
	({\bm{l}}_{2}\times{\bm{l}}_{3})^\top/2  
	&- B_2^\varepsilon(\sigma_{2},\sigma_{3}){\bm{l}}_{2}^\top+ B_2^\varepsilon(\sigma_{3},\sigma_{2}){\bm{l}}_{3}^\top \\
	({\bm{l}}_{2}\times{\bm{l}}_{4})^\top/2  
	&- B_2^\varepsilon(\sigma_{2},\sigma_{4}){\bm{l}}_{2}^\top+ B_2^\varepsilon(\sigma_{4},\sigma_{2}){\bm{l}}_{4}^\top \\
	({\bm{l}}_{3}\times{\bm{l}}_{4})^\top/2  
	&- B_2^\varepsilon(\sigma_{3},\sigma_{4}){\bm{l}}_{3}^\top+ B_2^\varepsilon(\sigma_{4},\sigma_{3}){\bm{l}}_{4}^\top 
	\end{matrix}\right)\left(\begin{matrix}
	\bm{j}_{ij}^E(\bm{x}) \\ \bm{\phi}_{ij}^E(\bm{x})
	\end{matrix}\right) = \left(\begin{matrix}
	 B_2^\varepsilon(\sigma_{23},-\sigma_2) \delta_{2i,3j}\\  B_2^\varepsilon(\sigma_{24},-\sigma_2)\delta_{2i,4j}\\ B_2^\varepsilon(\sigma_{34},-\sigma_3)\delta_{3i,4j}
	\end{matrix}\right).
	$$
	Observing that $\lambda_2+\lambda_3+\lambda_4=1$ holds true for $\bm{x}\in F_1$, we can assume $\lambda_2>0$ to eliminate the coefficient of $\bm{j}_{ij}^E(\bm{x})$ through the equations $\bm{0}=(\lambda_2\bm{l}_2+\lambda_4\bm{l}_4)\times\bm{l}_3=(\lambda_2\bm{l}_2+\lambda_3\bm{l}_3)\times\bm{l}_4$. Consequently, we arrive at
	\begin{equation} \label{eq:curl-F1}
	A_{1}\bm{\phi}_{ij}^E(\bm{x}) 
	= b_1 :=
	\left(\begin{matrix}
	\lambda_2 B_2^\varepsilon(\sigma_{23},-\sigma_2) \delta_{2i,3j}-\lambda_4B_2^\varepsilon (\sigma_{34},-\sigma_3)\delta_{3i,4j}\\\lambda_2B_2^\varepsilon(\sigma_{24},-\sigma_2)\delta_{2i,4j}+\lambda_3B_2^\varepsilon(\sigma_{34},-\sigma_3)\delta_{3i,4j}
	\end{matrix}\right),
	\end{equation}
	where $A_1 \in \mathbb{R}^{2\times 3}$ is defined as
	$$
	A_1 := \left(
	\begin{matrix}
		\lambda_2(-B^{23}{\bm{l}}_{2}^\top + B^{32} {\bm{l}}_{3}^\top ) - 
		\lambda_4(-B^{34} {\bm{l}}_{3}^\top+B^{43} {\bm{l}}_{4}^\top) \\
		\lambda_2(-B^{24}{\bm{l}}_{2}^\top+B^{42}{\bm{l}}_{4}^\top ) + 
		\lambda_3(-B^{34}{\bm{l}}_{3}^\top+B^{43}{\bm{l}}_{4}^\top )
	\end{matrix}\right).
	$$
	Once again, we utilize the conventional simplified notation $B^{ij} :=B_2^\varepsilon(\sigma_i,\sigma_j)$.
	Subsequently, by substituting $\bm{l}_2$ with $\bm{l}_3$ and $\bm{l}_4$ via $\sum_{i=2}^4\lambda_i\bm{l}_i=\bm{0}$, and taking into consideration the linear independence of $\bm{l}_3$ and $\bm{l}_4$ due to $\lambda_2>0$, we can reformulate \eqref{eq:curl-F1} as follows:
	\begin{equation}\label{eq:curl-F1-2}
	\tilde{A}_{1} \left(\begin{matrix}
		\bm{\phi}^E_{ij}(\bm{x}) \cdot\bm{l}_3 \\ 
		\bm{\phi}^E_{ij}(\bm{x}) \cdot\bm{l}_4
	\end{matrix}\right) = b_1,
	\end{equation}
	where 
	$$
	\tilde{A}_{1}=\left(\begin{matrix}
		\lambda_3B^{23}+\lambda_2B^{32}+\lambda_4B^{34} &\lambda_4B^{23}-\lambda_4B^{43}\\
		\lambda_3B^{24}-\lambda_3B^{34} & \lambda_4B^{24}+\lambda_2B^{42}+\lambda_3B^{43}
	\end{matrix}\right).
	$$
	A direct computation, coupled with the positivity of the Bernoulli function, also establishes the non-singularity of $\tilde{A}_1$. In cases where the edge $E_{ij}$ is not encompassed within the facet $F_1$, the term $b_1$ in \eqref{eq:curl-F1} becomes trivial. Consequently, we deduce that $\bm{\phi}_{ij}^E(\bm{x}) \cdot \bm{l}_{s}=0~(s=3,4)$, thereby leading to $\bm{\phi}_{ij}^E(\bm{x}) \times\bm{n}_1= \bm{0}$, resulting in \eqref{eq:curl-conforming-Fs}.
	
Conversely, if edge $E_{ij}$ is contained in facet $F_1$, we consider $T^+$ and $T^-$ that share $F_1$ as a common facet.
Note that $\tilde{A}_1$ and $b_1$ depends only on $\bm{l}_s$ and $\sigma_s = \bm{\beta}(\bm{x})\cdot \bm{l}_s~(s=2,3,4)$, which is continuous across the facet when the tangential component of $\bm{\beta}$ on $F_1$ is continuous. Then, the linear systems \eqref{eq:curl-F1-2} are the same for $T^\pm$.
Therefore, $\bm{\phi}_{ij}^E\cdot \bm{l}_s (s=3,4)$ is continuous across the facet in this case, which also means that $\bm{\phi}_{ij}^E\times \bm{n}_1$ is continuous across the facet. Above all, we obtain the $\bm{H}(\mathrm{curl})$-conformity.
\end{proof}

\section{Exponentially-fitted finite element space in $H({\rm div})$}\label{sec:div}
In this section, we will construct the exponentially-fitted FE space in $\bm{H}(\mathrm{div})$.  In the construction, we once again rely on the assumption that $\bm{\theta}$ is a constant field, allowing us to invoke \eqref{eq:identity}
$$
J_{\bm{\theta}}^2 \bm{u} = \varepsilon(\nabla \cdot \bm{u} + \bm{\theta} \cdot \bm{u}) = \varepsilon E_{\bm{\theta}}^{-1} \nabla \cdot (E_{\bm{\theta}} \bm{u}).
$$ 
By utilizing equation \eqref{eq:div-cd}, we can express $\mathcal{L}^2\bm{u} = -\nabla(J_{\bm{\theta}}^2 \bm{u})$.

\subsection{$\mathcal{L}^2$-spline shape function space}
Given an element $T$, we intend to construct the shape function space as defined below:
\begin{equation} \label{eq:div-shape}
\mathcal{S}_{1^-}^2(T):={\rm{span}} \{\bm{\phi}^F_i:~ 1\leq i \leq 4\},
\end{equation}
where $\bm{\phi}_{i}^F$ is the basis function associated with the facet $F_{i}$, which we will define explicitly later. Continuing with a logic similar to our previous approach, we will provide both $\bm{\phi}_i^F(\bm{x})$ and its corresponding flux $j_i^F(\bm{x})$ concurrently for each $\bm{x} \in \bar{T}$.

For a fixed point $\bm{x} \in \bar{T}$, we adopt a geometric convention (illustrated in Figure \ref{fig:sub-tetrahedron}). We proceed to appropriately extend $\bm{\phi}_{i}^F(\bm{x})$ and $j_{i}^F(\bm{x})$, denoted by $\widetilde{\bm{\phi}}_{i}^F(\bm{y})$ and $\widetilde{j}_{i}^F(\bm{y})$, onto sub-tetrahedrons $T_j~(1\leq j \leq 4)$. We introduce two distinct strategies as the essential components in the construction of the shape function space.

\begin{strategy}[normal constant on sub-facets]\label{sg:divconstant}
	Given $\bm{x}\in \bar{T}$, the extension $\widetilde{\bm{\phi}}_i^F(\bm{y})$ satisfies
	$$
	\widetilde{\bm{\phi}}_{i}^F(\bm{y}) \cdot \bm{\nu}_{st} = {\bm{\phi}}_{i}^F(\bm{x}) \cdot \bm{\nu}_{st}, \quad \forall \bm{y} \in f_{st}, ~1\leq s < t \leq 4.
	$$
\end{strategy}
\begin{strategy}[$\mathcal{L}^2$-spline in sub-tetrahedrons]\label{sg:divrestrict}
	Given $\bm{x} \in \bar{T}$, $\widetilde{\bm{\phi}}_i^F(\bm{y})$ is an $\mathcal{L}^2$-spline, i.e., 
	$$
	\mathcal{L}^2 \widetilde{\bm{\phi}}^F_{i} = -\nabla \widetilde{j}_i^F = \bm{0}, \ \ \text{ in } T_{j},~ 1 \leq j \leq 4.
	$$
\end{strategy}
\begin{remark}[consistency with the lowest order Raviart-Thomans element]
	When the convection field $\bm{\beta}$ vanishes, then the exponential function $E_{\bm{\theta}}\equiv1$ and the operator $\mathcal{L}^2$ only has the diffusion component. As is well known that the normal component of lowest order Raviart-Thomas element space (denoted by $\mathcal{P}_{1^-}^2(T)$ in accordance with the naming convention used in this paper) is constant on any fixed plane, which gives the Strategy \ref{sg:divconstant} exactly. Furthermore, it is straightforward that $\mathcal{P}_{1^-}^2(T)$ is $\mathcal{L}^2$-free when the convection is absent. Therefore, these strategies are mimetic to the properties of $\mathcal{P}_{1^-}^2(T)$, which are important for us to establish the subsequent algebraic equations for the $\mathcal{L}^2$-spline basis.
\end{remark}

Without loss of generality, we will deduce $\bm{\phi}^F_{1}(\bm{x})$ and an auxiliary flux ${j}^F_{1}(\bm{x})$ using the aforementioned strategies. 
Utilizing Strategy \ref{sg:divrestrict}, we have 
\begin{equation}\label{eq:subdiv}
	\left\{\begin{aligned}
		 -\nabla \widetilde{j}^F_1(\bm{y}) 
		&=- \nabla \left( \varepsilon E_{\bm{\theta}}^{-1}(\bm{y}) \nabla\cdot (E_{\bm{\theta}}(\bm{y}) \widetilde{\bm{\phi}}_1^F(\bm{y})) \right)= \bm{0},  & \text{ in } T_j,\\
		 \widetilde{\bm{\phi}}^F_1(\bm{y}) \cdot \bm{n}_j &= \frac{\delta_{1j}}{|F_j|}, & \text{ on }F_j,
	\end{aligned}\right.
\end{equation}
for $1 \leq j \leq 4$. Here, the second equation emulates the properties of the lowest-order Raviart-Thomas element basis. It can be easily derived that  that $\widetilde{j}_1^F(\bm{y})$ is constant in the tetrahedron $T_j$, whence 
\begin{equation} \label{eq:div-const}
j^F_1(\bm{x}) = \varepsilon E^{-1}_{\bm{\theta}}(\bm{y}) \nabla\cdot  (E_{\bm{\theta}}(\bm{y})\widetilde{\bm{\phi}}_1^F(\bm{y})), \quad \forall \bm{y} \in T_j.
\end{equation}

By multiplying the equation mentioned above with $E_{\bm{\theta}}(\bm{y})$ and integrating it over $T_{j}$, we obtain
\begin{equation}\label{eq:divjphi}
	j_1^F(\bm{x})\int_{T_j}E_{\bm{\theta}} = \varepsilon\int_{T_j} \nabla\cdot (E_{\bm{\theta}}(\bm{y}) \widetilde{\bm{\phi}}^F_1(\bm{y})) = \varepsilon\int_{\partial T_j} E_{\bm{\theta}}(\bm{y}) \widetilde{\bm{\phi}}_1^F(\bm{y}) \cdot \bm{n},
\end{equation}
where $\bm{n}$ denotes the unit outward normal vector of the boundary of $T_j$. Using Strategy \ref{sg:divconstant}, we have
$$
\begin{aligned}
	&\quad \int_{\partial T_j}  E_{\bm{\theta}}(\bm{y}) \widetilde{\bm{\phi}}_1^F(\bm{y}) \cdot \bm{n}
	= \sum_{1 \leq s < t \leq 4, s,t\neq j} \int_{f_{st}}E_{\bm{\theta}}(\bm{y}) \widetilde{\bm{\phi}}_1^F(\bm{y})\cdot \bm{n}
	   + \frac{\delta_{1j}}{|F_j|} \int_{F_j}E_{\bm{\theta}}\\
	&=  \sum_{1 \leq s < t \leq 4, s,t\neq j} (\pm)_{jst} \bm{\phi}_1^F(\bm{x}) \cdot \bm{\nu}_{st}\int_{f_{st}}E_{\bm{\theta}} 
	  + \frac{\delta_{1j}}{|F_j|} \int_{F_j}E_{\bm{\theta}}, \\
	&=  \sum_{1 \leq s < t \leq 4, s,t\neq j} (\pm)_{jst} \bm{\phi}_1^F(\bm{x}) \cdot \frac{\bm{l}_s \times \bm{l}_t}{2}\dashint_{f_{st}}E_{\bm{\theta}} 
	  + \delta_{1j} \dashint_{F_j}E_{\bm{\theta}},
\end{aligned}
$$
where $(\pm)_{jst}$ represents the sign of orientation along with the unit outward normal, or more precisely, $(\pm)_{jst} := \mathrm{sgn}\big((\bm{x}_s - \bm{x}_j) \times (\bm{x}_t - \bm{x}_j) \cdot \bm{n}_{\widehat{jst}} \big)$ where $\widehat{jst}$ is the number ranging from 1 to 4, excluding $j$, $s$, and $t$. Now, utilizing the geometric interpretation of 3D-Bernoulli function \eqref{eq:bernoulli3d}, we  have
$$
\begin{aligned}
\varepsilon\frac{\dashint_{f_{st}}E_{\bm{\theta}}}{\dashint_{T_j}E_{\bm{\theta}}} &= B_3^\varepsilon(\sigma_s, \sigma_t, \sigma_{\widehat{jst}}),\\
\varepsilon\frac{\dashint_{F_j}E_{\bm{\theta}}}{\dashint_{T_j}E_{\bm{\theta}}} &= B_3^\varepsilon(\sigma_{pq},\sigma_{p\widehat{jpq}},-\sigma_p) \quad \text{ for any } 1\leq p < q \leq 4,~ p,q\neq j, 
\end{aligned}
$$
where we recall that $\sigma_s = \bm{\beta}\cdot \bm{l}_s$, $\sigma_{st} = {\bm{\beta}}\cdot \bm{E}_{st} = \sigma_t - \sigma_s$ from \eqref{eq:sigma-diff}.
Thus, taking any $1 \leq p < q \leq 4$ and $p,q \neq j$, the equation \eqref{eq:subdiv} can be rewritten as 
\begin{equation}\label{eq:algdiv}
	\begin{aligned}
	j_1^F(\bm{x})|T_j| &- \sum_{1 \leq s < t \leq 4, s,t\neq j} (\pm)_{jst}  B_{3}^\varepsilon(\sigma_s, \sigma_t, \sigma_{\widehat{jst}}) \bm{\phi}_1^F(\bm{x}) \cdot \frac{\bm{l}_s \times \bm{l}_t}{2} \\
	&=  B_3^\varepsilon(\sigma_{pq},\sigma_{p\widehat{jpq}},-\sigma_p)\delta_{1j}, \quad \forall 1\leq j \leq 4,
		\end{aligned}
\end{equation}
which defines a linear algebraic system for the unknowns ${\bm{\phi}}_{1}^F(\bm{x})$ and ${j}_{1}^F(\bm{x})$. Again, these extensions are solely employed for the purpose of deriving the algebraic equation.

It is evident that both the variables and the number of equations in \eqref{eq:algdiv} amount to four. The specific expressions depend on the labeling of vertices. In order to write a more concrete matrix form, we adopt the labeling shown in Figure \ref{fig:sub-tetrahedron}, and express $|T_j|$ using mixed product. This yields the following problem.

\begin{problem}[$\mathcal{L}^2$-spline]\label{problem:div}
	Find $\bm{\phi}_1^F(\bm{x})$ and $j_1^F(\bm{x})$ such that for all $\bm{x}\in \bar{T}$
	\begin{equation}
		\label{eq:sysdiv}
		D^F(\bm{x})\left(\begin{matrix}
			j_1^F(\bm{x})\\\bm{\phi}_1^F(\bm{x})
		\end{matrix}\right)= B_3^\varepsilon(\sigma_{23},\sigma_{24},-\sigma_2)\bm{e}_1^F,
	\end{equation}
	where $\bm{e}_1^F=(1,0,0,0)$ and $D^F(\bm{x})$ is a $4\times 4$ matrix defined by
	\begin{equation}
		\begin{aligned}
	&\quad D^F(\bm{x}) :=\\& \left(\begin{matrix}
			\frac{\bm{l}_2\times\bm{l}_3\cdot\bm{l}_4}{6} &(B_3^\varepsilon(\sigma_2,\sigma_3,\sigma_4)\frac{\bm{l}_2\times\bm{l}_3}{2}+ B_3^\varepsilon(\sigma_3,\sigma_4,\sigma_2)\frac{\bm{l}_3\times\bm{l}_4}{2} + B_3^\varepsilon(\sigma_4,\sigma_2,\sigma_3)\frac{\bm{l}_4\times\bm{l}_2}{2})^\top\\
			\frac{\bm{l}_3\times\bm{l}_1\cdot\bm{l}_4}{6} &( B_3^\varepsilon(\sigma_3,\sigma_1,\sigma_4)\frac{\bm{l}_3\times\bm{l}_1}{2}+ B_3^\varepsilon(\sigma_1,\sigma_4,\sigma_3)\frac{\bm{l}_1\times\bm{l}_4}{2}+ B_3^\varepsilon(\sigma_4,\sigma_3,\sigma_1)\frac{\bm{l}_4\times\bm{l}_3}{2})^\top\\
			\frac{\bm{l}_4\times\bm{l}_1\cdot\bm{l}_2}{6} &(B_3^\varepsilon(\sigma_4,\sigma_1,\sigma_2)\frac{\bm{l}_4\times\bm{l}_1}{2}+ B_3^\varepsilon(\sigma_1,\sigma_2,\sigma_4)\frac{\bm{l}_1\times\bm{l}_2}{2}+ B_3^\varepsilon(\sigma_2,\sigma_4,\sigma_1)\frac{\bm{l}_2\times\bm{l}_4)}{2})^\top\\
			\frac{\bm{l}_1\times\bm{l}_3\cdot\bm{l}_2}{6} &( B_3^\varepsilon(\sigma_1,\sigma_3,\sigma_2)\frac{\bm{l}_1\times\bm{l}_3}{2}+ B_3^\varepsilon(\sigma_3,\sigma_2,\sigma_1)\frac{\bm{l}_3\times\bm{l}_2}{2}+ B_3^\varepsilon(\sigma_2,\sigma_1,\sigma_3)\frac{\bm{l}_2\times\bm{l}_1}{2})^\top
		\end{matrix}\right).
				\end{aligned}
	\end{equation}
\end{problem}

The solution to Problem~\ref{problem:div} (whose later proof establishes its uniqueness or existence) defines the point values of the spline basis function $\bm{\phi}_{1}^F$ and the auxiliary flux ${j}_{1}^F$ at $\bm{x} \in \bar{T}$. A similar process is carried out for $\bm{\phi}_{i}^F$ and ${j}_{i}^F~(1\leq i \leq 4)$ associated with the facet $F_{i}$. This concludes the construction of the shape function space given by \eqref{eq:div-shape}.

\begin{remark}[variable $\bm{\beta}$ and local smoothness of $\mathcal{S}_{1^-}^2(T)$]
	Similarly, given $\bm{x}\in\bar{T}$, we consider using the  value $\bm{\beta}(\bm{x})$ to evaluate the corresponding $\sigma_s$ and $\sigma_{st}$ for constructing the matrix $D^F(\bm{x})$ and right side term. In addition, if $\bm{\beta}$ is $C^1$ in the element $T$, then spline basis is $C^1$ within the element.
\end{remark}

\begin{remark}[approximation property of $\mathcal{S}_{1^-}^2$] \label{rm:approx-div}
	It is straightforward that any constant vector $\bm{c}$ is contained in $\mathcal{S}_{1^-}^2(T)$ since the constant vector $\bm{c}$ and the corresponding auxiliary flux $\bm{\beta}(\bm{x})\cdot\bm{c}$ satisfy all the Strategies in the construction.
\end{remark}

\subsection{Well-posedness of $\mathcal{S}_{1^-}^2(T)$}
We will show that Problem~\ref{problem:div} is well-posed, i.e., uniquely solvable for all $\bm{x}\in\bar{T}$.
\begin{lemma}[well-posedness of Problem \ref{problem:div}] \label{lm:divunisolve}
For any $\bm{x}\in \bar{T}$, there exists a unique solution to Problem~\ref{problem:div}.
\end{lemma}
\begin{proof}
Without loss of generalization, we assume that $\lambda_1(\bm{x})> 0$. Using $\sum_{i=1}^4\lambda_i\equiv 1$ and thus $\sum_{i=1}^{4}\lambda_i\bm{l}_1=\bm{0}$, we can eliminate the lower-left block of matrix $D^F(\bm{x})$ to obtain 
	$$
	{\lambda_1^3}\det(D^F(\bm{x}))=\left|\begin{matrix}
		C &B\\
		{0} & {A}
	\end{matrix}\right|,
	$$
	where $C=\frac{\bm{l}_2\times\bm{l}_3\cdot\bm{l}_4}{6}
	=\lambda_1|T|>0$, $B=( B_3^\varepsilon(\sigma_2,\sigma_3,\sigma_4)\frac{\bm{l}_2\times\bm{l}_3}{2}+ B_3^\varepsilon(\sigma_3,\sigma_4,\sigma_2)\frac{\bm{l}_3\times\bm{l}_4}{2} + B_3^\varepsilon(\sigma_4,\sigma_2,\sigma_3)\frac{\bm{l}_4\times\bm{l}_2}{2})^\top$, ${A}=\lambda_1A_0-\Lambda B \in \mathbb{R}^{3\times 3}$ with $\Lambda=(\lambda_2,\lambda_3,\lambda_4)^\top$ and
	$$
	A_0 = \left(\begin{matrix}
		(B_3^\varepsilon(\sigma_3,\sigma_1,\sigma_4)\frac{\bm{l}_3\times\bm{l}_1}{2} + B_3^\varepsilon (\sigma_1,\sigma_4,\sigma_3)\frac{\bm{l}_1\times\bm{l}_4}{2} + B_3^\varepsilon(\sigma_4,\sigma_3,\sigma_1)\frac{\bm{l}_4\times\bm{l}_3}{2})^\top\\
		 (B_3^\varepsilon(\sigma_4,\sigma_1,\sigma_2)\frac{\bm{l}_4\times\bm{l}_1}{2} + B_3^\varepsilon(\sigma_1,\sigma_2,\sigma_4)\frac{\bm{l}_1\times\bm{l}_2}{2} + B_3^\varepsilon(\sigma_2,\sigma_4,\sigma_1)\frac{\bm{l}_2\times\bm{l}_4}{2})^\top\\
		(B_3^\varepsilon(\sigma_1,\sigma_3,\sigma_2)\frac{\bm{l}_1\times\bm{l}_3}{2} + B_3^\varepsilon(\sigma_3,\sigma_2,\sigma_1)\frac{\bm{l}_3\times\bm{l}_2}{2} + B_3^\varepsilon(\sigma_2,\sigma_1,\sigma_3)\frac{\bm{l}_2\times\bm{l}_1}{2})^\top
	\end{matrix}\right).
	$$ 
	Next, we need to demonstrate the non-singularity of ${A}$. By substituting $\lambda_1\bm{l}_1=-\sum_{i=2}^{4}\lambda_i\bm{l}_i$, we derive
	$$
	A = {\tilde{A}}
	\left(\begin{matrix}
		 (\bm{l}_2\times \bm{l}_3)^\top\\ (\bm{l}_3\times \bm{l}_4)^\top\\(\bm{l}_4\times \bm{l}_2)^\top
	\end{matrix}\right),
	$$
	where
	$$\tiny {\tilde{A}}=-\left(\begin{matrix}
		-\lambda_2B^{314} &\lambda_4B^{314}+\lambda_1B^{431}+\lambda_3B^{143} & -\lambda_2B^{143}\\
		-\lambda_3B^{124} & -\lambda_3B^{412} & \lambda_2B^{412}+\lambda_4B^{124}+\lambda_1B^{241}\\
		\lambda_2B^{132}+\lambda_1B^{321}+\lambda_3B^{213} & -\lambda_4B^{132} &-\lambda_4B^{213}
	\end{matrix}\right)-\Lambda \tilde{B},
	$$
	with $B^{ijk}$ denotes $B_3^\varepsilon(\sigma_i,\sigma_j,\sigma_k)$, and $\tilde{B}=(B^{234},B^{342},B^{423})$. Observe that $\bm{l}_2\times \bm{l}_3,\bm{l}_3\times\bm{l}_4,\bm{l}_4\times\bm{l}_2$ are linearly independent when $\lambda_1>0$. To establish the non-singularity of ${A}$, it suffices to demonstrate the non-singularity of $\tilde{{A}}$.

	A direct computation shows that the determinant of $\tilde{A}$ satisfies
	$$
	\begin{aligned}
		-\det(\tilde{A})&=\sum_{l=1}^4
		\sum_{((i,j,k),(r,s,t),(m,n,o))\in S_l^F\backslash Q_l}
		 B^{ijk}B^{rst}B^{mno}\lambda_k\lambda_t\lambda_o,
	\end{aligned}
	$$
	where 
	$$
	\begin{aligned}
		S^F_l=\{\big((i,j,k),(r,s,t),(m,n,o)\big)\in O_x \times O_y\times O_z|&1\le x<y<z\le 4,\\ &~l\notin \{x,y,z\},l\in\{k,t,o\}\},
	\end{aligned}
	$$
	with
	$$
	\small
	\begin{aligned}
		&O_l=\{(i,j,k)|1\le i,j,k\le 4,i< j,i\neq k,j\neq k \text{ and }l\notin \{i,j,k\}\},\\
		&Q_1=\{\big((3,4,1),(1,2,4),(1,2,3)\big),\big((1,3,4),(2,4,1),(1,3,2)\big),\big((1,4,3),(1,4,2),(2,3,1)\big)\},\\
		&Q_2=\{\big((3,4,2),(1,2,3),(1,2,4)\big),\big((2,4,3),(1,3,2),(2,4,1)\big),\big((2,3,4),(2,3,1),(1,4,2)\big)\},\\
		&Q_3=\{\big((2,4,3),(1,3,4),(1,3,2)\big),\big((2,3,4),(1,4,3),(2,3,1)\big),\big((3,4,2),(3,4,1),(1,2,3)\big)\},\\
		&Q_4=\{\big((1,2,4),(3,4,2),(3,4,1)\big),\big((2,4,1),(2,4,3),(1,3,4)\big),\big((1,4,2),(2,3,4),(1,4,3)\big)\}.
	\end{aligned}
	$$ 
	Here, we leverage the symmetry property $B^{ijk}=B^{jik}$, which arises from the definition of the 3D-Bernoulli function \eqref{eq:bernoulli3d}. Combining this with the positivity of the Bernoulli function and the assumption $\lambda_1>0$, it becomes evident that ${\tilde{A}}$ is non-singular. Consequently, both the matrix ${A}$ and $D^F(\bm{x})$ are also non-singular.
	\end{proof}

We provide visualizations of the $\mathcal{S}_{1^-}^2(T)$ basis function on a reference element. In Figure \ref{fig:divbasis1}, we present plots illustrating the second component of the $\bm{H}({\rm div})$ exponentially-fitted basis function, accompanied by a vector plot for different diffusion coefficients: $\varepsilon=1$ and $\varepsilon=0.01$, in the scenario where $\bm{\beta}=(1,2,3)^T$. Notably, the $\bm{H}({\rm div})$ exponentially-fitted basis function also showcases a behavior akin to a linear-like basis for larger $\varepsilon$, while adopting an exponential-like structure for smaller $\varepsilon$. This once again  indicates the flexibility of the function in accommodating both convection-dominated and diffusion-dominated cases. 

\begin{figure}[!htbp]
	\centering
	\subfloat[contour of the second component $(\varepsilon=1)$]{\includegraphics[width=.4\textwidth]{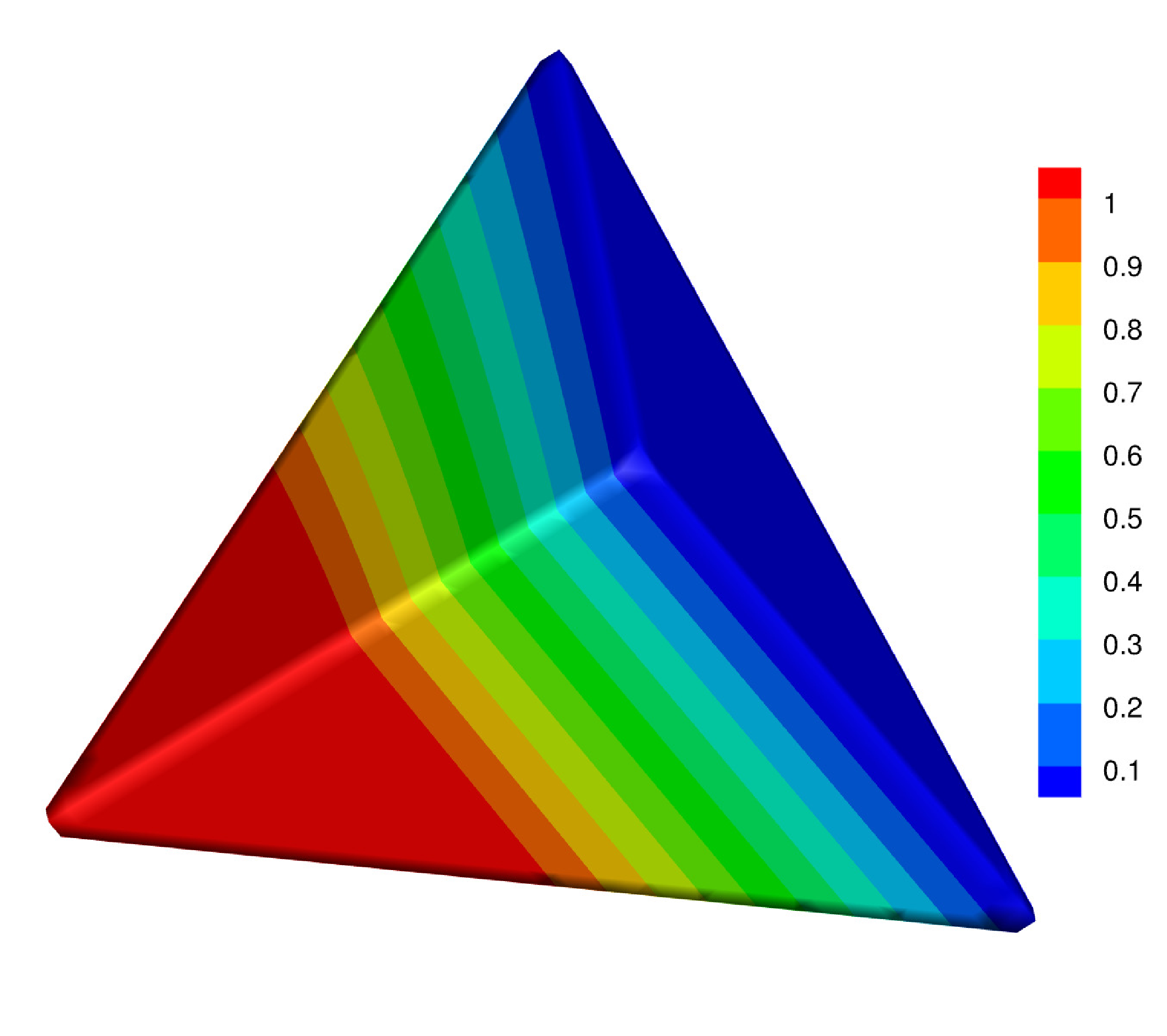}}
	\subfloat[vector plot $(\varepsilon=1)$]{\includegraphics[width=.42\textwidth]{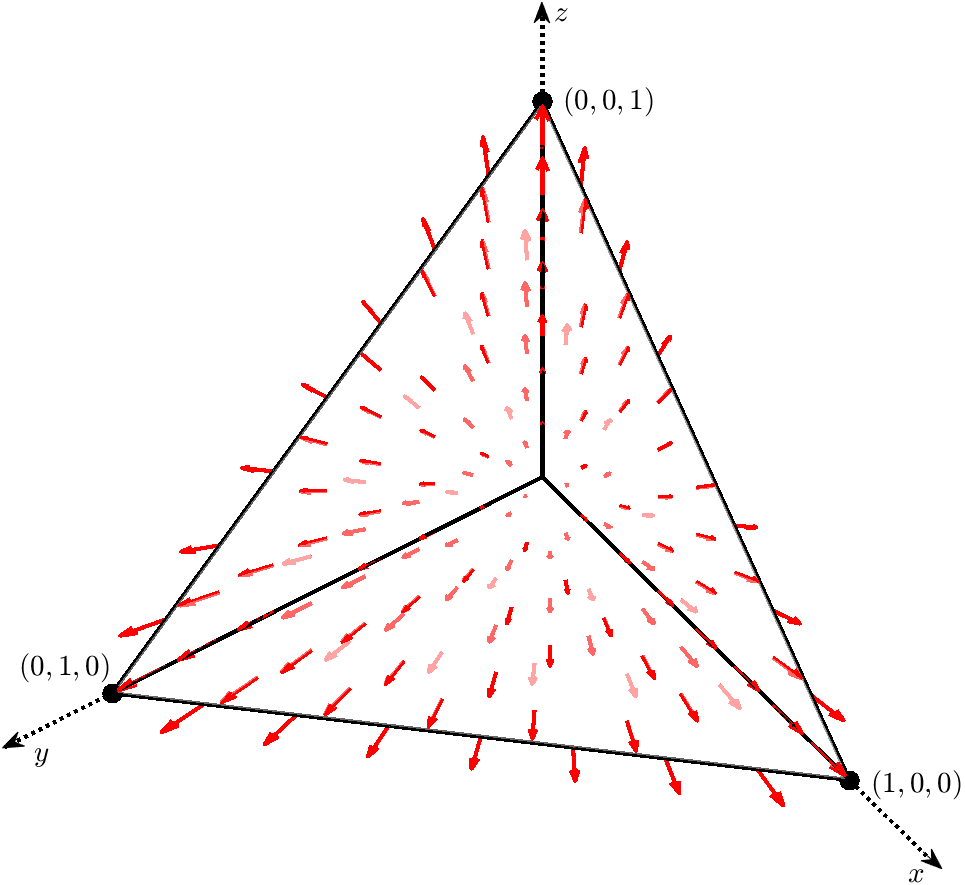}}\\
	\subfloat[contour of the second component $(\varepsilon=0.01)$]{\includegraphics[width=.42\textwidth]{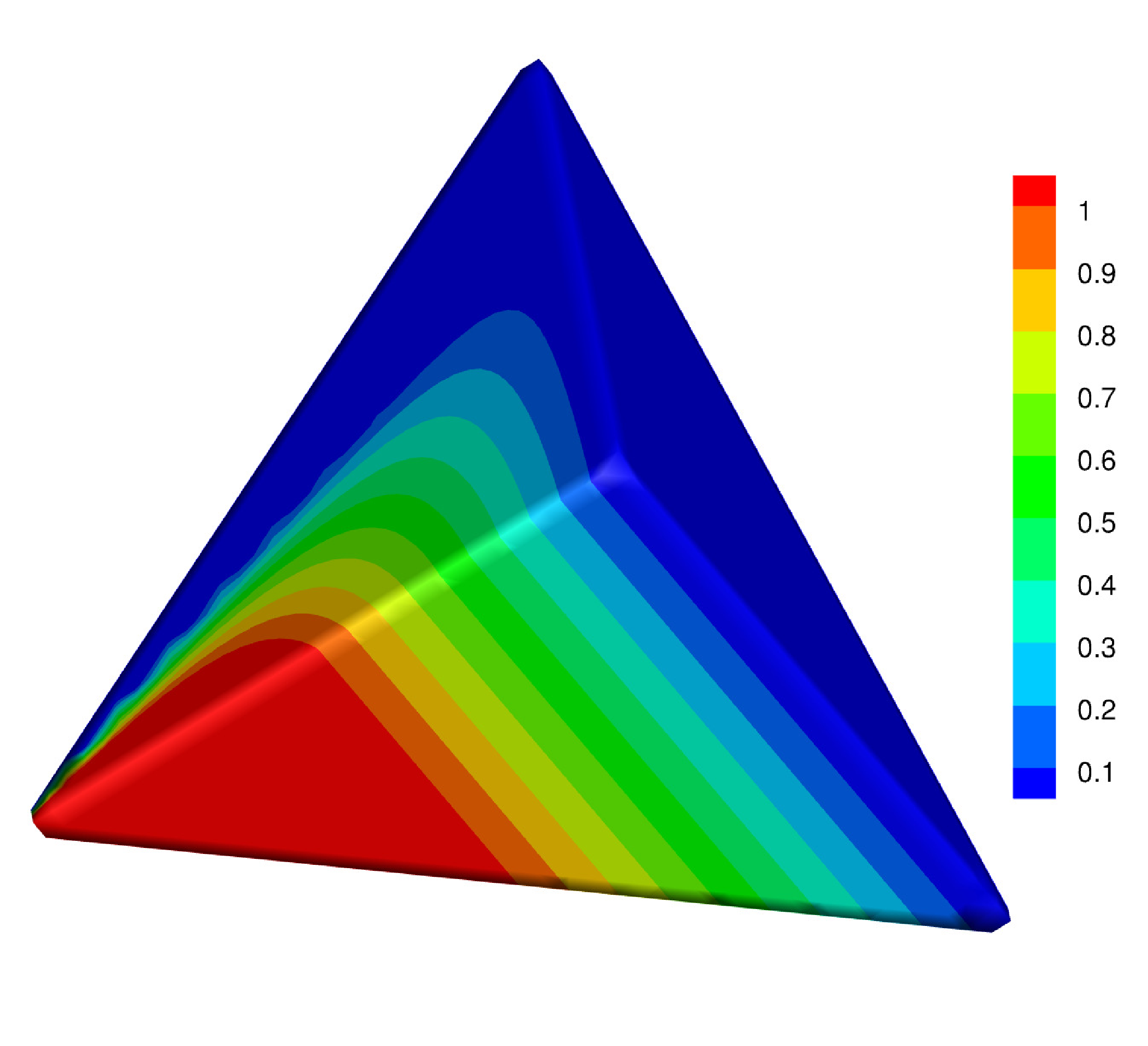}}
	\subfloat[vector plot $(\varepsilon=0.01)$]{\includegraphics[width=.4\textwidth]{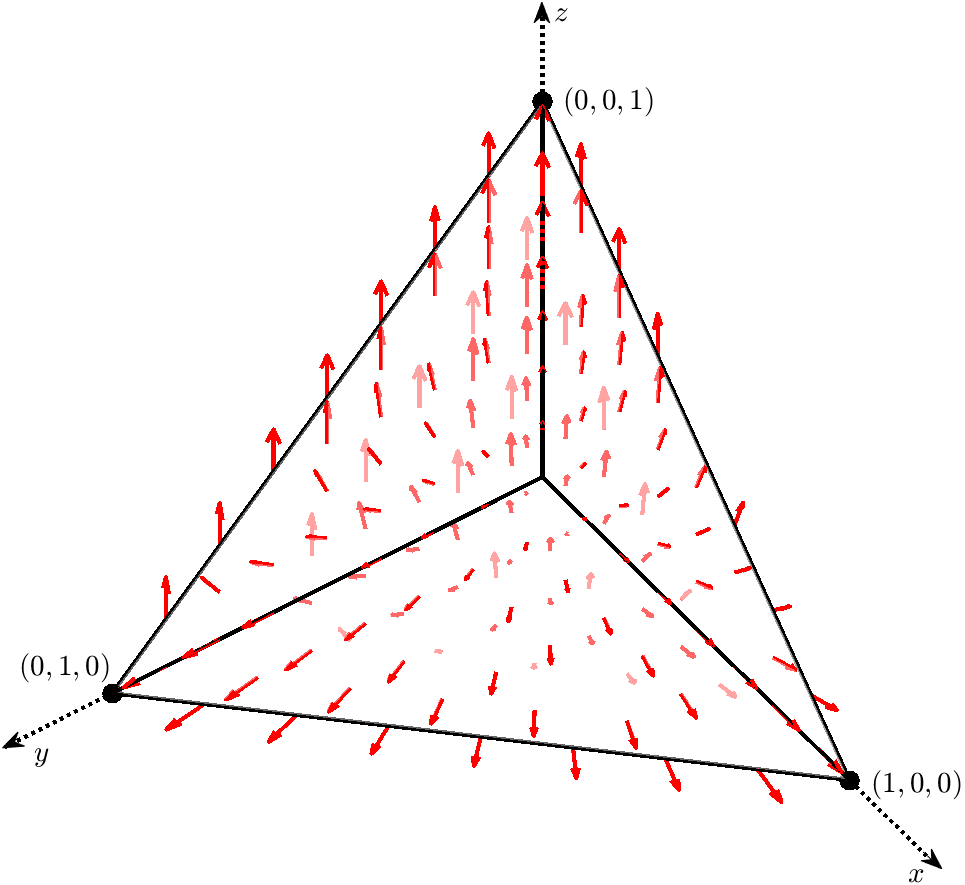}}
	\caption{Visualization of $\bm{H}({\rm div})$ exponentially-fitted basis on a reference element: contour of the second component (left) and vector plot (right) for cases dominated by diffusion (top) and convection (bottom).}\label{fig:divbasis1}
\end{figure}

\subsection{Properties of $\mathcal{S}_{1^-}^2$}
For Problem~\ref{problem:div}, if we multiply the row associated with $F_j$ by $\dashint_{T_j} E_{\bm{\theta}}$ respectively, then use the geometric interpretation of 3D-Bernoulli function and sum up all the rows together, we obtain the following corollary.
\begin{corollary}[constant flux of $\mathcal{S}_{1^-}^2$]\label{coro:divflux}
It holds that for any $\bm{x} \in \bar{T}$,
$$
j_i^F(\bm{x})\int_TE_{\bm{\theta}} = \varepsilon \dashint_{F_i} E_{\bm{\theta}}, \quad 1\leq i \leq 4,
$$
which implies $j_i^F(\bm{x})$ is constant in $T$ when $\bm{\beta}$ is constant.
\end{corollary}

Next, we establish an identical set of unisolvent degrees of freedom (DOFs) for $\mathcal{S}_{1^-}^2(T)$, matching those of the lowest-order Raviart-Thomas element. 
\begin{figure}[!htbp]
	\centering
	\includegraphics[width=.3\textwidth]{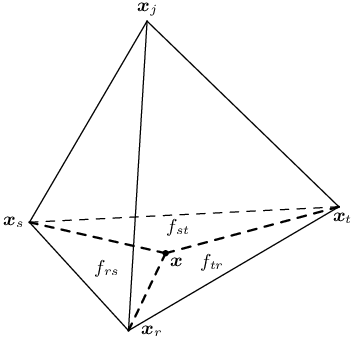}
	\caption{Visual representation in the proof of $\mathrm{div}$ DOFs.} \label{fig:div-dof}
	\label{fig:geo-notation}
\end{figure}

\begin{lemma}[DOFs of $\mathcal{S}_{1^-}^2(T)$]\label{lm:divdof}
	Let $T\in\mathcal{T}_h$, it holds that
	\begin{equation}\label{eq:divdof}
		{\bm{\phi}}_{i}^F(\bm{x}) \cdot {\bm{n}}_{j}= \frac{\delta_{ij}}{|F_{j}|}, \quad \forall \bm{x} \in F_{j},~ 1\leq j \leq 4,
	\end{equation}
	without necessitating that $\bm{\beta}$ is a constant vector.
\end{lemma}
\begin{proof}
As shown in Figure \ref{fig:div-dof}, we consider the facet $F_j$ formed by the vertices $\bm{x}_s$, $\bm{x}_t$, and $\bm{x}_r$, with an orientation aligned with the outward normal of $F_j$, i.e., $(\pm)_{str} = 1$. For any $\bm{x} \in F_j$, equation \eqref{eq:algdiv} for basis function $\bm{\phi}^F_i(\bm{x})$ turns out to be
	$$
	\begin{aligned}
	 \Big(B_3^\varepsilon(\sigma_s,\sigma_t,\sigma_r) \frac{\bm{l}_s \times \bm{l}_t}{2} 
	     &+ B_3^\varepsilon(\sigma_t,\sigma_r,\sigma_s) \frac{\bm{l}_t \times \bm{l}_r}{2} 
	     + B_3^\varepsilon(\sigma_r,\sigma_s,\sigma_t)\frac{\bm{l}_r\times\bm{l}_s}{2} \Big) \cdot\bm{\phi}_i^F(\bm{x}) \\
	 &= \delta_{ij}B_3^\varepsilon (\sigma_{st},\sigma_{sr},-\sigma_s).
	\end{aligned}
	$$
	Note that on the facet $F_j$, with the convention $f_{st} = f_{ts}$, it holds that 
	$$
	\bm{l}_s\times\bm{l}_t = 2|f_{st}|\bm{n}_j, \quad \bm{l}_t\times\bm{l}_r = 2|f_{tr}|\bm{n}_j, \quad \bm{l}_r\times\bm{l}_s = 2|f_{rs}|\bm{n}_j.
	$$
Then the above equation is
	$$
	\begin{aligned}
	\big( B_3^\varepsilon (\sigma_s,\sigma_t, \sigma_r)|f_{st}| &+ B_3^\varepsilon(\sigma_t,\sigma_r,\sigma_s)|f_{tr}|+B_3^{\varepsilon} (\sigma_r,\sigma_s,\sigma_t)|f_{rs}| \big)\bm{n}_j\cdot\bm{\phi}^F_i(\bm{x}) \\&= \delta_{ij}B_3^\varepsilon (\sigma_{st},\sigma_{sr},-\sigma_s).
	\end{aligned}
	$$
	Again, using the geometric interpretation of 3D-Bernoulli function, we have 
	$$
	\begin{aligned}
	& B_3^\varepsilon (\sigma_s,\sigma_t, \sigma_r)|f_{st}| + B_3^\varepsilon(\sigma_t,\sigma_r,\sigma_s)|f_{tr}|+B_3^{\varepsilon} (\sigma_r,\sigma_s,\sigma_t)|f_{rs}| \\
	= &~ \varepsilon \frac{\int_{f_{st}} E_{\bm{\theta}(\bm{x})} + \int_{f_{tr}} E_{\bm{\theta}(\bm{x})} + \int_{f_{rs}} E_{\bm{\theta}(\bm{x})}}{\dashint_{T_j} E_{\bm{\theta}(\bm{x})}} = \varepsilon \frac{\int_{F_j} E_{\bm{\theta}(\bm{x})}}{\dashint_{T_j} E_{\bm{\theta}(\bm{x})}} = |F_j|B_3^\varepsilon (\sigma_{st},\sigma_{sr},-\sigma_s),
        \end{aligned}
	$$
	the desired result \eqref{eq:divdof} then follows thanks to the positivity of Bernoulli function.
\end{proof}

So far, we have established the shape function space of $\mathcal{S}_{1^-}^2$ and the unisolvent DOFs, thus concluding the construction of the finite element space. For a given mesh $\mathcal{T}_h$, its finite element space is denoted by
\begin{equation} \label{eq:div-space}
\mathcal{S}_{1^-}^2(\mathcal{T}_h) := \{\bm{v}|_T \in \mathcal{S}_{1^-}^2(T),\forall T \in \mathcal{T}_h:~ \bm{v}\cdot \bm{n} \text{ is continuous across all facets} \}.
\end{equation}
From the above lemma, if $\bm{\beta}$ is tangential continuous across any facet $F \in \mathcal{F}_h^0$ and piecewise $C^1$, then the function space $\mathcal{S}_{1^-}^2(\mathcal{T}_h)$ is clearly $\bm{H}({\rm div})$-conforming.


\section{A commutative diagram} \label{sec:diagram}
In this section, we will introduce weighted interpolation operators $\Pi_{\bm{\theta},T}^k$ onto $\mathcal{S}_{1^-}^k(T)$ $(k=0,1,2,3)$ and discrete convection-diffusion differential operators $J_{\bm{\theta},T}^k$ $(k=0,1,2)$ on $\mathcal{S}_{1^-}^k(T)$ with $\bm{\theta}=\bm{\beta}/\varepsilon$. Here, in light of Corollary \ref{coro:divflux} (constant flux of $\mathcal{S}_{1^-}^2$), the space $\mathcal{S}_{1^-}^3(\mathcal{T}_h)$ is nothing but the piecewise constant function space with local basis function $\phi^T = \frac{1}{|T|}$.

Using these operators and function spaces, we intend to give the following local commutative diagram when $\bm{\beta}$ is constant.
\begin{equation} \label{eq:diagram}
	\begin{tikzcd}
		C^\infty(T) \arrow[r, "\bm{J}_{\bm{\theta}}^0"] 
		\arrow[d, swap, "{\Pi}_{\bm{\theta},T}^0"] 
		& C^\infty(T; \mathbb{R}^3) \arrow[r, "\bm{J}_{\bm{\theta}}^1"] 
		\arrow[d, swap, "\bm{\Pi}_{\bm{\theta},T}^1"] 
		& C^\infty(T;\mathbb{R}^3) \arrow[r,"J_{\bm{\theta}}^2"] 
		\arrow[d, swap, "\bm{\Pi}_{\bm{\theta},T}^2"] 
		& C^\infty(T)
		\arrow[d, swap, "{\Pi}_{\bm{\theta},T}^3"] \\ 
		\mathcal{S}_{1^-}^0(T) \arrow[r, "\bm{J}_{\bm{\theta},T}^0"] 
		& \mathcal{S}_{1^-}^1(T) \arrow[r, "\bm{J}_{\bm{\theta},T}^1"] 
		& \mathcal{S}_{1^-}^2(T) \arrow[r, "J_{\bm{\theta},T}^2"] 
		& \mathcal{S}_{1^-}^3(T) 
	\end{tikzcd}
\end{equation}

\subsection{Discrete flux and weighted interpolation}

Based on the previous introduction of finite elements, the degrees of freedom of $\mathcal{S}_{1^-}^k$ align seamlessly with those of $\mathcal{P}_{1^-}^k$. The corresponding basis functions are given in the construction of $\mathcal{S}_{1^-}^k$ in a pointwise fashion. To encapsulate, the local basis functions of $\mathcal{S}_{1^-}^k(T)$, associated with the sub-simplices of $T$, are represented by $\phi^V_i,\bm{\phi}^E_{ij},\bm{\phi}^F_{i}$, and $\phi^T$ respectively. The local degrees of freedom are denoted as follows:
$$ 
\begin{aligned}
	l_{j}^0(\phi_{i}^V) := \phi_{i}^V(x_j) = \delta_{ij}, &
	\qquad l_{st}^1(\bm{\phi}^E_{ij}) = \int_{E_{st}} \bm{\phi}_{ij}^E \cdot \bm{t}_{st} =
	\delta_{is,jt}, \\ 
	l_{j}^2(\bm{\phi}_{i}^F) = \int_{F_j} \bm{\phi}_{i}^F \cdot \bm{n}_{j} =
	\delta_{ij}, & 
	\qquad l^3_{T'}(\phi^{T}) = \int_{T'} \phi^{T} =\delta_{TT'}. 
\end{aligned}
$$ 

A distinctive feature in the construction of $\mathcal{S}_{1^-}^k$ is that the algebraic system simultaneously incorporates the values of the basis functions and their corresponding fluxes; Refer to Problems \ref{problem:grad}, \ref{problem:curl}, and \ref{problem:div} for cases $k=0,1,2$ respectively. In essence, the definition of discrete flux   below is intrinsic.
\begin{definition}[discrete flux]
The local discrete flux operators $J_{\bm{\theta},T}^k$ on $\mathcal{S}_{1^-}^k(T)$ are defined by
	\begin{equation}
		\label{eq:discretediff}
		\begin{aligned}
			\bm{J}_{\bm{\theta},T}^0v_h :& = \sum_{x_i\in\mathcal{V}_T} l_i^0(v_h) \bm{j}_i^V,\\
			\bm{J}_{\bm{\theta},T}^1\bm{v}_h :& =\sum_{E_{ij}\in\mathcal{E}_T} l_{ij}^1(\bm{v}_h) \bm{j}_{ij}^E,\\
			J_{\bm{\theta},T}^2\bm{v}_h :&= \sum_{F_i\in\mathcal{F}_T} l_i^2(\bm{v}_h){j}_i^F.
		\end{aligned}
	\end{equation}
The global discrete flux on the mesh $\mathcal{T}_h$ is defined by $(J_{\bm{\theta},h}^k v_h)|_T := J_{\bm{\theta},T}^k (v_h|_T)$, for all $T \in \mathcal{T}_h$.
\end{definition}
	
\begin{remark}[consistency when $\bm{\beta} = \bm{0}$]
Indeed, the discrete flux is not a direct analytical representation of the convection-diffusion derivative of the function within $\mathcal{S}_{1^-}^k$. However, it can be conceptualized as a form of weak convection-diffusion. Notably, in the specific case of $\bm{\beta}=0$, where $\mathcal{S}_{1^-}^k$ aligns with $\mathcal{P}_{1^-}^k$, the discrete flux $J_{\bm{0},T}^k v_h$ precisely corresponds to $\varepsilon d^k v_h$, as can be easily seen from \eqref{eq:grad-const}, \eqref{eq:curl-const}, \eqref{eq:div-const} for $k=0,1,2$ respectively.
\end{remark}

Next, we introduce a class of weighted interpolation operators.

\begin{definition}[weighted interpolation]
The local weighted interpolation operators $\Pi_{\bm{\theta},T}^k$ onto $\mathcal{S}_{1^-}^k(T)$ are defined by
\begin{equation} \label{eq:weightinterpolate}
	\begin{aligned}
		& \Pi_{\bm{\theta},T}^0v := \sum_{x_i\in\mathcal{V}_T}
 		\frac{l_i^0(E_{\bm{\theta}}v)}{E_{\bm{\theta}}(\bm{x}_i)}\phi_i^V, \quad
		 &\bm{\Pi}_{\bm{\theta},T}^1\bm{v} := \sum_{E_{ij}\in\mathcal{E}_T}
		\frac{l_{ij}^1(E_{\bm{\theta}}\bm{v})}{\dashint_{E_{ij}}E_{\bm{\theta}}}\bm{\phi}_{ij}^E,\\
		 &\bm{\Pi}_{\bm{\theta},T}^2\bm{v}:=\sum_{F_i\in\mathcal{F}_T}
 		\frac{l_i^2(E_{\bm{\theta}}\bm{v})}{\dashint_{F_i}E_{\bm{\theta}}}\bm{\phi}_i^F, \quad
		& \Pi_{\bm{\theta},T}^3v:= \frac{l_T^3(E_{\bm{\theta}}v)}{\dashint_T E_{\bm{\theta}}}\phi^T.
 			\end{aligned}
		\end{equation}
The global weighted interpolation on the mesh $\mathcal{T}_h$ is defined by $(\Pi_{\bm{\theta},h}^k v)|_T := \Pi_{\bm{\theta},T}^k (v|_T)$, for all $T \in \mathcal{T}_h$.
	\end{definition}
	These interpolations can be viewed as employing the integration weights $E_{\bm{\theta}}$ at each degree of freedom.
	Using the result in Lemma \ref{lemma:curldof} (DOFs of $\mathcal{S}_{1^-}^1(T)$) and
Lemma \ref{lm:divdof} (DOFs of $\mathcal{S}_{1^-}^2(T)$), it is straightforward that $\Pi_{\bm{\theta},T}^k|_{\mathcal{S}_{1^-}^k(T)}$ is the identity operator.

\subsection{Kernel preservation}
In this subsection, our aim is to demonstrate the preservation of the interpolation of the $\ker(J_{\bm{\theta}}^k)$ for the aforementioned discrete flux operator $J_{\bm{\theta},T}^k$ when $\bm{\theta}$ is constant within $T$. By employing the identity \eqref{eq:identity} which states that
$$
{J}^k_{\bm{\theta}}=\varepsilon E_{\bm{\theta}}^{-1}d^kE_{\bm{\theta}},
$$
we can deduce that the kernel of ${J}^{k}_{\bm{\theta}}$ can be expressed as $E_{\bm{\theta}}^{-1}\ker(d^k)$, while the characterization of $\ker(d^k)$ has been extensively studied.

\begin{lemma}[kernel preservation] \label{lm:kernel}
Assume $\bm{\theta}$ is constant in $T$. For any $w \in \ker({J}_{\bm{\theta}}^k)$, it holds that $\Pi_{\bm{\theta},T}^k{w} \in \ker(J_{\bm{\theta},T}^k)$, for $k=0,1,2$.
\end{lemma}

%

\begin{proof}[proof of Lemma \ref{lm:kernel} for $k=0$] Given that $w \in \ker(\bm{J}_{\bm{\theta}}^0)$, it follows that $w = E_{\bm{\theta}}^{-1} q$ for a constant $q$. Let $\bm{j}_w^V := \bm{J}_{\bm{\theta},T}^0 \Pi_{\bm{\theta},T}^0 w$, which, as implied by Problem \ref{problem:grad} and \eqref{eq:weightinterpolate}, satisfies the equation:
	$$
	D^V(\bm{x})
	\begin{pmatrix}
		\bm{j}^V_{{w}}(\bm{x})\\\Pi_{\bm{\theta},T}^0{w}(\bm{x}) 
	\end{pmatrix}
	= \begin{pmatrix}
		w(\bm{x}_1)B_1^\varepsilon (-\sigma_1)\\
		w(\bm{x}_2)B_1^\varepsilon (-\sigma_2)\\
		w(\bm{x}_3)B_1^\varepsilon (-\sigma_3)\\
		w(\bm{x}_4)B_1^\varepsilon (-\sigma_4)
	\end{pmatrix}.
	$$
	Considering the geometric interpretation of the 1D-Bernoulli function, it is shown that
	$$ 
	w(\bm{x}_i) B_1^\varepsilon(-\sigma_i) = w(\bm{x}_i) \varepsilon \frac{E_{\bm{\theta}}(\bm{x}_i)}{\dashint_{l_i} E_{\bm{\theta}}} = 
	\frac{\varepsilon q}{\dashint_{l_i} E_{\bm{\theta}}}.
	$$ 
	As a result, we can multiply each row corresponding to vertex $\bm{x}_{i}$ by $\dashint_{l_{i}}E_{\bm{\theta}}$, leading to the following linear system:
	$$
	\begin{pmatrix}
		\bm{l}_1^\top \dashint_{l_1}E_{\bm{\theta}} & \varepsilon E_{\bm{\theta}}(\bm{x})\\
		\bm{l}_2^\top \dashint_{l_2}E_{\bm{\theta}}& \varepsilon E_{\bm{\theta}}(\bm{x})\\
		\bm{l}_3^\top \dashint_{l_3}E_{\bm{\theta}}& \varepsilon E_{\bm{\theta}}(\bm{x})\\
		\bm{l}_4^\top \dashint_{l_4}E_{\bm{\theta}}& \varepsilon E_{\bm{\theta}}(\bm{x})\\
	\end{pmatrix}
	\begin{pmatrix}
	\bm{j}^V_{{w}}(\bm{x})\\\Pi_{\bm{\theta},T}^0{w}(\bm{x}) 
	\end{pmatrix}
	=  
	\begin{pmatrix}
	\varepsilon q \\ 
	\varepsilon q \\ 
	\varepsilon q \\ 
	\varepsilon q 
	\end{pmatrix}.
	$$
	This implies that $\bm{j}_w^V(\bm{x}) \cdot \bm{l}_i \dashint_{l_i} E_{\bm{\theta}}~(1\leq i \leq 4)$ are the same. Without loss of generality, we may assume $\lambda_1(\bm{x})>0$. Using the identity \eqref{eq:identityli} that $\sum_{i=1}^{4}\lambda_i\bm{l}_i=\bm{0}$ and substituting $\bm{l}_1$ by $\bm{l}_2,\bm{l}_3$ and $\bm{l}_4$, we therefore have
	$$
	\lambda_1^{-1}\mathcal{K}^V\left(\begin{matrix}
	     \bm{j}^V_w(\bm{x}) \cdot \bm{l}_2 \\
	     \bm{j}^V_w(\bm{x}) \cdot \bm{l}_3\\
	     \bm{j}^V_w(\bm{x}) \cdot \bm{l}_4
	\end{matrix}\right)=\bm{0},
	$$
	where 
	$$
	\mathcal{K}^V:=\left(\begin{matrix}
		(\lambda_1\dashint_{l_{2}}+\lambda_2\dashint_{l_{1}})E_{\bm{\theta}}& \lambda_3\dashint_{l_{1}}E_{\bm{\theta}}&\lambda_4\dashint_{l_{1}}E_{\bm{\theta}}\\
		-\dashint_{l_2}E_{\bm{\theta}}&\dashint_{l_3}E_{\bm{\theta}}&0\\
	0&-\dashint_{l_3}E_{\bm{\theta}}&\dashint_{l_4}E_{\bm{\theta}}\\
	\end{matrix}\right).
	$$
	It is evident that $\det(\mathcal{K}^V)>0$, and $\bm{l}_2,\bm{l}_3,\bm{l}_4$ are linearly independent due to $\lambda_1>0$. Consequently, we arrive at the conclusion that $\bm{j}^V_{w}(\bm{x}) \equiv \bm{0}$ holds for all $\bm{x}\in \bar{T}$, thus implying that $\Pi_{\bm{\theta},T}^0 w \in \ker(\bm{J}_{\bm{\theta},T}^0)$.
\end{proof}

\begin{proof}[proof of Lemma \ref{lm:kernel} for $k=1$] Given $\bm{w} \in \ker(\bm{J}_{\bm{\theta}}^1)$, it follows that 
$\bm{w} = E_{\bm{\theta}}^{-1} \nabla q$ for some smooth scalar function $q$. Let $\bm{j}_{\bm{w}}^E := \bm{J}_{\bm{\theta},T}^1 \bm{\Pi}_{\bm{\theta},T}^1 \bm{w}$, which, as implied by Problem \ref{problem:curl} and \eqref{eq:weightinterpolate}, satisfies the equation
	$$
	D^E(\bm{x})\left(\begin{matrix}
		\bm{j}^E_{\bm{w}}(\bm{x})\\\bm{\Pi}_{\bm{\theta},T}^1\bm{w}(\bm{x}) 
	\end{matrix}\right)
	= \left(\begin{matrix}
		\frac{\int_{E_{12}}E_{\bm{\theta}}(\bm{w}\cdot\bm{t}_{12})}{\dashint_{E_{12}}E_{\bm{\theta}}} 
		B_2^\varepsilon(\sigma_{12},-\sigma_1)\\
		\frac{\int_{E_{13}}E_{\bm{\theta}}(\bm{w}\cdot\bm{t}_{13})}{\dashint_{E_{13}}E_{\bm{\theta}}}
		B_2^\varepsilon(\sigma_{13},-\sigma_1)\\
		\frac{\int_{E_{14}}E_{\bm{\theta}}(\bm{w}\cdot\bm{t}_{14})}{\dashint_{E_{14}}E_{\bm{\theta}}}
		B_2^\varepsilon(\sigma_{14},-\sigma_1)\\
		\frac{\int_{E_{23}}E_{\bm{\theta}}(\bm{w}\cdot\bm{t}_{23})}{\dashint_{E_{23}}E_{\bm{\theta}}}
		B_2^\varepsilon(\sigma_{23},-\sigma_2)\\
		\frac{\int_{E_{24}}E_{\bm{\theta}}(\bm{w}\cdot\bm{t}_{24})}{\dashint_{E_{24}}E_{\bm{\theta}}}
		B_2^\varepsilon(\sigma_{24},-\sigma_2)\\
		\frac{\int_{E_{34}}E_{\bm{\theta}}(\bm{w}\cdot\bm{t}_{34})}{\dashint_{E_{34}}E_{\bm{\theta}}}
		B_2^\varepsilon(\sigma_{34},-\sigma_3)
	\end{matrix}\right).
	$$
	Considering the geometric interpretation of the 2D-Bernoulli function, it is shown that
	$$
	\frac{\int_{E_{st}}E_{\bm{\theta}}(\bm{w}\cdot\bm{t}_{st})}{\dashint_{E_{st}}E_{\bm{\theta}}}
		B_2^\varepsilon(\sigma_{st},-\sigma_s) = \frac{q(\bm{x}_t) - q(\bm{x}_s)}{\dashint_{E_{st}}E_{\bm{\theta}}} \varepsilon \frac{\dashint_{E_{st}}E_{\bm{\theta}}}{\dashint_{f_{st}}E_{\bm{\theta}}} = \varepsilon \frac{q(\bm{x}_t) - q(\bm{x}_s)}{\dashint_{f_{st}}E_{\bm{\theta}}}.
	$$ 
	As a result, we can multiply each row corresponding to edge $E_{st}$ by $\dashint_{f_{st}}E_{\bm{\theta}}$, leading to the following linear system:
	$$
	\begin{pmatrix}
		({\bm{l}}_{1}\times{\bm{l}}_{2})^\top\frac{\dashint_{f_{12}} E_{\bm{\theta}} }{2}& -\varepsilon \dashint_{l_1} E_{\bm{\theta}} {\bm{l}}_{1}^\top+\varepsilon \dashint_{l_2} E_{\bm{\theta}}{\bm{l}}_{2}^\top \\
		({\bm{l}}_{1}\times{\bm{l}}_{3})^\top\frac{\dashint_{f_{13}} E_{\bm{\theta}} }{2} & -\varepsilon \dashint_{l_1} E_{\bm{\theta}}{\bm{l}}_{1}^\top+\varepsilon \dashint_{l_3} E_{\bm{\theta}}{\bm{l}}_{3}^\top \\
		({\bm{l}}_{1}\times{\bm{l}}_{4})^\top\frac{\dashint_{f_{14}} E_{\bm{\theta}} }{2} & -\varepsilon \dashint_{l_1} E_{\bm{\theta}}{\bm{l}}_{1}^\top+\varepsilon \dashint_{l_4} E_{\bm{\theta}}{\bm{l}}_{4}^\top \\
		({\bm{l}}_{2}\times{\bm{l}}_{3})^\top\frac{\dashint_{f_{23}} E_{\bm{\theta}} }{2} & -\varepsilon \dashint_{l_2} E_{\bm{\theta}}{\bm{l}}_{2}^\top+\varepsilon \dashint_{l_3} E_{\bm{\theta}}{\bm{l}}_{3}^\top \\
		({\bm{l}}_{2}\times{\bm{l}}_{4})^\top\frac{\dashint_{f_{24}} E_{\bm{\theta}} }{2} & -\varepsilon \dashint_{l_2} E_{\bm{\theta}}{\bm{l}}_{2}^\top+\varepsilon \dashint_{l_4} E_{\bm{\theta}}{\bm{l}}_{4}^\top \\
		({\bm{l}}_{3}\times{\bm{l}}_{4})^\top\frac{\dashint_{f_{34}} E_{\bm{\theta}} }{2} & -\varepsilon \dashint_{l_3} E_{\bm{\theta}}{\bm{l}}_{3}^\top+\varepsilon \dashint_{l_4} E_{\bm{\theta}}{\bm{l}}_{4}^\top 
	\end{pmatrix}
	\begin{pmatrix}
		\bm{j}^E_{\bm{w}}(\bm{x})\\\bm{\Pi}_{\bm{\theta},T}^1\bm{w}(\bm{x}) 
	\end{pmatrix}
	= \varepsilon \begin{pmatrix}
		q(\bm{x}_2)-q(\bm{x}_1)\\
		q(\bm{x}_3)-q(\bm{x}_1)\\
		q(\bm{x}_4)-q(\bm{x}_1)\\
		q(\bm{x}_3)-q(\bm{x}_2)\\
		q(\bm{x}_4)-q(\bm{x}_2)\\
		q(\bm{x}_4)-q(\bm{x}_3)\\
	\end{pmatrix}.
	$$
	A linear combination of the equation associated with edge $E_{12},E_{14},E_{24}$ (i.e., subtracting the first equation from the third and adding the fifth equation) gives
	$$
	\frac{\bm{j}^E_{\bm{w}}(\bm{x})}{2}\cdot 
	\left( \bm{l}_1\times\bm{l}_2\dashint_{f_{12}}E_{\bm{\theta}} + \bm{l}_2\times\bm{l}_4\dashint_{f_{24}}E_{\bm{\theta}}+\bm{l}_4\times\bm{l}_1\dashint_{f_{14}}E_{\bm{\theta}} \right)=0.
	$$
	Similarly, we have 
	$$
	\begin{aligned}
		&\frac{\bm{j}^E_{\bm{w}}(\bm{x}) }{2}\cdot 
		\left( \bm{l}_1\times\bm{l}_3\dashint_{f_{13}}E_{\bm{\theta}}+ \bm{l}_3\times\bm{l}_4\dashint_{E_{34}}E_{\bm{\theta}}+ \bm{l}_4\times\bm{l}_1\dashint_{E_{14}}E_{\bm{\theta}} \right)=0,\\
		&\frac{\bm{j}^E_{\bm{w}}(\bm{x}) }{2}\cdot 
		\left( \bm{l}_2\times\bm{l}_3\dashint_{f_{23}}E_{\bm{\theta}}+ \bm{l}_3\times\bm{l}_4\dashint_{E_{34}}E_{\bm{\theta}}+\bm{l}_4\times\bm{l}_2\dashint_{E_{24}}E_{\bm{\theta}} \right)=0.
	\end{aligned}
	$$
	Using the identity \eqref{eq:identityli} that $\sum_{i=1}^{4}\lambda_i\bm{l}_i=\bm{0}$ and substituting $\bm{l}_1$ by $\bm{l}_2,\bm{l}_3$ and $\bm{l}_4$, we have
	$$
	\lambda_1^{-2}\mathcal{K}^E\begin{pmatrix}
		(\bm{l}_2\times\bm{l}_3) \cdot \bm{j}^E_{\bm{w}}(\bm{x}) \\
		(\bm{l}_3\times\bm{l}_4) \cdot \bm{j}^E_{\bm{w}}(\bm{x}) \\
		(\bm{l}_4\times\bm{l}_2) \cdot \bm{j}^E_{\bm{w}}(\bm{x}) \\
	\end{pmatrix}=\bm{0},
	$$
	where 
	$$
	\mathcal{K}^E:=\left(\begin{matrix}
		\lambda_3\dashint_{E_{12}}E_{\bm{\theta}} & \lambda_3\dashint_{E_{14}}E_{\bm{\theta}}&-(\lambda_1\dashint_{E_{24}}+\lambda_4\dashint_{E_{12}}+\lambda_2\dashint_{E_{14}})E_{\bm{\theta}}\\
		-\lambda_2\dashint_{E_{13}}E_{\bm{\theta}}&(\lambda_1\dashint_{E_{34}}+\lambda_3\dashint_{E_{14}}+\lambda_4\dashint_{E_{13}})E_{\bm{\theta}} &-\lambda_2\dashint_{E_{14}}E_{\bm{\theta}}\\
		\dashint_{E_{23}}E_{\bm{\theta}}&\dashint_{E_{34}}E_{\bm{\theta}}&\dashint_{E_{24}}E_{\bm{\theta}}
	\end{matrix}\right).
	$$
	It is obvious that $\det(\mathcal{K}^E)>0$, and $\bm{l}_2\times\bm{l}_3,\bm{l}_3\times\bm{l}_4,\bm{l}_4\times\bm{l}_2$ are linearly independent due to $\lambda_1>0$. Consequently, we arrive at the conclusion that $\bm{j}^E_{\bm{w}}(\bm{x}) \equiv \bm{0}$ holds for all $\bm{x}\in \bar{T}$, thus implying that $\bm{\Pi}_{\bm{\theta},T}^1 \bm{w} \in \ker(\bm{J}_{\bm{\theta},T}^1)$.
\end{proof}

\begin{proof}[proof of Lemma \ref{lm:kernel} for $k=2$]
Given $\bm{w} \in \ker(J_{\bm{\theta}}^2)$, it follows that 
$\bm{w} = E_{\bm{\theta}}^{-1} \nabla \times \bm{q}$ for some smooth vector function $\bm{q}$.
Let $j_{\bm{w}}^F := J_{\bm{\theta},T}^2 \bm{\Pi}_{\bm{\theta},T}^2 \bm{w}$, which, as implied by Problem \ref{problem:div} and \eqref{eq:weightinterpolate}, satisfies the equation
	$$
	D^F(\bm{x})\left(\begin{matrix}
		{j}_{\bm{w}}^F(\bm{x})\\\bm{\Pi}_{\bm{\theta},T}^2\bm{w}
	\end{matrix}\right)
	= \left(\begin{matrix}
		\frac{\int_{F_{1}}E_{\bm{\theta}}(\bm{w}\cdot\bm{n}_{1})}{\dashint_{F_{1}}E_{\bm{\theta}}}B_2^\varepsilon(\sigma_{23},\sigma_{34},-\sigma_2)\\
		\frac{\int_{F_{2}}E_{\bm{\theta}}(\bm{w}\cdot\bm{n}_{2})}{\dashint_{F_{2}}E_{\bm{\theta}}}B_2^\varepsilon(\sigma_{31},\sigma_{14},-\sigma_3)\\
		\frac{\int_{F_{3}}E_{\bm{\theta}}(\bm{w}\cdot\bm{n}_{3})}{\dashint_{F_{3}}E_{\bm{\theta}}}B_2^\varepsilon(\sigma_{41},\sigma_{12},-\sigma_4)\\
		\frac{\int_{F_{4}}E_{\bm{\theta}}(\bm{w}\cdot\bm{n}_{4})}{\dashint_{F_{4}}E_{\bm{\theta}}}B_2^\varepsilon(\sigma_{13},\sigma_{32},-\sigma_1)\\
	\end{matrix}\right).
	$$
	Considering the geometric interpretation of the 3D-Bernoulli function, it is shown that
	$$
	\frac{\int_{F_{1}}E_{\bm{\theta}}(\bm{w}\cdot\bm{n}_{1})}{\dashint_{F_{1}}E_{\bm{\theta}}}B_2^\varepsilon(\sigma_{23},\sigma_{34},-\sigma_2) = \varepsilon \frac{\int_{F_1} (\nabla \times \bm{q}) \cdot \bm{n}_1}{ 
	\dashint_{T_{1}}E_{\bm{\theta}}}.
	$$ 
	Therefore, we can multiply the row associated with facet $F_{i}$ by $\dashint_{T_{i}}E_{\bm{\theta}}$ respectively and sum up all the rows, which gives
	$$
	{j}_{\bm{w}}^F(\bm{x}) \left(\sum_{i=1}^4\int_{T_i}E_{\bm{\theta}}\right) = \varepsilon\sum_{i=1}^4\int_{F_i}(\nabla\times\bm{q})\cdot\bm{n}_i=\varepsilon\int_{\partial T}\nabla\times\bm{q}\cdot\bm{n}=\varepsilon\int_T\nabla\cdot(\nabla\times\bm{q})=0.
	$$
	Then we have ${j}_{\bm{w}}^F(\bm{x})\equiv0$ in $T$, implying that $\bm{\Pi}_{\bm{\theta},T}^2 \bm{w} \in \ker(J_{\bm{\theta},T}^2)$.
\end{proof}

\subsection{Commutativity}
With the kernel preserving properties, we are able to give the commutativity of the diagram \eqref{eq:diagram}.

\begin{theorem}[commutativity] \label{tm:commutativity}
Assume that $\bm{\theta}$ is constant in $T$. It holds that $\Pi^{k+1}_{\bm{\theta},T}J_{\bm{\theta}}^k=J_{\bm{\theta},T}^k\Pi_{\bm{\theta},T}^k$ for $k=0,1,2$.
\end{theorem}

\begin{proof}[proof of Theorem \ref{tm:commutativity} for $k=0$]
	For $w\in C^\infty(T)$, employing a technique similar to that utilized in the proof of Lemma \ref{lm:kernel} for the case of $k=0$, it becomes evident that $(\bm{J}_{\bm{\theta},T}^0\Pi_{\bm{\theta},T}^0w,\Pi_{\bm{\theta},T}^0w)$ satisfy the following equation:
	$$
	\begin{pmatrix}
		\bm{l}_1^\top \dashint_{l_1}E_{\bm{\theta}} & \varepsilon E_{\bm{\theta}}(\bm{x})\\
		\bm{l}_2^\top \dashint_{l_2}E_{\bm{\theta}}& \varepsilon E_{\bm{\theta}}(\bm{x})\\
		\bm{l}_3^\top \dashint_{l_3}E_{\bm{\theta}}& \varepsilon E_{\bm{\theta}}(\bm{x})\\
		\bm{l}_4^\top \dashint_{l_4}E_{\bm{\theta}}& \varepsilon E_{\bm{\theta}}(\bm{x})\\
	\end{pmatrix}
	\begin{pmatrix}
		\bm{J}_{\bm{\theta},T}^0\Pi_{\bm{\theta},T}^0w\\\Pi_{\bm{\theta},T}^0{w}
	\end{pmatrix}
	= \varepsilon\left(\begin{matrix}
		E_{\bm{\theta}}(\bm{x}_1)w(\bm{x}_1)\\
		E_{\bm{\theta}}(\bm{x}_2)w(\bm{x}_2)\\
		E_{\bm{\theta}}(\bm{x}_3)w(\bm{x}_3)\\
		E_{\bm{\theta}}(\bm{x}_4)w(\bm{x}_4)\\
	\end{matrix}\right),
	$$
	which implies that, for $1\leq i < j \leq 4$
	\begin{equation} \label{eq:diagram-k0-1}
	\left( 
	-\bm{l}_i \dashint_{l_i}E_{\bm{\theta}}+\bm{l}_j \dashint_{l_j}E_{\bm{\theta}} \right) 
	\cdot \bm{J}_{\bm{\theta},T}^0\Pi_{\bm{\theta},T}^0w = -\varepsilon E_{\bm{\theta}}(\bm{x}_i)w(\bm{x}_i)+\varepsilon E_{\bm{\theta}}(\bm{x}_j)w(\bm{x}_j).
	\end{equation}
	
	Next, we consider the equation for $(\bm{J}_{\bm{\theta},T}^1 \bm{\Pi}_{\bm{\theta},T}^1\bm{J}_{\bm{\theta}}^0w,\bm{\Pi}_{\bm{\theta},T}^1 \bm{J}_{\bm{\theta}}^0w)$. Utilizing the geometric interpretation of the 2D-Bernoulli function, it is shown that
	$$
	\frac{\int_{E_{st}}E_{\bm{\theta}}(\bm{J}_{\bm{\theta}}^0w \cdot\bm{t}_{st})}{\dashint_{E_{st}}E_{\bm{\theta}}}
		B_2^\varepsilon(\sigma_{st},-\sigma_s) 
	= \varepsilon^2 \frac{\int_{E_{st}} \nabla (E_{\bm{\theta}} w)\cdot \bm{t}_{st} }{\dashint_{f_{st}}E_{\bm{\theta}}} 
	= \varepsilon^2 \frac{ E_{\bm{\theta}}(\bm{x}_t)w(\bm{x}_t) -  E_{\bm{\theta}}(\bm{x}_s)w(\bm{x}_s)}{\dashint_{f_{st}}E_{\bm{\theta}}}.
	$$ 
Therefore, as shown in the proof of Lemma \ref{lm:kernel} for $k=1$, we have
\begin{equation} \label{eq:diagram-k0-2}
		\begin{pmatrix}
		({\bm{l}}_{1}\times{\bm{l}}_{2})^\top\frac{\dashint_{f_{12}} E_{\bm{\theta}} }{2}& -\varepsilon \dashint_{l_1} E_{\bm{\theta}} {\bm{l}}_{1}^\top+\varepsilon \dashint_{l_2} E_{\bm{\theta}}{\bm{l}}_{2}^\top \\
		({\bm{l}}_{1}\times{\bm{l}}_{3})^\top\frac{\dashint_{f_{13}} E_{\bm{\theta}} }{2} & -\varepsilon \dashint_{l_1} E_{\bm{\theta}}{\bm{l}}_{1}^\top+\varepsilon \dashint_{l_3} E_{\bm{\theta}}{\bm{l}}_{3}^\top \\
		({\bm{l}}_{1}\times{\bm{l}}_{4})^\top\frac{\dashint_{f_{14}} E_{\bm{\theta}} }{2} & -\varepsilon \dashint_{l_1} E_{\bm{\theta}}{\bm{l}}_{1}^\top+\varepsilon \dashint_{l_4} E_{\bm{\theta}}{\bm{l}}_{4}^\top \\
		({\bm{l}}_{2}\times{\bm{l}}_{3})^\top\frac{\dashint_{f_{23}} E_{\bm{\theta}} }{2} & -\varepsilon \dashint_{l_2} E_{\bm{\theta}}{\bm{l}}_{2}^\top+\varepsilon \dashint_{l_3} E_{\bm{\theta}}{\bm{l}}_{3}^\top \\
		({\bm{l}}_{2}\times{\bm{l}}_{4})^\top\frac{\dashint_{f_{24}} E_{\bm{\theta}} }{2} & -\varepsilon \dashint_{l_2} E_{\bm{\theta}}{\bm{l}}_{2}^\top+\varepsilon \dashint_{l_4} E_{\bm{\theta}}{\bm{l}}_{4}^\top \\
		({\bm{l}}_{3}\times{\bm{l}}_{4})^\top\frac{\dashint_{f_{34}} E_{\bm{\theta}} }{2} & -\varepsilon \dashint_{l_3} E_{\bm{\theta}}{\bm{l}}_{3}^\top+\varepsilon \dashint_{l_4} E_{\bm{\theta}}{\bm{l}}_{4}^\top 
	\end{pmatrix}
	\begin{pmatrix}
		\bm{J}_{\bm{\theta},T}^1 \bm{\Pi}_{\bm{\theta},T}^1\bm{J}_{\bm{\theta}}^0w \\ 
		\bm{\Pi}_{\bm{\theta},T}^1\bm{J}_{\bm{\theta}}^0w
	\end{pmatrix}
	= \varepsilon^2 
	\begin{pmatrix}
		E_{\bm{\theta}}(\bm{x}_2)w(\bm{x}_2)-E_{\bm{\theta}}(\bm{x}_1)w(\bm{x}_1)\\
		E_{\bm{\theta}}(\bm{x}_3)w(\bm{x}_3)-E_{\bm{\theta}}(\bm{x}_1)w(\bm{x}_1)\\
		E_{\bm{\theta}}(\bm{x}_4)w(\bm{x}_4)-E_{\bm{\theta}}(\bm{x}_1)w(\bm{x}_1)\\
		E_{\bm{\theta}}(\bm{x}_3)w(\bm{x}_3)-E_{\bm{\theta}}(\bm{x}_2)w(\bm{x}_2)\\
		E_{\bm{\theta}}(\bm{x}_4)w(\bm{x}_4)-E_{\bm{\theta}}(\bm{x}_2)w(\bm{x}_2)\\
		E_{\bm{\theta}}(\bm{x}_4)w(\bm{x}_4)-E_{\bm{\theta}}(\bm{x}_3)w(\bm{x}_3)\\
	\end{pmatrix}.
\end{equation}
Note that $\bm{J}_{\bm{\theta}}^0w \in \ker(\bm{J}_{\bm{\theta}}^1)$, which subsequently invokes Lemma \ref{lm:kernel} (kernel preservation) to obtain $\bm{J}_{\bm{\theta},T}^1 \bm{\Pi}_{\bm{\theta},T}^1\bm{J}_{\bm{\theta}}^0w\equiv\bm{0}$. Then, the algebraic system \eqref{eq:diagram-k0-2} for $\bm{\Pi}_{\bm{\theta},T}^1 \bm{J}_{\bm{\theta}}^0w$ is the same as \eqref{eq:diagram-k0-1} for $\bm{J}_{\bm{\theta},T}^0 \Pi_{\bm{\theta},T}^0w$, and the system admits a unique solution by virtue of Lemma \ref{lemma:curlunisolve} (well-posedness of Problem \ref{problem:curl}). Consequently, we deduce that $\bm{\Pi}_{\bm{\theta},T}^1 \bm{J}_{\bm{\theta}}^0w = \bm{J}_{\bm{\theta},T}^0 \Pi_{\bm{\theta},T}^0w$.
\end{proof}
	
\begin{proof}[proof of Theorem \ref{tm:commutativity} for $k=1$]
	For $\bm{w}\in C^\infty(T;\mathbb{R}^3)$, employing a same technique in the proof of Lemma \ref{lm:kernel}, $(\bm{J}_{\bm{\theta},T}^1 \bm{\Pi}_{\bm{\theta},T}^1\bm{w}, \bm{\Pi}_{\bm{\theta},T}^1\bm{w})$ satisfies the following equation:
	$$
	\begin{pmatrix}
		({\bm{l}}_{1}\times{\bm{l}}_{2})^\top\frac{\dashint_{f_{12}} E_{\bm{\theta}} }{2}& -\varepsilon \dashint_{l_1} E_{\bm{\theta}} {\bm{l}}_{1}^\top+\varepsilon \dashint_{l_2} E_{\bm{\theta}}{\bm{l}}_{2}^\top \\
		({\bm{l}}_{1}\times{\bm{l}}_{3})^\top\frac{\dashint_{f_{13}} E_{\bm{\theta}} }{2} & -\varepsilon \dashint_{l_1} E_{\bm{\theta}}{\bm{l}}_{1}^\top+\varepsilon \dashint_{l_3} E_{\bm{\theta}}{\bm{l}}_{3}^\top \\
		({\bm{l}}_{1}\times{\bm{l}}_{4})^\top\frac{\dashint_{f_{14}} E_{\bm{\theta}} }{2} & -\varepsilon \dashint_{l_1} E_{\bm{\theta}}{\bm{l}}_{1}^\top+\varepsilon \dashint_{l_4} E_{\bm{\theta}}{\bm{l}}_{4}^\top \\
		({\bm{l}}_{2}\times{\bm{l}}_{3})^\top\frac{\dashint_{f_{23}} E_{\bm{\theta}} }{2} & -\varepsilon \dashint_{l_2} E_{\bm{\theta}}{\bm{l}}_{2}^\top+\varepsilon \dashint_{l_3} E_{\bm{\theta}}{\bm{l}}_{3}^\top \\
		({\bm{l}}_{2}\times{\bm{l}}_{4})^\top\frac{\dashint_{f_{24}} E_{\bm{\theta}} }{2} & -\varepsilon \dashint_{l_2} E_{\bm{\theta}}{\bm{l}}_{2}^\top+\varepsilon \dashint_{l_4} E_{\bm{\theta}}{\bm{l}}_{4}^\top \\
		({\bm{l}}_{3}\times{\bm{l}}_{4})^\top\frac{\dashint_{f_{34}} E_{\bm{\theta}} }{2} & -\varepsilon \dashint_{l_3} E_{\bm{\theta}}{\bm{l}}_{3}^\top+\varepsilon \dashint_{l_4} E_{\bm{\theta}}{\bm{l}}_{4}^\top 
	\end{pmatrix}
	\left(\begin{matrix}
		\bm{J}_{\bm{\theta},T}^1 \bm{\Pi}_{\bm{\theta},T}^1\bm{w} \\ \bm{\Pi}_{\bm{\theta},T}^1\bm{w}
	\end{matrix}\right)
	= \varepsilon\left(\begin{matrix}
		\int_{E_{12}}(\bm{w}\cdot\bm{t}_{12})E_{\bm{\theta}}\\
		\int_{E_{13}}(\bm{w}\cdot\bm{t}_{13})E_{\bm{\theta}}\\
		\int_{E_{14}}(\bm{w}\cdot\bm{t}_{14})E_{\bm{\theta}}\\
		\int_{E_{23}}(\bm{w}\cdot\bm{t}_{23})E_{\bm{\theta}}\\
		\int_{E_{24}}(\bm{w}\cdot\bm{t}_{24})E_{\bm{\theta}}\\
		\int_{E_{34}}(\bm{w}\cdot\bm{t}_{34})E_{\bm{\theta}}\\
	\end{matrix}\right).
	$$
	Through several linear combinations, we can eliminate $\bm{\Pi}_{\bm{\theta},T}^1\bm{w}$. More precisely,
\begin{align}
	& ~~ \Big( \dashint_{f_{st}}E_{\bm{\theta}}\frac{\bm{l}_s\times\bm{l}_t}{2} + \dashint_{f_{tr}}E_{\bm{\theta}}\frac{\bm{l}_t\times\bm{l}_r}{2} - \dashint_{f_{sr}}E_{\bm{\theta}}\frac{\bm{l}_s\times\bm{l}_r}{2} \Big) \cdot \bm{J}_{\bm{\theta},T}^1 \bm{\Pi}_{\bm{\theta},T}^1\bm{w} \label{eq:diagram-k1-1} \\
= ~& \varepsilon\int_{E_{st}}(\bm{w}\cdot\bm{t}_{st})E_{\bm{\theta}} + \varepsilon\int_{E_{tr}}(\bm{w}\cdot\bm{t}_{tr})E_{\bm{\theta}} - \varepsilon\int_{E_{sr}}(\bm{w}\cdot\bm{t}_{sr})E_{\bm{\theta}}, \quad \forall 1\leq s < t < r \leq 4.\notag
\end{align}
	
	Next, we consider the equation for $(J_{\bm{\theta},T}^2 \bm{\Pi}_{\bm{\theta},T}^2\bm{J}_{\bm{\theta}}^1\bm{w},\bm{\Pi}_{\bm{\theta},T}^2 \bm{J}_{\bm{\theta}}^1\bm{w})$. Utilizing the geometric interpretation of the 3D-Bernoulli function, it is shown that
$$
\frac{\int_{F_{1}}E_{\bm{\theta}}(\bm{J}_{\bm{\theta}}^1 \bm{w} \cdot\bm{n}_{1})}{\dashint_{F_{1}}E_{\bm{\theta}}}B_2^\varepsilon(\sigma_{23},\sigma_{34},-\sigma_2) = 
\varepsilon^2 \frac{\int_{F_1} \nabla \times (E_{\bm{\theta}} \bm{w}) \cdot \bm{n}_1}{\dashint_{T_1} E_{\bm{\theta}}}.
$$  
Therefore, regarding to Problem \ref{problem:div}, we can multiply the row associated by facet $F_{i}$ by $\dashint_{T_{i}}E_{\bm{\theta}}$ respectively, we have 
\begin{equation} \label{eq:diagram-k1-2}
	\tilde{D}^F(\bm{x})
	\begin{pmatrix}
		J_{\bm{\theta},T}^2 \bm{\Pi}_{\bm{\theta},T}^2 \bm{J}_{\bm{\theta}}^1\bm{w} \\
		\bm{\Pi}_{\bm{\theta},T}^2 \bm{J}_{\bm{\theta}}^1\bm{w}
	\end{pmatrix}
	=\varepsilon^2 
	\begin{pmatrix}
	\int_{F_1} \nabla\times (E_{\bm{\theta}}\bm{w})\cdot \bm{n}_1\\
	\int_{F_2} \nabla\times (E_{\bm{\theta}}\bm{w})\cdot \bm{n}_2\\
	\int_{F_3} \nabla\times (E_{\bm{\theta}}\bm{w})\cdot \bm{n}_3\\
	\int_{F_4} \nabla\times (E_{\bm{\theta}}\bm{w})\cdot \bm{n}_4
	\end{pmatrix},
\end{equation}
	where
	$$
	\tilde{D}^F(\bm{x}) :=
	\begin{pmatrix}
	\int_{T_1}E_{\bm{\theta}}	&\varepsilon( \dashint_{f_{23}}E_{\bm{\theta}}\frac{\bm{l}_2\times\bm{l}_3}{2}+ \dashint_{f_{34}}E_{\bm{\theta}}\frac{\bm{l}_3\times\bm{l}_4}{2}+ \dashint_{f_{24}}E_{\bm{\theta}}\frac{\bm{l}_4\times\bm{l}_2}{2})^\top\\
		\int_{T_2}E_{\bm{\theta}}	&\varepsilon( \dashint_{f_{13}}E_{\bm{\theta}}\frac{\bm{l}_3\times\bm{l}_1}{2}+ \dashint_{f_{14}}E_{\bm{\theta}}\frac{\bm{l}_1\times\bm{l}_4}{2}+ \dashint_{f_{34}}E_{\bm{\theta}}\frac{\bm{l}_4\times\bm{l}_3}{2})^\top\\
		\int_{T_3}E_{\bm{\theta}}	&\varepsilon(\dashint_{f_{14}}E_{\bm{\theta}}\frac{\bm{l}_4\times\bm{l}_1}{2}+ \dashint_{f_{12}}E_{\bm{\theta}}\frac{\bm{l}_1\times\bm{l}_2}{2}+ \dashint_{f_{24}}E_{\bm{\theta}}\frac{\bm{l}_2\times\bm{l}_4)}{2})^\top\\
		\int_{T_4}E_{\bm{\theta}}	&\varepsilon( \dashint_{f_{13}}E_{\bm{\theta}}\frac{\bm{l}_1\times\bm{l}_3}{2}+ \dashint_{f_{23}}E_{\bm{\theta}}\frac{\bm{l}_3\times\bm{l}_2}{2}+\dashint_{f_{12}}E_{\bm{\theta}}\frac{\bm{l}_2\times\bm{l}_1}{2})^\top
	\end{pmatrix}.
	$$
Note that $\bm{J}_{\bm{\theta}}^1\bm{w} \in \ker(J_{\bm{\theta}}^2)$, which subsequently invokes Lemma \ref{lm:kernel} (kernel preservation) to obtain $J_{\bm{\theta},T}^2 \bm{\Pi}_{\bm{\theta},T}^2 \bm{J}_{\bm{\theta}}^1 \bm{w} \equiv 0$. Then, the algebraic system \eqref{eq:diagram-k1-2} for $\bm{\Pi}_{\bm{\theta},T}^2 \bm{J}_{\bm{\theta}}^1\bm{w}$ is the same as \eqref{eq:diagram-k1-1} for $\bm{J}_{\bm{\theta},T}^1 \bm{\Pi}_{\bm{\theta},T}^1\bm{w}$. Therefore, $\bm{\Pi}_{\bm{\theta},T}^2 \bm{J}_{\bm{\theta}}^1\bm{w} = \bm{J}_{\bm{\theta},T}^1 \Pi_{\bm{\theta},T}^1\bm{w}$ by applying Lemma \ref{lm:divunisolve} (well-posedness of Problem \ref{problem:div}).
\end{proof}

\begin{proof}[proof of Theorem \ref{tm:commutativity} for $k=2$] For $\bm{w} \in C^\infty(T;\mathbb{R}^3)$, Using \eqref{eq:weightinterpolate} and the fact that $\phi^T = \frac{1}{|T|}$, a direct calculation shows that
$$ 
\Pi_{\bm{\theta},T}^3 J_{\bm{\theta}}^2 \bm{w} = \frac{l_T^3(E_{\bm{\theta}}J_{\bm{\theta}}^2 \bm{w})}{ \dashint_T E_{\bm{\theta}}} \phi^T 
= \varepsilon \frac{\int_T \nabla \cdot (E_{\bm{\theta}} \bm{w})}{ \int_T E_{\bm{\theta}}}  
= \varepsilon \frac{\int_{\partial T} E_{\bm{\theta}} \bm{w} \cdot \bm{n} }{ \int_T E_{\bm{\theta}}}.
$$ 
Furthermore, applying Corollary \ref{coro:divflux} (constant flux of $\mathcal{S}_{1^-}^2$),
$$ 
\begin{aligned}
J_{\bm{\theta},T}^2 \bm{\Pi}_{\bm{\theta},T}^2 \bm{w} 
& = \sum_{F_i \in \mathcal{F}_T} l_i^2(\bm{\Pi}_{\bm{\theta},T}^2 \bm{w}) j_i^F(\bm{x}) 
= \sum_{F_i\in\mathcal{F}_T} \frac{l_i^2(E_{\bm{\theta}}\bm{w})}{\dashint_{F_i}E_{\bm{\theta}}} j_i^F(\bm{x}) \\
& = \sum_{F_i\in\mathcal{F}_T} \varepsilon \frac{l_i^2(E_{\bm{\theta}}\bm{w})}{\dashint_{F_i}E_{\bm{\theta}}}
\cdot  \frac{\dashint_{F_i}E_{\bm{\theta}}}{\int_{T}E_{\bm{\theta}}} 
= \varepsilon  \sum_{F_i\in\mathcal{F}_T} \frac{\int_{F_i} E_{\bm{\theta}} \bm{w} \cdot \bm{n}_i}{\int_{T}E_{\bm{\theta}}}, 
\end{aligned}
$$ 
we have $\Pi_{\bm{\theta},T}^3 J_{\bm{\theta}}^2 \bm{w} =  J_{\bm{\theta},T}^2 \bm{\Pi}_{\bm{\theta},T}^2 \bm{w}$.\end{proof}

Since constant is contained in the local space $\mathcal{S}_{1^-}^k(T)$, the following corollary directly follows from  Theorem \ref{tm:commutativity} (commutativity).

\begin{corollary}[constant preservation for flux] \label{co:const-J}
Assume that $\bm{\theta}$ is constant in $T$. If the flux ${J}^k_{\bm{\theta}}{w}$ is constant in a simplex $T$, then the discrete flux $J_{\bm{\theta},T}^k\Pi_{\bm{\theta},T}^kw$ is the same constant.
\end{corollary}

\begin{remark}[piecewise constant approximation of variable $\bm{\theta}$] \label{rm:theta-T}
	Although the above induction assume that $\bm{\theta}$ is constant in $T$, given a fixed point $\bm{x}_T\in\bar{T}$, we can apply  $\bm{\theta} \equiv \bm{\theta}(\bm{x}_T)$ to derive  such that $J_{\bm{\theta}(\bm{x}_T),T}^k \Pi_{{\bm{\theta}}(\bm{x}_T),T}^kw =\Pi_{{\bm{\theta}}(\bm{x}_T),T}^{k+1}J_{\bm{\theta}(\bm{x}_T)}^kw $, for $k=0,1,2$.
\end{remark}
%
%

\section{Applications to convection-diffusion problems} \label{sec:application}
In this section, we apply the exponentially-fitted finite elements designed above to solve convection-diffusion equations. Following the definitions of the differential operator $d^k$ and the flux $J_{\bm{\theta}}^k$ given by \eqref{eq:flux}, the convection-diffusion equations \eqref{eq:grad-cd}, \eqref{eq:curl-cd}, and \eqref{eq:div-cd} adopt a unified form for different values of $k$, corresponding to $k=0,1,2$ respectively:

\begin{equation}\label{eq:generalcd}
	\left\{\begin{aligned}
		&d^{k,*}  J_{\bm{\theta}}^k u + \gamma {u}={f}&&\text{ in }\Omega,\\
		&{\rm tr}(u)={0}&&\text{ on }\partial\Omega,
	\end{aligned}\right.
\end{equation}
where $d^{k,*}$ represents the dual operator of $d^k$, and the expression for the trace operator is provided within the equations. The corresponding variational problem has the form:
\begin{equation}\label{eq:bilinear}
a(u,v)=F(v) \qquad \forall v\in V := \{w \in H(d^k;\Omega):~ \mathrm{tr}(w) = 0 \text{ on }\partial \Omega\},
\end{equation}
with $H(d^k;\Omega) := \{w \in L^2(\Omega):~ d^k w \in L^2(\Omega)\}$, 
$$
a(u,v):=(J_{\bm{\theta}}^ku,d^kv)+(\gamma u,v), \quad F(v) := (f,v).
$$

\subsection{An intrinsic Petrov-Galerkin method}
The expression of the bilinear form \eqref{eq:bilinear} reveals a distinction: for the trail functions, our focus lies on both their values and flux, while for the test functions, we emphasize their values and differentials. The discretization of the former precisely aligns with the exponentially-fitted finite element $\mathcal{S}_{1^-}^k$, while the latter corresponds to the conventional finite element $\mathcal{P}_{1^-}^k$. Therefore, we define 
\begin{equation} \label{eq:discrete-space}
\begin{aligned}
 S_h &:= \{{v}_h\in \mathcal{S}_{1^-}^k(\mathcal{T}_h):{\rm tr}(v_h)=0 \text{ on }\partial\Omega\},  \\
 V_h &:=\{v_h\in\mathcal{P}_{1^-}^k(\mathcal{T}_h):{\rm tr}(v_h)=0 \text{ on }\partial\Omega\}.
 \end{aligned}
 \end{equation}
 
 We define the bilinear form as follows:
\begin{equation}
	\label{eq:globalbilinear}
	a_h(u_h,v_h) : =\sum_{T\in\mathcal{T}_h} \underbrace{(J_{\bm{\theta},T}^ku_h,d^kv_h)_T}_{:=a_T(u_h,v_h)} +(\gamma u_h,v_h).
\end{equation}
Then, the corresponding discrete variational problem using a {\it Petrov-Galerkin method} can be stated as follows: Find ${u}_h\in S_h$ such that
\begin{equation} \label{eq:discretevration}
	a_h({u}_h,{v}_h) = F({v}_h), \quad \forall {v}_h\in V_h.
\end{equation}
The standard energy norm in $V$ is defined as:
\begin{equation}
	\label{eq:standardnorm}
	\|{v}\|_{H(d^k),\Omega}:=(\|d^kv\|_{0,\Omega}^2+\|v\|_{0,\Omega}^2)^{1/2},\quad \forall v\in V.
\end{equation}
For a function $w_h\in S_h$, we utilize its discrete flux to define the following discrete energy norm in $S_h$:
\begin{equation}
	\label{eq:discretenorm}
	\|{w_h}\|_{\mathcal{S}^k,h}:=\big(\sum_{T\in\mathcal{T}_h} \|w_h\|_{\mathcal{S}^k,T}^2)^{1/2},\quad\forall w_h \in S_h.
\end{equation}
where $\|w_h\|_{\mathcal{S}^k,T} ^2:= \|J^k_{\bm{\theta},T}w_h\|_{0,T}^2+\|w_h\|_{0,T}^2$.

Subsequently, we will present an analysis of the Petrov-Galerkin formulation \eqref{eq:discretevration}. 
We firstly provide some local error estimates associated with the exponentially-fitted function spaces $\mathcal{S}_{1^-}^k$. 
Under the well-posedness of the model problems, we proceed to demonstrate the well-posedness of the discrete problems, accompanied by a comprehensive examination of their corresponding convergence properties.

\subsection{Local error estimates}

Let us begin by briefly examining the scaling argument of the exponentially-fitted finite element spaces, as these spaces are not exclusively composed of polynomials. Consider the reference element $\hat{T}$, and let the mapping $F: \hat{T} \to T$ be defined as $\bm{x} = F(\hat{\bm{x}}) := B \hat{\bm{x}} + \bm{b}$. Here, the reference velocity is characterized by $\hat{\bm{\beta}}:= B^T \bm{\beta} \circ F$.
\begin{itemize}
\item $\mathcal{S}_{1^-}^0$. For functions $\hat{\phi}(\hat{\bm{x}})$ and $\hat{\bm{j}}(\hat{\bm{x}})$ defined on $\hat{T}$, we define $\phi(\bm{x})$ and $\bm{j}(\bm{x})$ on $T$ by
\begin{equation} \label{eq:scaling-grad}
\phi := \hat{\phi}\circ F^{-1}, \quad \bm{j} := B^{-T} \hat{\bm{j}} \circ F^{-1}.
\end{equation}
Observing that $\sigma_i = \bm{\beta} \cdot \bm{l}_i = \bm{\beta} \cdot (B\hat{\bm{l}}_i) = (B^T\bm{\beta}) \cdot \hat{\bm{l}}_i$, the left-hand side of Problem \ref{problem:grad} yields
$$
\begin{pmatrix}
		\bm{l}_1^\top &  B_1^\varepsilon(\sigma_1)\\
		\bm{l}_2^\top &  B_1^\varepsilon(\sigma_2)\\
		\bm{l}_3^\top &  B_1^\varepsilon(\sigma_3)\\
		\bm{l}_4^\top &  B_1^\varepsilon(\sigma_4)
\end{pmatrix}
\begin{pmatrix}
	\bm{j}_1^V \\\phi_1^V
\end{pmatrix} =
\begin{pmatrix}
		\hat{\bm{l}}_1^\top &  B_1^\varepsilon(\hat{\sigma}_1)\\
		\hat{\bm{l}}_2^\top &  B_1^\varepsilon(\hat{\sigma}_2)\\
		\hat{\bm{l}}_3^\top &  B_1^\varepsilon(\hat{\sigma}_3)\\
		\hat{\bm{l}}_4^\top &  B_1^\varepsilon(\hat{\sigma}_4)
\end{pmatrix}
\begin{pmatrix}
	B^T\bm{j}_1^V \\ \phi_1^V
\end{pmatrix}.
$$ 
Therefore, for $\mathcal{S}_{1^-}^0$, the basis and its flux on element $T$ can be given by the transformation \eqref{eq:scaling-grad} when referred back to the reference.

\item $\mathcal{S}_{1^-}^1$. For functions $\hat{\bm{\phi}}(\hat{\bm{x}})$ and $\hat{\bm{j}}(\hat{\bm{x}})$ defined on $\hat{T}$, we define $\bm{\phi}(\bm{x})$ and $\bm{j}(\bm{x})$ on $T$ by
\begin{equation} \label{eq:scaling-curl}
\bm{\phi} := B^{-T} \hat{\bm{\phi}} \circ F^{-1}, \quad \bm{j} = (\det B)^{-1} B \hat{\bm{j}} \circ F^{-1}.
\end{equation}
Observing that the cross product adheres to the subsequent identity: 
$$ 
(B\hat{\bm{l}}_i) \times (B\hat{\bm{l}}_i) = (\det B) B^{-T} (\hat{\bm{l}}_i \times \hat{\bm{l}}_j).
$$ 
Consequently, it becomes evident from Problem \ref{problem:curl} that, in the context of $\mathcal{S}_{1^-}^1$, the basis and its functions on element $T$ can be elucidated by the transformation \eqref{eq:scaling-curl} when referred back to the reference.

\item $\mathcal{S}_{1^-}^2$. For functions $\hat{\bm{\phi}}(\hat{\bm{x}})$ and $\hat{j}(\hat{\bm{x}})$ defined on $\hat{T}$, we define $\bm{\phi}(\bm{x})$ and $j(\bm{x})$ on $T$ by
\begin{equation} \label{eq:scaling-div}
\bm{\phi} := (\det B)^{-1}B \hat{\phi} \circ F^{-1}, \quad j := (\det B)^{-1} \hat{j} \circ F^{-1}. 
\end{equation}
Similarly, from Problem \ref{problem:div}, it is evident that, for $\mathcal{S}_{1^-}^2$, it corresponds to the transformation from the reference element to $T$.
\end{itemize}

For the sake of uniform notation and simplicity, we employ $\phi_S^k$ and $j_S^k$ to denote the basis function and discrete flux corresponding to the sub-simplex $S$ in $\mathcal{V}_T$ (for $k=0$), $\mathcal{E}_T$ (for $k=1$), or $\mathcal{F}_T$ (for $k=2$).

For the basis functions and their discrete fluxes on the reference element, it is reasonable to assume their boundedness, i.e., there exists a constant $\hat{C}_b$ such that:
\begin{equation} \label{eq:bound-basis}
\|\hat{\phi}_{\hat{S}}^k\|_{0,\infty, \hat{T}} \leq \hat{C}_b, \quad 
\|\hat{j}_{\hat{S}}^k\|_{0,\infty, \hat{T}} \leq \hat{C}_b.
\end{equation}
Here, $\hat{C}_b$ may depend on $\varepsilon$ and $\bm{\beta}$. Given a prescribed $\varepsilon > 0$, the the Bernoulli function with respect to $\bm{\beta}$ is a continuous positive function. Therefore, when $\bm{\beta}$ is bounded, the bound $\hat{C}_b$ above can be uniform with respect to $\bm{\beta}$. However, as the Bernoulli function approaches zero from the positive side as $\varepsilon \to 0^+$, while its limit exists (implying the meaningfulness of the linear system for basis functions and flux), its overall behavior remains only non-negative. As a result, establishing the $\varepsilon$-uniform boundedness of the corresponding basis functions cannot be inferred solely through rudimentary analysis. Despite our belief and the observed $\varepsilon$-uniform boundedness in the earlier basis function illustrations (see Figures \ref{fig:curlbasis1} and \ref{fig:divbasis1}), meticulous analysis and verification are still required. Nevertheless, the bound \eqref{eq:bound-basis} does exist, where the constant $\hat{C}_b$ should at most depend on $\varepsilon$.

Utilizing the aforementioned scaling argument and the boundedness of the basis functions, we will establish several local error estimates.
\begin{lemma}[approximation property]\label{lemma:approxinterpolate}
	For any $T\in\mathcal{T}_h$, if $g\in W^{1,p}(T)$ and $p>3$, we have
	\begin{equation}
		\|g-\Pi_{\bm{\theta},T}^k g \|_{0,s,T} \lesssim
		\hat{C}_{b}C(p)h_T^{1+3({1\over s}-{1\over p})}|g|_{1,p,T} \qquad 1\leq s \leq
		\infty.
	\end{equation}
	Here, $C(p) \eqsim \max \{1, (p-3)^{-\sigma}\}$ where $\sigma$ is a
	positive number determined by Sobolev embedding.
\end{lemma}
\begin{proof}
	The interpolation operator $\Pi_{{\bm{\theta}},T}^k$ shares similarities with the weighted interpolation operator introduced in the work in \cite{wu2020simplex}. The key distinction lies in the utilization of exponentially-fitted basis functions instead of polynomials. Despite this difference, the proof remains unchanged and is achieved through scaling arguments \eqref{eq:scaling-grad}--\eqref{eq:scaling-div}, the Sobolev embedding theorem, and the  boundedness $\hat{C}_{b}$ of the basis functions on a reference element given by \eqref{eq:bound-basis}. We refer to \cite[Lemma 5.1]{wu2020simplex} for a detailed proof.
\end{proof}
Through the proof of Lemma \ref{lemma:approxinterpolate}, we can obtain the following stability of $\Pi_{{\bm{\theta}},T}^k$ that is mimic to \cite[Corollary 5.2]{wu2020simplex}.
\begin{corollary}[stability]\label{coro:stableinterpol}
	For any $T\in \mathcal{T}_h$, if $w\in L^\infty(T)$, we have
	\begin{equation}
		\|\Pi_{{\bm{\theta}},T}^kw\|_{0,s,T}\lesssim \hat{C}_{b}h_T^{\frac{3}{s}}\|w\|_{0,\infty,T} \qquad 1\le s\le \infty.
	\end{equation}
\end{corollary}

Note that for any function $w_h\in\mathcal{S}_{1^-}^k(T)$, we denote $\tilde{\Pi}_T^k$ as the local canonical interpolation operator onto $\mathcal{P}_{1^-}^k(T)$ and $\tilde{w}_h=\tilde{\Pi}_T^kw_h$. Using the properties in Lemma \ref{lemma:curldof} (DOFs of $\mathcal{S}_{1^-}^1(T)$) and Lemma \ref{lm:divdof} (DOFs of $\mathcal{S}_{1^-}^2(T)$), it is obvious that
\begin{equation} \label{eq:identification}
		w_h =\Pi_{\bm{\theta},T}^k\tilde{w}_h,
\end{equation}
which also imply that $\tilde{\Pi}_T^k\Pi_{{\bm{\theta}},T}^k|_{\mathcal{P}_{1^-}^k(T)}=Id_{\mathcal{P}_{1^-}^k(T)}$ and $\Pi_{{\bm{\theta}},T}^k\tilde{\Pi}_T^k|_{\mathcal{S}_{1^-}^k(T)}=Id_{\mathcal{S}_{1^-}^k(T)}$. Then, we have the following result for the relationship between $\|\tilde{w}_h\|_{H(d^k),T}$ and $\|w_h\|_{\mathcal{S}^k,T}$.

\begin{lemma}[relationship between energy norms]\label{lemma:normequiv}
	For any $w_h\in\mathcal{S}_{1^-}^k(T)$ with $\tilde{w}_h=\tilde{\Pi}_T^kw_h\in \mathcal{P}_{1^-}^k(T)$, it holds that
	\begin{equation} \label{eq:bound1}
		\begin{aligned}
			\|J_{\bm{\theta},T}^kw_h\|_{0,T}^2+\|w_h\|_{0,T}^2\lesssim \hat{C}_{b} ^2\left(\varepsilon^2\|d^k\tilde{w}_h\|_{0,T}^2+(1+\|\bm{\beta}\|_{0,T}^2)\|\tilde{w}_h\|_{0,T}^2\right),
		\end{aligned}
	\end{equation}
	where the hidden constant does not depend on $h$.
\end{lemma}

\begin{proof}
It is readily seen from \eqref{eq:identification} that
	$$
	\begin{aligned}
		\|J_{\bm{\theta},T}^kw_h\|_{0,T} \lesssim h_T^{\frac{3}{2}} \|J_{\bm{\theta},T}^k w_h\|_{0,\infty,T}
		= h_T^{\frac32}\|J_{\bm{\theta},T}^k \Pi_{{\bm{\theta}},T}^k\tilde{w}_h\|_{0,\infty,T}.
	\end{aligned}
	$$	
	Assuming that the above $L^\infty$ norm is obtained on $\bm{x}_T\in \bar{T}$, and we have 
	$$
	J_{\bm{\theta},T}^k\Pi_{{\bm{\theta}},T}^k\tilde{w}_h(\bm{x}_T)=J_{\bm{\theta}(\bm{x}_T),T}^k\Pi_{{\bm{\theta}(\bm{x}_T)},T}^k\tilde{w}_h(\bm{x}_T)=\Pi_{{\bm{\theta}(\bm{x}_T)},T}^{k+1}J_{\bm{\theta}(\bm{x}_T)}^k\tilde{w}_h(\bm{x}_T)$$ 
	since $\tilde{w}_h$ shares the same DOF-condition and the discrete flux is defined pointwisely and the commutativity holds for a constant scaled convection $\bm{\theta}(\bm{x}_T)$. Then, utilizing the stability of $\Pi_{\bm{\theta}(\bm{x}_T),T}^{k+1}$ in Corollary \ref{coro:stableinterpol}, it holds that
	$$
	\begin{aligned}
		h^{\frac{3}{2}}\|J_{\bm{\theta},T}^k\Pi_{{\bm{\theta}},T}^k\tilde{w}_h\|_{0,\infty,T}
		&=h^{\frac{3}{2}}|\Pi_{{\bm{\theta}(\bm{x}_T)},T}^{k+1}J_{\bm{\theta}(\bm{x}_T)}^k\tilde{w}_h(\bm{x}_T)|\\
		&\lesssim \hat{C}_{b}h^{\frac{3}{2}}\|J_{\bm{\theta}(\bm{x}_T)}^k\tilde{w}_h\|_{0,\infty,T}\\
		&\le\hat{C}_{b} h^{\frac{3}{2}}( \varepsilon\|d^k\tilde{w}_h\|_{0,\infty,T}+\|\bm{\beta}\|_{0,\infty,T}\|\tilde{w}_h\|_{0,\infty,T}).
	\end{aligned}
	$$
	Recognizing that both $d^k\tilde{w}_h$ and $\tilde{w}_h$ are polynomials, and leveraging the standard inverse inequality, we deduce:
	$$
	\begin{aligned}
		\|J_{\bm{\theta},T}^kw_h\|_{0,T}&\lesssim\hat{C}_{b}\left( \varepsilon\|d^k\tilde{w}_h\|_{0,T}+\|\bm{\beta}\|_{0,\infty,T}\|\tilde{w}_h\|_{0,T}\right).
	\end{aligned}
	$$
	Likewise, considering the term $\|w_h\|_{0,T}$, we obtain:
	$$
	\|w_h\|_{0,T}=\|\Pi_{{\bm{\theta}},T}^k\tilde{w}_h\|_{0,T}\lesssim\hat{C}_{b}  h^{\frac{3}{2}}\|\tilde{w}_h\|_{0,\infty,T} \lesssim \hat{C}_{b} \|\tilde{w}_h\|_{0,T}.
	$$
	Thus, we arrive at the desired result \eqref{eq:bound1}.

\end{proof}

To utilize the commutativity of the diagram, we also need the estimate of flux between $\Pi_{\bm{\theta},T}^k$ and  $\Pi_{{\bm{\theta}(\bm{x}_T)},T}^k$ for a given $\bm{x}_T\in \bar{T}$, which is analogue to the result in \cite[Lemma 5.3]{wu2020simplex}.
\begin{lemma}[flux difference]\label{lemma:diffinter}
	For any $\bm{x}_T\in \bar{T}$, if ${w}\in {W}^{1,p}(T)$, $p>3$, and $h_T
	\lesssim \|\bm{\theta}\|_{1,\infty,T}^{-1}$, we have 
	\begin{equation}
			\|J_{\bm{\theta}(\bm{x}_T),T}^k\Pi_{{\bm{\theta}}(\bm{x}_T),T}^k(I-
		{\Pi}_{\bm{\theta},T}^k){w}\|_{0,\infty,T} \lesssim \hat{C}_{b} C(p)h_T^{1 
			- {3\over p}} |\bm{\theta}|_{1,\infty,T}
		|{w}|_{1,p,T}.
	\end{equation}
	
\end{lemma}

\begin{proof}
	It is important to observe that the disparity between these two interpolations can be attributed to the coefficient of the flux value:
	$$
	\begin{aligned}
	J_{\bm{\theta}(\bm{x}_T),T}^k\Pi_{{\bm{\theta}}(\bm{x}_T),T}^kw &= \sum_S\frac{l_S^k(E_{\bm{\theta}(\bm{x}_T)}w)}{\dashint_S E_{\bm{\theta}(\bm{x}_T)}}j_S^k, \\
	J_{\bm{\theta}(\bm{x}_T),T}^k\Pi_{{\bm{\theta}}(\bm{x}_T),T}^k{\Pi}_{\bm{\theta},T}^kw &= \sum_S\frac{l_S^k(E_{\bm{\theta}}w)}{\dashint_S E_{\bm{\theta}}}j_S^k,
	\end{aligned}
	$$
	where $S\in\mathcal{V}_T$ for $k=0$, $S\in\mathcal{E}_T$ for $k=1$, $S\in\mathcal{F}_T$ for $k=2$ and  $j_S^k$ is the corresponding discrete flux of the local exponentially-fitted FE space  with the same notation mentioned before.
	
	 A similar induction for the difference of the coefficient, as presented in \cite[Lemma 5.3]{wu2020simplex} (There appears to be a typographical error in the first equation on Page 898 of \cite{wu2020simplex}, as a term of $|S|$ seems to be omitted), shows that
	 $$
	 \left|
	 \frac{l_S^k(E_{\bm{\theta}}w)}{\dashint_S E_{\bm{\theta}}} - 
	 \frac{l_S^k(E_{\bm{\theta}(\bm{x}_T)}w)}{\dashint_S E_{\bm{\theta}(\bm{x}_T)}}\right| 
	 \lesssim h_T |\bm{\theta}|_{1,\infty,T}\|w\|_{0,\infty,S}|S|\lesssim h_T^{k+1} |\bm{\theta}|_{1,\infty,T}\|w\|_{0,\infty,T}.
	 $$
	 Note that for any $w_h \in \mathcal{S}_{1^-}^k(T)$, $(I-\Pi_{{\bm{\theta}},T}^k)w = (I-\Pi_{{\bm{\theta}},T}^k)(w-w_h)$, and a constant is encompassed within $\mathcal{S}_{1^-}^k(T)$ due to Remarks \ref{rm:approx-curl} and \ref{rm:approx-div}. This entails the application of the Bramble-Hilbert lemma, which asserts that:
	  $$
	  \left|
	  \frac{l_S^k(E_{\bm{\theta}}w)}{\dashint_S E_{\bm{\theta}}} - 
	  \frac{l_S^k(E_{\bm{\theta}(\bm{x}_T)}w)}{\dashint_S E_{\bm{\theta}(\bm{x}_T)}}\right| 
	  \lesssim  C(p)h_T^{k+2-{3\over p}} |\bm{\theta}|_{1,\infty,T}\|w\|_{1,p,T}.
	  $$
	  Through scaling argument \eqref{eq:scaling-grad}--\eqref{eq:scaling-div}, we can obtain the boundedness of the discrete flux \eqref{eq:bound-basis} on $T$ that 
	  $$
	  \|j_S^k\|_{0,\infty,T}\lesssim \hat{C}_{b}h_T^{-(k+1)}.
	  $$
	  Thus, we arrive at the desired result.
\end{proof}	

\subsection{Error analysis}
Initially, we present the subsequent assumption concerning the well-posedness of the convection-diffusion problems  \eqref{eq:bilinear}.
\begin{assumption}[well-posedness]
	\label{ass:infsupa} 
	There exists a constant $c_0$ (which may depend on $\varepsilon,\bm{\beta},\gamma$) such that
	\begin{equation}
		\label{eq:infsupa}
		\inf_{u\in V}\sup_{v\in V} \frac{a(u,v)}{\|u\|_{H(d^k),\Omega}\|v\|_{H(d^k),\Omega}}=c_0>0.
	\end{equation}
\end{assumption}
\begin{remark}
	The above assumption holds for $H(\mathrm{grad})$
	convection-diffusion problem by using the weak maximum principle (cf.
	\cite[Section 8.1]{gilbarg1977elliptic}) and Fredholm alternative
	theory (cf. \cite[Theorem 4, pp. 303]{evans2010partial}). A sufficient condition
	for the above assumption is that $4\varepsilon\gamma(\bm{x}) \geq |\bm{\beta}(\bm{x})|_{l^2}^2$ 
	for all $\bm{x} \in \Omega$.
\end{remark}

%
\begin{lemma}[consistency error]\label{lemma:diff}
	For any $T\in\mathcal{T}_h$, assume that $J_{{\bm{\theta}}}^k{w}\in {W}^{1,p}(T)$ and ${w}\in W^{1,p}(T)$ where $p>3$ and $h_T\lesssim\|\bm{\theta}\|_{1,\infty,T}^{-1}$. Then the following inequality holds
	\begin{equation}
		|a({w},{v}_h)-a_h(\Pi_{\bm{\theta},h}^k{w},{v}_h)|\lesssim \Theta(\varepsilon,\bm{\theta},\gamma,{w})h\|{v}_h\|_{H(d^k),\Omega}\qquad \forall {v}_h\in V_h,
	\end{equation} 
	where 
	\begin{equation} \label{eq:theta}
		\begin{aligned}
			&\Theta(\varepsilon,\bm{\theta},{\gamma},{w}):=\\
			&\hat{C}_{b}C(p)\bigg\{\sum_{T\in\mathcal{T}_h}\Big( h_T^{3(\frac{1}{2}-\frac{1}{p})}(|J_{{\bm{\theta}}}^k{w}|_{1,p,T}+ |\bm{\beta}|_{1,\infty,T}\|w\|_{0,p,T}+ h_T|\bm{\beta}|_{1,\infty,T}|w|_{1,p,T})\Big)^2\\
			&~+\sum_{T\in\mathcal{T}_h}\Big(|\bm{\theta}|_{1,\infty,T} h_T^{3(\frac{1}{2}-\frac{1}{p})}|w|_{1,p,T}\Big)^2+\Big(\|\gamma\|_{0,\infty,T} h_T ^{3(\frac{1}{2}-\frac{1}{p})}|{w}|_{1,p,T}\Big)^2\bigg\}^{\frac{1}{2}}.
		\end{aligned}	
	\end{equation}
\end{lemma}

\begin{proof}
	Assuming that $\|J_{\bm{\theta}}^kw-J_{\bm{\theta},T}^k\Pi_{{\bm{\theta}},T}^k{w}\|_{0,\infty,T}$ is obtain on $\bm{x}_T\in \bar{T}$. By \eqref{eq:globalbilinear} and the commutativity of diagram \eqref{eq:diagram} for $\bm{\theta}\equiv \bm{\theta}(\bm{x}_T)$ in $\bar{T}$ (see Remark \ref{rm:theta-T}), we have
	$$
	\begin{aligned}
		a({w},v_h)-a_h&(\Pi_{\bm{\theta},h}^k{w},v_h) \\
	 =&\sum_{T\in\mathcal{T}_h} (J_{\bm{\theta}}^k{w}-J_{\bm{\theta},T}^k\Pi_{{\bm{\theta}},T}^k{w},d^k{v}_h)_T+(\gamma({w}-\Pi_{{\bm{\theta}},T}^k{w}),{v}_h)_T\\
		\lesssim &\sum_{T\in\mathcal{T}_h}h_T^{3\over2}|J_{\bm{\theta}}^k{w}(\bm{x}_T)-J_{\bm{\theta},T}^k\Pi_{{\bm{\theta}},T}^k{w}(\bm{x}_T)| \|d^kv_h\|_{0,T}+(\gamma({w}-\Pi_{{\bm{\theta}},T}^k{w}),{v}_h)_T\\
		\lesssim &\sum_{T\in\mathcal{T}_h}h_T^{3\over2}\bigg(\underbrace{|J^k_{\bm{\theta}}w(\bm{x}_T)-\Pi_{{\bm{\theta}(\bm{x}_T)},T}^{k+1}J_{\bm{\theta}(\bm{x}_T)}^k{w}(\bm{x}_T)|}_{I_{1,T}}\\
		&+\underbrace{|J_{\bm{\theta}(\bm{x}_T),T}^k\Pi_{{\bm{\theta}(\bm{x}_T)},T}^k{w}(\bm{x}_T)-J_{\bm{\theta},T}^k\Pi_{{\bm{\theta}},T}^k{w}(\bm{x}_T)|}_{I_{2,T}}\bigg)\|d^kv_h\|_{0,T}\\
		&+\sum_{T\in\mathcal{T}_h}\underbrace{\|\gamma (\bm{w}-\Pi_{{\bm{\theta}},T}^k{w})\|_{0,T}}_{I_{3,T}}\|{v}_h\|_{0,T}.
	\end{aligned}
	$$
	Thanks to Lemma \ref{lemma:approxinterpolate} (approximation property), we have
	\begin{align}
		|I_{1,T}|
		=&~ |(I-\Pi_{{\bm{\theta}(\bm{x}_T)},T}^{k+1})J_{\bm{\theta}(\bm{x}_T)}^kw(\bm{x}_T)| \notag \\
		\lesssim & ~ \|(I-\Pi_{{\bm{\theta}(\bm{x}_T)},T}^{k+1})J_{\bm{\theta}(\bm{x}_T)}^kw\|_{0,\infty,T}\\
		\lesssim & ~\hat{C}_{b}C(p) h_T^{1-\frac{3}{p}}(|J_{{\bm{\theta}}}^k{w}|_{1,p,T} + |\bm{\beta}|_{1,\infty,T}\|w\|_{0,p,T} + h_T|\bm{\beta}|_{1,\infty,T}|w|_{1,p,T}), \notag \\
		|I_{3,T}| \lesssim & \hat{C}_{b} \|\gamma\|_{0,\infty,T}C(p)h_T^{1+3(\frac{1}{2}-\frac{1}{p})}|{w}|_{1,p,T}.
	\end{align}
	Using the result in Lemma \ref{lemma:diffinter} (flux difference), we have
	\begin{equation}
		\begin{aligned}
			|I_{2,T}|
			&= ~|J_{\bm{\theta}(\bm{x}_T),T}^k\Pi_{{\bm{\theta}(\bm{x}_T)},T}^k(I-\Pi_{{\bm{\theta}},T}^k){w}(\bm{x}_T)|\\
			&\lesssim ~\|J^k_{\bm{\theta}(\bm{x}_T)}\Pi_{{\bm{\theta}(\bm{x}_T)},T}^k(I-\Pi_{{\bm{\theta}},T}^k){w}\|_{0,\infty,T}\\
			& \lesssim ~\hat{C}_{b} C(p) |\bm{\theta}|_{1,\infty,T} h_T^{1-\frac{3}{p}}|w|_{1,p,T}.
		\end{aligned}
	\end{equation}
	With the above inequalities, we obtain the desired results.
\end{proof}

In the proof of Lemma \ref{lemma:diff}, we have the following consistency error of flux.
\begin{corollary}[consistency error of flux]\label{coro:fluxdiff}
For any $T\in\mathcal{T}_h$, assume that $J_{{\bm{\theta}}}^k{w}\in {W}^{1,p}(T)$, ${w}\in W^{1,p}(T)$ where $p>3$ and $h_T\lesssim\|\bm{\theta}\|_{1,\infty,T}^{-1}$. Then the following inequality holds
\begin{equation}
		\|J_{\bm{\theta}}^kw-J_{\bm{\theta},h}^k\Pi_{{\bm{\theta}},h}^kw\|_{0,\Omega}\lesssim \Theta_{J}(\varepsilon,\bm{\theta},\gamma,w)h,
\end{equation}
where
\begin{equation}
	\begin{aligned}
	\Theta_{J}(\varepsilon,\bm{\theta},\gamma,w)=&
~\hat{C}_{b}C(p)  \bigg\{ \sum_{T \in \mathcal{T}_h}
\Big( h_T^{3(\frac{1}{2}-\frac{1}{p})}(|J_{{\bm{\theta}}}^k{w}|_{1,p,T} + |\bm{\beta}|_{1,\infty,T}\|w\|_{0,p,T}\\
&~+ h_T|\bm{\beta}|_{1,\infty,T}|w|_{1,p,T}+|\bm{\theta}|_{1,\infty,T}|w|_{1,p,T})\Big)^2 \bigg\}^{\frac12}.
			\end{aligned}
\end{equation}
\end{corollary}

\begin{theorem}[discrete inf-sup]
	Under Assumption \ref{ass:infsupa}, for sufficiently small $h$, the following inf-sup condition hold:
	\begin{equation}\label{eq:infsupah}
		\inf_{w_h\in S_h}\sup_{v_h\in V_h}\frac{a_h(w_h,v_h)}{\|w_h\|_{\mathcal{S}^k,h}\|v_h\|_{H(d^k),\Omega}}=c_1>0,
	\end{equation}
	where $c_1$ is independent of $h$.
\end{theorem}

\begin{proof}
In accordance with Assumption \ref{ass:infsupa} (well-posedness), the bilinear form $a(w_h, v_h)$ fulfills the discrete inf-sup condition on $V_h$ for sufficiently small $h$:
	\begin{equation}
		\sup_{v_h\in V_h} \frac{a(\tilde{w}_h,v_h)}{\|v_h\|_{H(d^k),\Omega}}\ge \frac{c_0}{2}\|\tilde{w}_h\|_{H(d^k),\Omega} \qquad \forall \tilde{w}_h\in V_h.
	\end{equation}
	This result can be derived through an extension of the induction presented in \cite{schatz1974observation} and \cite{xu1996two}. Given $w_h \in S_h$, take $\tilde{w}_h = \tilde{\Pi}_h^k w_h \in V_h$, which satisfies the relation $w_h = \Pi_{{\bm{\theta}},h}^k \tilde{w}_h$. By virtue of Lemma \ref{lemma:diff} (consistency error), we then have:
	$$
	\begin{aligned}
		|a(\tilde{w}_h,v_h)-a_h(w_h,v_h) | &= |a(\tilde{w}_h,v_h)-a_h(\Pi_{{\bm{\theta}},h}^k\tilde{w}_h,v_h) |\\
		& \lesssim \Theta(\varepsilon,\bm{\theta},\gamma,\tilde{w}_h)\|v_h\|_{H(d^k),\Omega}.
	\end{aligned}
	$$
	Observe that $|d^k\tilde{w}_h|_{1,p,T}=0$ for any $\tilde{w}_h\in V_h$ and $T\in\mathcal{T}_h$. By inverse equality, we have
	$$
	\begin{aligned}
	h_T^{3(\frac{1}{2}-\frac{1}{p})}|J_{\bm{\theta}}^k\tilde{w}_h|_{1,p,T} 
	& \lesssim h_T^{3(\frac{1}{2}-\frac{1}{p})} \big( \|\bm{\beta}\|_{0,\infty,T}|\tilde{w}_h|_{1,p,T} + |\bm{\beta}|_{1,\infty,T} \|\tilde{w}_h\|_{0,p,T} \big) \\
	& \lesssim \|\bm{\beta}\|_{1,\infty,T}\|\tilde{w}_h\|_{H(d^k),T}.
	\end{aligned}
	$$
	The rest terms in $\Theta(\alpha,\bm{\theta},\gamma,\tilde{w}_h)$ can be estimated by the inverse inequality and we have
	\begin{equation}
		|a(\tilde{w}_h,v_h)-a_h(w_h,v_h) |\lesssim \tilde{\Theta}(\varepsilon,\bm{\beta},\gamma)h\|\tilde{w}_h\|_{H(d^k),\Omega}\|v_h\|_{H(d^k),\Omega},
	\end{equation}
	where
	\begin{equation}
		\begin{aligned}
		\tilde{\Theta}(\varepsilon,\bm{\beta},\gamma):=\hat{C}_{b}C(p)\max_{T\in\mathcal{T}_h}\left\{ 
		\|\bm{\beta}\|_{1,\infty,T}+\|\gamma\|_{0,\infty,T}+|\bm{\theta}|_{1,\infty,T}\right\}.
		\end{aligned}
	\end{equation}
	Using \eqref{eq:bound1} in Lemma \ref{lemma:normequiv} (relationship between energy norms) and the boundness of $\varepsilon$, we have
	\begin{equation}
		\|w_h\|_{\mathcal{S}^k,h} \lesssim \hat{C}_{b}(1+\|\bm{\beta}\|_{0,\infty,\Omega}) \|\tilde{w}_h\|_{H(d^k),\Omega}.
	\end{equation}
	The desired result then follows when
	$
	h\lesssim h_0:=c_0(\tilde{\Theta}(\varepsilon,\bm{\beta},\gamma))^{-1}
	$:
	$$
	\begin{aligned}
		\sup_{{v}_h\in V_h}\frac{a_h(w_h,v_h)}{\|v_h\|_{H(d^k),\Omega}}&\ge \sup_{{v}_h\in V_h}\frac{a(\tilde{w}_h,v_h)}{\|v_h\|_{H(d^k),\Omega}} -Ch\tilde{\Theta}(\varepsilon,\bm{\beta},\gamma)\|\tilde{w}_h\|_{H(d^k),\Omega}\\
		& \ge \frac{c_0}{4}\|\tilde{w}_h\|_{H(d^k),\Omega} \gtrsim \frac{c_0}{\hat{C}_{b}(1+\|\bm{\beta}\|_{0,\infty,\Omega})}\|w_h\|_{\mathcal{S}^k,h}.	
	\end{aligned}
	$$
This completes the proof.
\end{proof}

\begin{theorem}
	Let ${u}$ be the solution of \eqref{eq:generalcd} and $u_h$ be the solution of \eqref{eq:discretevration}. Assume that for all $T\in\mathcal{T}_h$, $h_T \lesssim \|\bm{\theta}\|_{1,\infty,T}^{-1}$, $u\in W^{1,p}(T)$ and $J_{\bm{\theta}}^ku\in W^{1,p}(T)$, $p>3$. Then, the following estimate holds for sufficiently small $h$:
	\begin{equation}\label{eq:uhPiu}
		\|u_h-\Pi_{\bm{\theta},h}^ku\|_{\mathcal{S}^k,h} \lesssim \frac{1}{c_1}\Theta(\varepsilon,\bm{\beta},\gamma,u)h,
	\end{equation}
	where the constant $\Theta(\varepsilon,\bm{\beta},\gamma,u)$ is given in \eqref{eq:theta}.
\end{theorem}

\begin{proof}
	By Lemma \ref{lemma:diff} (consistency error),
	$$
	\begin{aligned}
		| a_h(u_h-\Pi_{{\bm{\theta}},h}^ku,v_h)| &= |(f,v_h)-a_h(\Pi_{{\bm{\theta}},h}^ku,v_h)| 
		= | a(u,v_h)-a_h(\Pi_{{\bm{\theta}},h}^ku,v_h)| \\
		&\lesssim \Theta(\varepsilon,\bm{\beta},\gamma,u)h\|v_h\|_{H(d^k),\Omega}.
	\end{aligned}
	$$
	By the discrete inf-sup condition \eqref{eq:infsupah},
	$$
	\|u_h-\Pi_{\bm{\theta},h}^ku\|_{\mathcal{S}^k,h} \lesssim \frac{1}{c_1}\Theta(\varepsilon,\bm{\beta},\gamma,u)h.
	$$
	This completes the proof.
\end{proof}

\begin{theorem}[error estimate]
	Let ${u}$ be the solution of \eqref{eq:generalcd} and $u_h$ be the solution of \eqref{eq:discretevration}. Assume that for all $T\in\mathcal{T}_h$, $h_T\lesssim \|\bm{\theta}\|_{1,\infty,T}^{-1}$, $u\in W^{1,p}(T)$ and $J_{\bm{\theta}}^ku\in W^{1,p}(T)$ with $p>3$. Then, the following estimate holds for sufficiently small $h$:
	\begin{equation}
		(\|J_{\bm{\theta}}^ku-J_{\bm{\theta},h}^ku_h\|_{0,\Omega}^2+\|u-u_h\|_{0,\Omega}^2)^{1/2} \lesssim (1+{1\over c_1}) \Theta(\varepsilon,\bm{\beta},\max\{\gamma,1\},u)h,
	\end{equation}
	where the constant $\Theta(\varepsilon,\bm{\beta},\gamma,u)$ is given in \eqref{eq:theta}.
\end{theorem}
\begin{proof}
	The approximation property of $\Pi_{{\bm{\theta}},T}^k$ in Lemma \ref{lemma:approxinterpolate} shows that 	
	$$
	\|u-\Pi_{{\bm{\theta}},T}^ku\|_{0,T} \lesssim \hat{C}_{b}C(p) h_T^{1+3(\frac{1}{2}-\frac{1}{p})}|u|_{1,p,T}.	
	$$
	Summing up the aforementioned inequality across $T \in \mathcal{T}_h$, and amalgamating this result with Corollary \ref{coro:fluxdiff} (consistency error of flux), yields:
	\begin{equation}\label{eq:interpenergy}
		(\|J_{\bm{\theta}}^ku-J_{\bm{\theta},h}^k\Pi_{{\bm{\theta}},h}^ku\|_{0,\Omega}^2+\|u-\Pi_{{\bm{\theta}},h}^ku\|_{0,\Omega}^2)^{1/2} \lesssim \Theta(\varepsilon,\bm{\beta},1,u)h.
	\end{equation}
	Then the desired result is obtained through triangle inequality on \eqref{eq:uhPiu} and \eqref{eq:interpenergy}.
\end{proof}

\section{Numerical tests}\label{sec:numerical}
In this section, we present a series of numerical tests for $\bm{H}(\mathrm{curl})$ and $\bm{H}(\mathrm{div})$ convection-diffusion problems conducted in both two dimensions (2D) and three dimensions (3D). These tests are designed to showcase the convergence and accuracy of the proposed exponentially-fitted finite element spaces, as well as their performance in handling convection-dominated problems featuring solutions with distinct layers. 

Our experimental investigations are conducted on the unit square $\Omega=(0,1)^2$ in 2D and on the unit cube $\Omega=(0,1)^3$ in 3D. We employ uniform meshes with varying mesh sizes across all the tests. The results of these tests offer insights into the efficacy and robustness of the proposed approach, shedding light on its capabilities in accurately capturing intricate solution behaviors, particularly in scenarios characterized by strong convection effects and layered solution structures.

{\it Evaluation of $(J_{{\theta},T}^ku_h,d^kv_h)_T$.}
In the numerical implementation of $(J_{\bm{\theta},T}^ku_h,d^kv_h)_T$ for $w_h\in \mathcal{S}_{1^-}^k(T)$ and $v_h\in \mathcal{P}_{1^-}^k(T)$, it is noteworthy that the variation of $J_{\bm{\theta},T}^ku_h$ is relatively mild and has the property of constant and kernel preservation for a constant $\bm{\theta}$. This characteristic allows us to approximate $(J_{\bm{\theta},T}^ku_h,d^kv_h)_T$ as follows:
$$
\begin{aligned}
	(J_{\bm{\theta},T}^ku_h,d^kv_h)_T\approx&\sum_{T\in\mathcal{T}_h} {J}_{\bm{\theta},T}^ku_h(\bm{b}_T) \cdot \int_T d^kv_h\\
	=&\sum_{T\in\mathcal{T}_h}{J}_{\bm{\theta},T}^ku_h(\bm{b}_T) \cdot \int_{\partial T} {\rm tr}({v}_h),
\end{aligned}
$$
where $\bm{b}_T$ denotes the barycenter of the element $T$. 

\subsection{$H({\rm div})$ convection-diffusion in 2D} Similar to the 3D case, we can establish the local $\bm{H}({\rm div})$ exponentially-fitted finite element space in  2D and adopt the labeling shown in Figure \ref{fig:divbasis2D}, which yields the following problem:
\begin{figure}[!htbp]
	\centering
	\includegraphics[width=.3\textwidth]{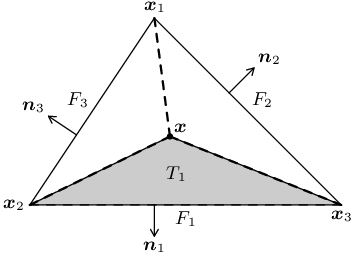}
	\caption{Geometric notation for $2D$ $H({\rm div})$ convection diffusion problem.}\label{fig:divbasis2D}
\end{figure}

\begin{problem}[2D $\mathcal{L}$-spline for $\bm{H}({\rm div})$ convection-diffusion]\label{problem:div2D}
	Find $\bm{\phi}_1^F(\bm{x})$ and $j_1^F(\bm{x})$ such that for all $x\in \bar{T}$
	\begin{equation}
		\label{eq:sysdiv}
		D^F(\bm{x})\left(\begin{matrix}
			j_1^F(\bm{x})\\\bm{\phi}_1^F(\bm{x})
		\end{matrix}\right)= B_2^\varepsilon(\sigma_{23},-\sigma_{2})\bm{e}_1^F.
	\end{equation}
	Here, $\bm{e}^F_1=(1,0,0)^\top$ and $D^F(\bm{x})$ is a $3\times 3$ matrix defined by
	\begin{equation}
		\begin{aligned}
			&	D^F(\bm{x}) =\left(\begin{matrix}
				\frac{\bm{l}_2\times\bm{l}_3}{2} & B_2^\varepsilon(\sigma_2,\sigma_3)\bm{l}^\perp_2-B_2^\varepsilon(\sigma_3,\sigma_2)\bm{l}_3^\perp \\
				\frac{\bm{l}_1\times\bm{l}_3}{2} & B_2^\varepsilon(\sigma_3,\sigma_1)\bm{l}^\perp_3-B_2^\varepsilon(\sigma_1,\sigma_3)\bm{l}_1^\perp \\
				\frac{\bm{l}_1\times\bm{l}_2}{2} & B_2^\varepsilon(\sigma_1,\sigma_2)\bm{l}^\perp_1-B_2^\varepsilon(\sigma_2,\sigma_1)\bm{l}_2^\perp \\
			\end{matrix}\right),
		\end{aligned}
	\end{equation}
where the 2D cross product is given by $\bm{x}\times\bm{y} :=x_1y_2-x_2y_1$ and 2D rotation is $(x_1,x_2)^\perp :=(x_2,-x_1)$. 
\end{problem}
It is worth noting that the properties of the 2D $\bm{H}({\rm div})$ exponentially-fitted finite element space remain consistent with those observed in the 3D case. In the context of the current experiment, we apply the homogeneous boundary condition $\bm{u}\cdot\bm{n}|_{\partial \Omega}=0$, wherein the convection speed is defined as $\bm{\beta}=(-x_2,x_1)$, and the reaction coefficient is given by $\gamma=1$.

{\it Convergence order test}. $\bm{f}$ is analytically selected so that the exact solution of \eqref{eq:generalcd} is
$$
\bm{u} = \left(x_1x_2(1-x_1)(1-x_2),\sin(\pi x_1)\sin(\pi x_2)\right)^T.
$$
Results presented in Tables \ref{tab:div2dL2} and \ref{tab:div2dJ} demonstrate a consistent trend of first-order convergence in the $L^2$ error of both the solution and the flux. This convergence behavior holds across various diffusion coefficients, spanning the range from $1$ to $10^{-6}$. It is intriguing to observe that, in cases devoid of boundary or internal layers, the $L^2$ error of both the solution and the flux appear to exhibit stability with respect to the diffusion coefficient $\varepsilon$.

\begin{table}[!htbp]
	\centering
	\caption{2D $\bm{H}({\rm div})$ problem: $L^2$ error convergence test of solution.}\label{tab:div2dL2}
	\begin{tabular}{c|c c|c c|c c}
		\hline
		\multirow{2}{*}{$1/h$} & \multicolumn{2}{c|}{$\varepsilon=1$} & \multicolumn{2}{c|}{$\varepsilon=0.01$} & \multicolumn{2}{c}{$\varepsilon=10^{-6}$}\\
		& $\|\bm{u}-\bm{u}_h\|_{0,\Omega}$ & order & $\|\bm{u}-\bm{u}_h\|_{0,\Omega}$ & order & $\|\bm{u}-\bm{u}_h\|_{0,\Omega}$ & order\\
		\hline
	4& 1.51e-01 & --& 1.61e-01 & --& 1.74e-01 & --\\ 
	8& 7.71e-02 & 0.97& 7.92e-02 & 1.03& 8.84e-02 & 0.98\\ 
	16& 3.88e-02 & 0.99& 3.92e-02 & 1.02& 4.46e-02 & 0.99\\ 
	32& 1.94e-02 & 1.00& 1.95e-02 & 1.01& 2.24e-02 & 0.99\\ 
	64& 9.70e-03 & 1.00& 9.72e-03 & 1.00& 1.13e-02 & 0.99\\ 
	128& 4.85e-03 & 1.00& 4.85e-03 & 1.00& 5.65e-03 & 1.00\\ 
		\hline
	\end{tabular}
\end{table}

\begin{table}[!htbp]
	\centering
	\caption{2D $\bm{H}({\rm div})$ problem: $L^2$ error convergence test of flux.}\label{tab:div2dJ}
	\begin{tabular}{c|c c|c c|c c}
		\hline
		\multirow{2}{*}{$1/h$} & \multicolumn{2}{c|}{$\varepsilon=1$} & \multicolumn{2}{c|}{$\varepsilon=0.01$} & \multicolumn{2}{c}{$\varepsilon=10^{-6}$}\\
		& $\|J_{\bm{\theta}}^1\bm{u}-J_{\bm{\theta},h}^1\bm{u}_h\|_{0,\Omega}$ & order &   $\|J_{\bm{\theta}}^1\bm{u}-J_{\bm{\theta},h}^1\bm{u}_h\|_{0,\Omega}$ & order &   $\|J_{\bm{\theta}}^1\bm{u}-J_{\bm{\theta},h}^1\bm{u}_h\|_{0,\Omega}$ & order\\
		\hline
		4& 4.25e-01 & --& 7.08e-02 & --& 7.22e-02 & --\\ 
		8& 2.15e-01 & 0.98& 3.59e-02 & 0.98& 3.66e-02 & 0.98\\ 
		16& 1.08e-01 & 1.00& 1.80e-02 & 1.00& 1.84e-02 & 0.99\\ 
		32& 5.40e-02 & 1.00& 8.98e-03 & 1.00& 9.22e-03 & 1.00\\ 
		64& 2.70e-02 & 1.00& 4.49e-03 & 1.00& 4.61e-03 & 1.00\\ 
		128& 1.35e-02 & 1.00& 2.25e-03 & 1.00& 2.31e-03 & 1.00\\ 
		\hline
	\end{tabular}
\end{table}
{\it Numerical stability}. By selecting $\bm{f}=(1,1)^\top$ and utilizing a mesh size of $h=1/16$, we calculate the numerical solution for different diffusion coefficients, specifically $\varepsilon=0.01$ and $\varepsilon=10^{-6}$. As illustrated in Figure \ref{fig:div2Dstable}, the obtained numerical solution showcases a remarkable stability, devoid of oscillations near the boundary, across all values of $\varepsilon$. This observation serves to affirm the inherent stabilizing effect of the exponentially-fitted finite element method in the context of convection-dominated cases.
\begin{figure}[!htbp]
	\centering
	\subfloat[$\varepsilon=0.01$]{
		\includegraphics[width=.45\textwidth]{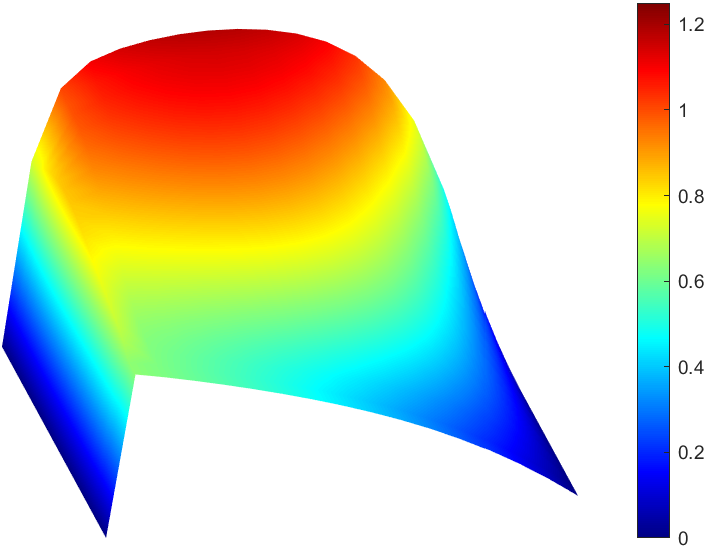}
	}
	\subfloat[$\varepsilon=10^{-6}$]{
		\includegraphics[width=.45\textwidth]{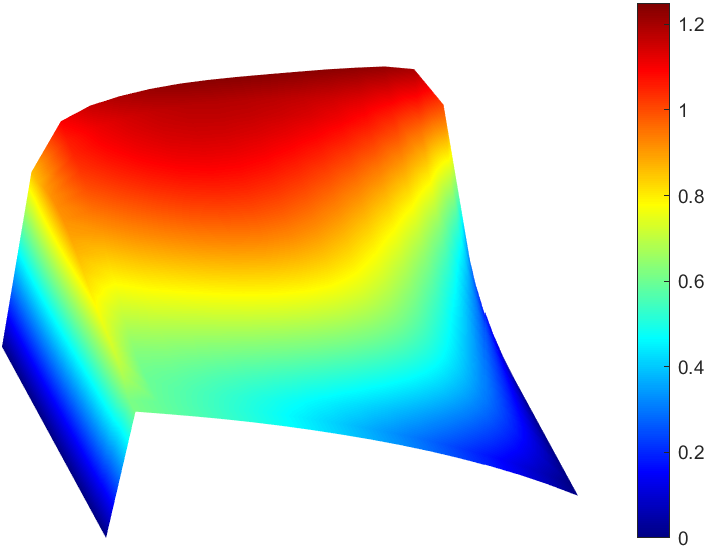} 
	}
	\caption{Plots of the first component of the numerical solution for 2D $\bm{H}({\rm div})$ convection-diffusion problems.}
	\label{fig:div2Dstable}
\end{figure}
\subsection{$H({\rm curl})$ convection-diffusion in 3D} Consider the exact solution given by:
$$
\bm{u} = \big(\sin x_3,\sin x_1,\sin x_2\big)^\top,
$$
where the convection speed is defined as $\bm{\beta}=(x_2,x_3,x_1)^\top$, and the reaction coefficient is $\gamma=1$. Both the Dirichlet boundary condition and the source term $\bm{f}$ can be obtained analytically.

As demonstrated in Table \ref{tab:curl3dL2} and Table \ref{tab:curl3dJ}, we observe first-order convergence in both the $L^2$ error of the solution and the flux, across a range of diffusion coefficients spanning from $1$ to $10^{-6}$. Moreover, our observations continue to indicate that the $L^2$ error of both the solution and the flux remain stable with regard to changes in the diffusion coefficient $\varepsilon$, in scenarios where the solution lacks boundary or internal layers.

\begin{table}[!htbp]
	\centering
	\caption{3D $\bm{H}({\rm curl})$ problem: $L^2$ error convergence test of solution.}\label{tab:curl3dL2}
	\begin{tabular}{c|c c|c c|c c}
		\hline
		\multirow{2}{*}{$1/h$} & \multicolumn{2}{c|}{$\varepsilon=1$} & \multicolumn{2}{c|}{$\varepsilon=0.01$} & \multicolumn{2}{c}{$\varepsilon=10^{-6}$}\\
		& $\|\bm{u}-\bm{u}_h\|_{0,\Omega}$ & order & $\|\bm{u}-\bm{u}_h\|_{0,\Omega}$ & order & $\|\bm{u}-\bm{u}_h\|_{0,\Omega}$ & order\\
		\hline
		2& 0.258763 & --& 0.247988 & --& 0.252041 & --\\ 
		4& 0.129828 & 1.00& 0.118162 & 1.07& 0.120991 & 1.06\\ 
		8& 0.064972 & 1.00& 0.057758 & 1.03& 0.058927 & 1.04\\ 
		16& 0.032494 & 1.00& 0.029555 & 0.97& 0.029035 & 1.02\\ 
		\hline
	\end{tabular}
\end{table}

\begin{table}[!htbp]
	\centering
	\caption{3D $\bm{H}({\rm curl})$ problem: $L^2$ error convergence test of flux.}\label{tab:curl3dJ}
	\begin{tabular}{c|c c|c c|c c}
		\hline
		\multirow{2}{*}{$1/h$} & \multicolumn{2}{c|}{$\varepsilon=1$} & \multicolumn{2}{c|}{$\varepsilon=0.01$} & \multicolumn{2}{c}{$\varepsilon=10^{-6}$}\\
		& $\|\bm{J}_{\bm{\theta}}^1\bm{u}-\bm{J}_{\bm{\theta},h}^1\bm{u}_h\|_{0,\Omega}$ & order &   $\|\bm{J}_{\bm{\theta}}^1\bm{u}-\bm{J}_{\bm{\theta},h}^1\bm{u}_h\|_{0,\Omega}$ & order &   $\|\bm{J}_{\bm{\theta}}^1\bm{u}-\bm{J}_{\bm{\theta},h}^1\bm{u}_h\|_{0,\Omega}$ & order\\
		\hline
		2& 0.135113 & --& 0.187823 & --& 0.195079 & --\\ 
		4& 0.056780 & 1.25& 0.091150 & 1.04& 0.099255 & 0.97\\ 
		8& 0.025060 & 1.18& 0.041244 & 1.14& 0.049401 & 1.01\\ 
		16& 0.011639 & 1.11& 0.017674 & 1.22& 0.024603 & 1.01\\ 
		\hline
	\end{tabular}
\end{table}

\backmatter
%
%
%
%
%
\bmhead{Acknowledgments}
The work is supported in part by the National Natural Science Foundation of China grant No. 12222101 and the Beijing Natural Science Foundation No. 1232007.

\bibliography{spline}

\end{document}